\RequirePackage{tikz}
\documentclass[sn-mathphys]{sn-jnl}

\usepackage[utf8]{inputenc}
\usepackage{dsfont}
\usepackage{transparent}
\usepackage{placeins}
\usepackage[nameinlink]{cleveref}

\renewcommand{\R}{\mathbb{R}}
\newcommand{\N}{\mathbb{N}}
\newcommand{\U}{\mathcal{U}}
\newcommand{\X}{\mathcal{X}}

\newcommand{\sX}{{_\X}}
\newcommand{\sU}{{_\U}}
\newcommand{\ddes}{{d_{\textnormal{o}}}}
\newcommand{\dran}{{d_{\textnormal{r}}}}
\renewcommand{\e}{{\varepsilon}}

\newcommand{\CIEL}{\mathbf{L}}
\newcommand{\CIEa}{\mathbf{a}}
\newcommand{\CIEb}{\mathbf{b}}

\DeclareMathOperator{\supp}{supp}

\DeclareMathOperator{\Abs}{Abs}
\DeclareMathOperator{\Sca}{Sca}
\DeclareMathOperator{\Geo}{Geo}

\jyear{2021}%

\theoremstyle{thmstyleone}%
\newtheorem{theorem}{Theorem}[section]

\Crefname{assumption}{Assumption}{Assumptions}

\newtheorem{claimnew}[theorem]{Claim}
\newtheorem{proposition}[theorem]{Proposition}%

\theoremstyle{thmstyletwo}%

\theoremstyle{thmstylethree}%

\begin{document}

\title[The Continuous Stochastic Gradient Method]{The Continuous Stochastic Gradient Method}
\subtitle{Part II - Application and Numerics}


\author[1]{\fnm{Max} \sur{Grieshammer}}\email{max.grieshammer@fau.de}

\author[1,2]{\fnm{Lukas} \sur{Pflug}}\email{lukas.pflug@fau.de}

\author[1]{\fnm{Michael} \sur{Stingl}}\email{michael.stingl@fau.de}
\author*[1]{\fnm{Andrian} \sur{Uihlein}}\email{andrian.uihlein@fau.de}

\affil[1]{\orgdiv{Department of Mathematics, Chair of Applied Mathematics}, \orgname{Friedrich-Alexander-Universität Erlangen-Nürnberg (FAU)}}
\affil[2]{\orgdiv{Competence Unit for Scientific Computing}, \orgname{Friedrich-Alexander-Universität Erlangen-Nürnberg (FAU)}}




\abstract{In this contribution, we present a numerical analysis of the \textit{continuous stochastic gradient} (CSG) method, including applications from topology optimization and convergence rates. In contrast to standard stochastic gradient optimization schemes, CSG does not discard old gradient samples from previous iterations. Instead, design dependent integration weights are calculated to form a linear combination as an approximation to the true gradient at the current design. As the approximation error vanishes in the course of the iterations, CSG represents a hybrid approach, starting off like a purely stochastic method and behaving like a full gradient scheme in the limit.

In this work, the efficiency of CSG is demonstrated for practically relevant applications from topology optimization. These settings are characterized by both, a large number of optimization variables \textit{and} an objective function, whose evaluation requires the numerical computation of multiple integrals concatenated in a nonlinear fashion. Such problems could not be solved by any existing optimization method before.

Lastly, with regards to convergence rates, first estimates are provided and confirmed with the help of numerical experiments.
}

\maketitle
\bmhead{Acknowledgments}
The research was funded by the Deutsche Forschungsgemeinschaft (DFG, German Research Foundation) - Project-ID 416229255 - CRC 1411).

\section{Introduction}
In this paper, we present a numerical analysis of the \textit{Continuous Stochastic Gradient} (CSG) method, which was first proposed in \cite{pflug_CSG}. Later, in \cite{CSGPart1}, it was shown that the error in the CSG gradient and objective function approximation vanishes during the course of the iterations. This key property of CSG yields strong convergence results known from classic gradient methods, e.g., convergence of the sequence of iterates for constant step sizes, which are beyond the scope of standard stochastic approaches known from literature, like the \textit{Stochastic Gradient} (SG) method \cite{Monro1951}, or the \textit{Stochastic Average Gradient} (SAG) method \cite{LeRoux2017}.

Furthermore, the approximation property of CSG significantly increases the set of possible applications, allowing for more complex structures in the optimization problem than the schemes listed before. While CSG was shown to perform superior to various stochastic optimization approaches on academic examples \cite{CSGPart1}, it remains to see if this is also the case for more involved applications. For this purpose, we consider several optimization problems arising in the context of optimal nanoparticle design. These applications focus on optimization with respect to the resulting color of a particulate product, as it represents one of the most prominent fields of research within this setting \cite{Color1,Color2,Color3,Color4,Color5,PAMM}.

Moreover, all convergence results stated in \cite{CSGPart1} provide no insight on the \textit{rate} of convergence. Since this plays a crucial role for the practicability of CSG, it is of great importance to further analyze this quantity. In this contribution, we propose estimated convergence rates for the general CSG method and verify the numerically.

\subsection{Structure of the Paper}
\Cref{sec:Application} introduces the application from nanoparticle optics, mentioned above. Two different methods to model the particle, varying greatly in computational effort and design dimension, are presented. After detailing the setting and challenges in the low-dimensional optimization problem, we compare the results of the CSG method to different approaches based on the \textit{fmincon} algorithm provided by MATLAB (\Cref{subsec:ApplicationsNumerics}). Later on, we analyze the high-dimensional problem formulation purely within the CSG framework, since a comparison with generic deterministic optimization schemes is out of scope, due to the associated computational complexity.

Afterwards, \Cref{sec:ErrorEstimates} shortly covers techniques to estimate the gradient approximation error during the optimization, before we focus on the convergence rate of CSG in \Cref{sec:ConvergenceRates}. While the expected rates stated therein are \textit{not} proven, we present detailed numerical examples to solidify our claims. Furthermore, we analyze how the convergence rate depends on the dimension of integration and how to avoid slow convergence, if the objective function admits additional structure.

\section{Nanoparticle Design Optimization}\label{sec:Application}
 Since the design of a nanoparticle, i.e., its shape, size, material distribution, etc., heavily impacts its optical properties, the task of optimizing a nanoparticle design with respect to a specific optical property arises naturally \cite{taylor2011painting}. In this section, we are interested in using hematite nanoparticles to optimize the color of a paint film \cite{buxbaum2008industrial}. Thus, we start by introducing our main framework for this application.
\subsection{Color Spaces}
First off, we should explain what \textit{optimal color} means in our setting. There are several different methods to describe color mathematically, e.g., assigning each color an RGB representation vector $\mathbf{v}\in\R^3$, where the three components of $\mathbf{v}$ correspond to the red, green and blue value of the color. In our application, we are interested in the color of the paint film as it appears to the human eye. Therefore, the underlying color space should be chosen based on the following property: 
\begin{quote}
\textit{If the euclidean distance between the representation vectors of two colors is small, the colors should be almost indistinguishable to the human eye.}
\end{quote}
As it turns out, the RGB color space is a very poor choice with respect to this feature. Hence, we instead choose the CIELAB color space \cite{CIELAB1}, which was introduced by the International Commission of Illumination (Commission Internationale de l'Eclairage, CIE), as it was designed with this exact purpose in mind. The CIELAB representation of a color consists of three values $\CIEL$, $\CIEa$ and $\CIEb$. Here, $\CIEL$ corresponds to the lightness of a color and ranges from 0 (black) to 100 (white). The values of $\CIEa$ and $\CIEb$, typically within the range of $\pm 150$, describe the colors position with respect to the opponent color pairs green-red and blue-yellow. A short overview is given in \Cref{fig:CIELABColormap}.

Another color space, which naturally arises from our setting, is the CIE 1931 XYZ color space \cite{CIECMF}. The values of X, Y and Z can be calculated by integrating the optical properties of a particle over the spectrum of visible light (400nm - 700nm), which we denote by $\Lambda$. Each of these integrations is weighted by the corresponding color matching functions $x,y,z:\Lambda\to\R$.

Thus, in our application, we will first calculate the CIE 1931 XYZ representation of the resulting color and then use the (nonlinear) color space transformation $\Psi:\R^3\to\R^3$ with $\Psi(\text{X,Y,Z}) = (\CIEL,\CIEa,\CIEb)^\top$, to work in the CIELAB color space. For this transformation, we define a reference white point
\begin{equation*}
    \begin{pmatrix} \text{X}_r \\ \text{Y}_r \\ \text{Z}_r \end{pmatrix} = \begin{pmatrix} 94.72528492 \\ 100 \\ 107.13012997\end{pmatrix}
\end{equation*}
and denote the relative XYZ values by
\begin{equation*}
    \tilde{\text{X}} = \tfrac{\text{X}}{\text{X}_r}, \quad \tilde{\text{Y}} = \tfrac{\text{Y}}{\text{Y}_r}, \quad\text{and}\quad \tilde{\text{Z}} = \tfrac{\text{Z}}{\text{Z}_r}.
\end{equation*}
Utilizing the intended CIE parameters $\epsilon = \tfrac{216}{24389}$ and $\kappa = \tfrac{24389}{27}$, the LAB color values are then given by
\begin{equation*}
    \CIEL = 116f(\tilde{\text{Y}}) -16,\quad \CIEa = 500\big(f(\tilde{\text{X}})-f(\tilde{\text{Y}})\big)\quad\text{and}\quad \CIEb = 200\big(f(\tilde{\text{Y}})-f(\tilde{\text{Z}})\big),
\end{equation*}
where $f:\R\to\R$ is defined as
\begin{equation*}
    f(t) = \begin{cases} \sqrt[3]{t} & \text{if }\tilde{\text{X}} > \epsilon \\ \tfrac{\kappa t + 16}{116} & \text{otherwise} \end{cases}.
\end{equation*}
\begin{figure}
  \begin{minipage}[c]{0.5\textwidth}
    \includegraphics[width=\textwidth]{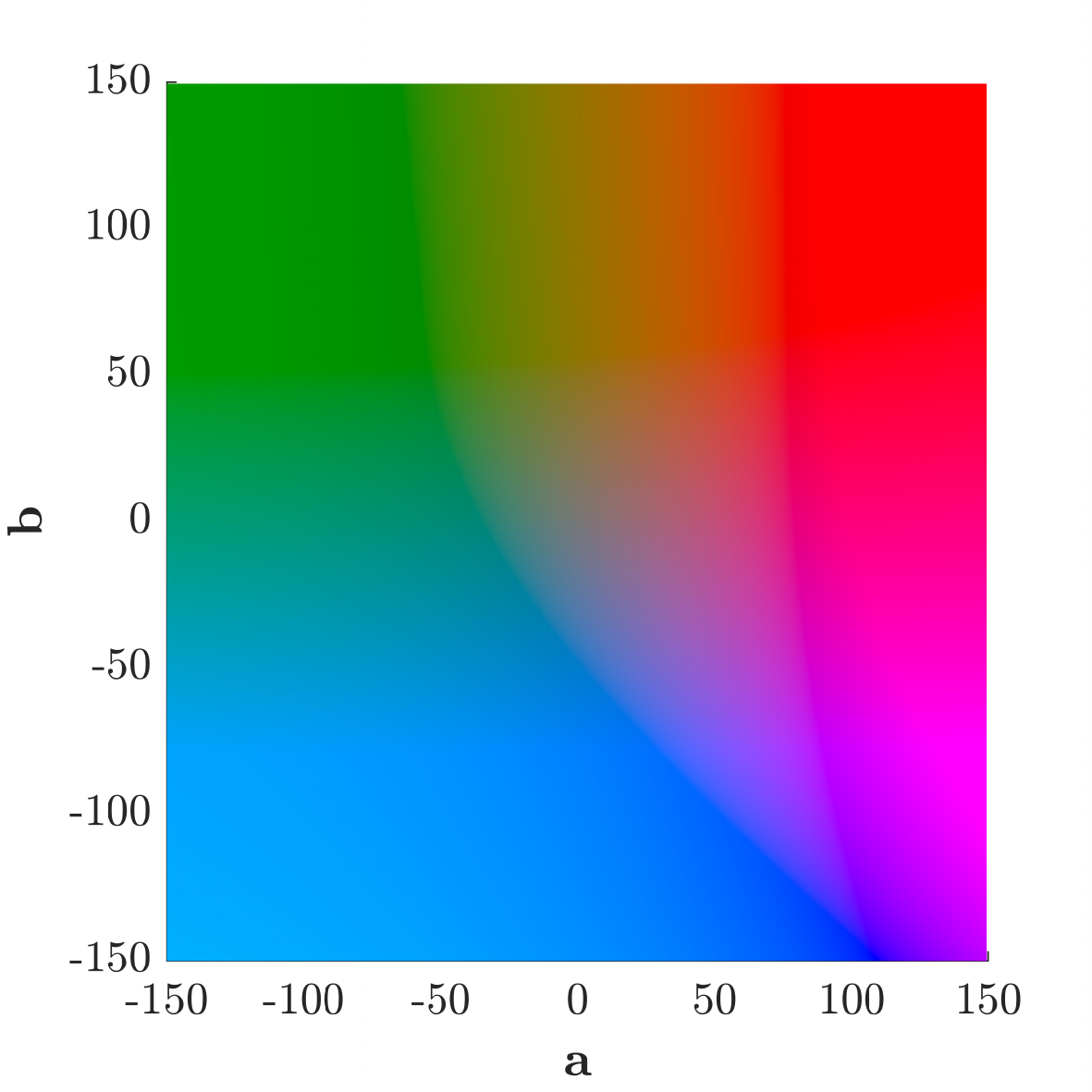}
  \end{minipage}\hfill
  \begin{minipage}[c]{0.45\textwidth}
    \caption{Resulting color for various different values of $\CIEa$ and $\CIEb$. Positive values of $\CIEa$ result in red colors, while colors corresponding to negative values of $\CIEa$ appear green. Similarly, positive $\CIEb$ values yield yellow colors, while negative $\CIEb$ values shift the color into the blue spectrum. In this figure, we fixed $\CIEL = 50$.}
    \label{fig:CIELABColormap}
  \end{minipage}
\end{figure}
\subsection{Mie Theory and Discrete Dipole Approximation}
Given a nanoparticle shape and material, we can use the time-harmonic Maxwell's equations to calculate its optical properties. Specifically, in our setting, we are interested in the absorption ($\Abs$), scattering ($\Sca$) and geometry factor ($\Geo$). The time required and precision achieved are, of course, dependent on our model of the nanoparticle and the method used to solve Maxwell's equations. For our setting, we choose two different approaches.

On the one hand, we will use the discrete dipole approximation (DDA) \cite{DDA1,DDA2,Yurkin2011}, in which the particle is discretized into an equidistant grid of dipole cells. Thus, DDA allows the analysis of arbitrary particle shapes and material distributions. The downside lies within the computational complexity of the method, which scales with the total number of dipoles and therefore grows rapidly when increasing the resolution. While the CSG method is still capable of solving the resulting optimization problem in our experiments, the tremendous computational cost associated to the DDA approach severely impede a detailed analysis of the problem. Especially, there is no computationally feasible, generic optimization scheme to compare our results with. However, we want to note that optimization in the DDA model has already been done in a slightly simpler setting, where the full integral over $\Lambda$ was replaced by summation over a small number of different wavelengths \cite{SGP}.

On the other hand, Mie theory \cite{MieOriginal,Mie} provides a numerically cheap alternative, at the price of a more restrictive setting. In Mie theory, one only considers radially symmetric particles. In this special setting, it is possible to find analytic solutions based on series expansions to the time-harmonic Maxwell's equations. Therefore, in our first approach, we will only consider core-shell particles, as the utilization of Mie theory allows for a much deeper analysis of the resulting optimization problem and comparison to deterministic optimization approaches, which rely on discretization of the integrals. 

\subsection{Nanoparticles in Paint Film -- Kubelka-Munk Theory}
As mentioned above, the XYZ color values of the paint film can be calculated by integration of the corresponding color matching functions $x,y,z$ and the important optical properties of the nanoparticle. The precise method to obtain X, Y and Z is given by the Kubelka-Munk theory \cite{kubelka1931article}, augmented by a Saunderson correction \cite{garcia2011assessment}. For a paint film, in which nanoparticles with design $u$ are present and which is illuminated by light with wavelength $\lambda\in\Lambda$, the resulting color can be expressed by the $K$ and $S$ value
\begin{equation*}
    K(u,\lambda) = \Abs(u,\lambda)\quad\text{and}\quad S(u,\lambda) = \Sca(u,\lambda)\big(1-\Geo(u,\lambda)\big)
\end{equation*}
via the reflectance
\begin{equation*}
    R_\infty (u,\lambda) = 1 + \frac{8}{3}\frac{K(u,\lambda)}{S(u,\lambda)} - \sqrt{\left(\frac{8}{3}\frac{K(u,\lambda)}{S(u,\lambda)}\right)^2 + \frac{16}{3}\frac{K(u,\lambda)}{S(u,\lambda)}}\, .
\end{equation*}
Now, X, Y and Z can be obtained by
\begin{align*}
    \text{X}(u) &= \int_\Lambda x(\lambda)\frac{(1-\rho_0-\rho_1)R_\infty(u,\lambda)+\rho_0}{1-\rho_1 R_\infty(u,\lambda)}\mathrm{d}\lambda, \\
    \text{Y}(u) &= \int_\Lambda y(\lambda)\frac{(1-\rho_0-\rho_1)R_\infty(u,\lambda)+\rho_0}{1-\rho_1 R_\infty(u,\lambda)}\mathrm{d}\lambda, \\
    \text{Z}(u) &= \int_\Lambda z(\lambda)\frac{(1-\rho_0-\rho_1)R_\infty(u,\lambda)+\rho_0}{1-\rho_1 R_\infty(u,\lambda)}\mathrm{d}\lambda,
\end{align*}
where $\rho_0$ and $\rho_1$ are material parameters. In our setting, which we introduce in the next section, we have $\rho_0 = 0.04$ and $\rho_1 = 0.6$.
\subsection{Problem Formulation}
In our first setting, we consider a radially symmetric core-shell nanoparticle (see \Cref{fig:CoreShellParticle}), where the inner core consists of water, while the outer shell is made of hematite. Thus, the design $u$ consists of the radius $R$ (1nm - 75nm) of the core and the thickness $d$ (1nm - 250nm) of the outer hematite shell. 
\begin{figure}
  \begin{minipage}[c]{0.35\textwidth}
    \includegraphics[width=\textwidth]{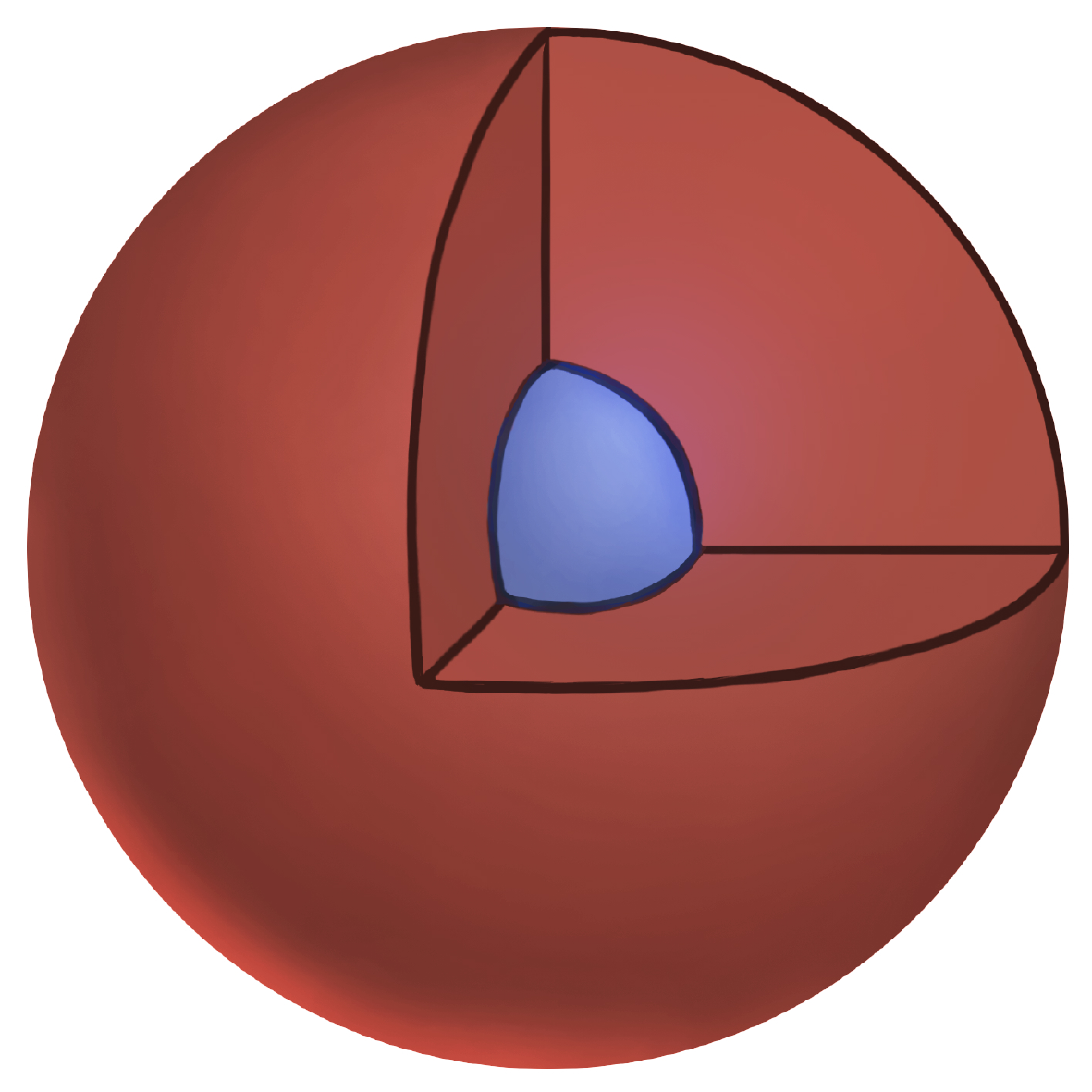}
  \end{minipage}\hfill
  \begin{minipage}[c]{0.6\textwidth}
    \caption{Radially symmetric core-shell nanoparticle. The inner core (blue) has radius $R$ in the range of 1nm - 75nm and consists of water. The thickness of the hematite shell (red) is denoted by $d$ and ranges from 1nm to 250nm.}
    \label{fig:CoreShellParticle}
  \end{minipage}
\end{figure}

As an additional layer of difficulty, we can, in practice, not expect all nanoparticles present in the paint film to be identical copies of design $u$. Instead, when trying to produce nanoparticles of a specific design in large quantities, one usually ends up with a mixture of particles of different designs, following a certain probability distribution $\mu_u$, which is dependent on the intended design $u$. 

We model this aspect by assuming that, given a design $u=(R,d)$, the particles present in the paint film follow a normal distribution (truncated to a reasonable design space $\mathcal{R}\times\mathcal{D}$) centered around $u$, i.e.,
\begin{equation*}
    \tilde{R}\sim\mathcal{N}(R,\tfrac{1}{10}R)\quad\text{and}\quad\tilde{d}\sim\mathcal{N}(d,\tfrac{1}{10}d).
\end{equation*}
Therefore, the $K$ and $S$ values in the Kubelka-Munk model need to be replaced by their averaged counterparts
\begin{align*}
    K(u,\lambda) &= \iint_{\mathcal{R}\times\mathcal{D}}\Abs(\tilde{R},\tilde{d},\lambda)\mathrm{d}\mu_u(\tilde{R},\tilde{d})\\ 
    \intertext{and}
     S(u,\lambda) &= \iint_{\mathcal{R}\times\mathcal{D}} \Sca(\tilde{R},\tilde{d},\lambda)\big(1-\Geo(\tilde{R},\tilde{d},\lambda)\big)\mathrm{d}\mu_u(\tilde{R},\tilde{d}),
\end{align*}
before calculating the reflectance $R_\infty(u,\lambda)$ and integrating it over $\Lambda$.

The objective in our application is to produce a paint of bright red color. Thus, the complete optimization problem reads
\begin{equation}
    \max_{u\in\U}\quad \tfrac{1}{20}\,\CIEL(u) + \tfrac{19}{20}\,\CIEa(u).\label{eq:ProblemSetting1}
\end{equation}
\subsection{Challenges}
The highly condensed fashion, in which \eqref{eq:ProblemSetting1} is formulated, may obscure a lot of the difficulties that arise when trying to solve it. To get a better understanding of the problem, let us first analyze the abstract structure of the objective function $J(u) = \tfrac{1}{20}\,\CIEL(u) + \tfrac{19}{20}\,\CIEa(u)$:
\begin{equation*}
    \begin{pmatrix} \Abs \\ \Sca \\ \Geo \end{pmatrix} \xrightarrow{\substack{\text{integrate} \\ \mathcal{R}\times\mathcal{D}} } \begin{pmatrix} K\\S\end{pmatrix} \xrightarrow{\substack{\text{Kubelka-} \\ \text{Munk}} } R_\infty\xrightarrow{\substack{\text{integrate}\\ \Lambda}} \begin{pmatrix} \text{X} \\ \text{Y} \\ \text{Z}\end{pmatrix} \xrightarrow{\substack{\text{color} \\ \text{transf.}\Psi} } \begin{pmatrix} \CIEL\\ \CIEa\\ \CIEb\end{pmatrix}\xrightarrow[]{}J(u).
\end{equation*}
Since calculating $J(u)$ and $\nabla J(u)$ requires integrating the optical properties in multiple dimensions and since evaluating said properties for any combination of $\tilde{R}$, $\tilde{d}$ and $\lambda$ requires solving the time-harmonic Maxwell's equations, standard deterministic approaches, e.g., full gradient methods, run into a prediscretization problem.

On the one hand, the number of integration points needs to be sufficiently large for our setting. In \Cref{fig:SliceSetting1}, a slice through the objective function for a fixed value of $R$ and several different amounts of integration points is shown. While we actually do not care too much about the approximation error resulting from a small number of integration points, the artificial local maxima introduced into the objective function by the discretization severely impact the quality of the optimization. In other words, many solutions to the discretized problem are completely unrelated to solutions to \eqref{eq:ProblemSetting1}. We want to note that, even though not all of the stationary points in \Cref{fig:SliceSetting1} correspond to stationary points of \eqref{eq:ProblemSetting1}, the prediscretization still leads to very flat regions in the objective functions, which hinder the performance of many solvers. In \Cref{fig:ApplicationFlatRegions}, this effect is displayed.

On the other hand, the number of integration points is heavily restricted by the computational cost associated to the evaluation of $\Abs$, $\Sca$ and $\Geo$. While medium resolutions ($25^3\sim15000$ points in total) are still numerically tractable for simple Mie particles, they are outright impossible to achieve in the more general DDA setting, which we want to consider later. For comparison: The optimization in \cite{SGP} was carried out using a discretization consisting of 20 points in total.

We want to emphasize that standard SG-type schemes, or even the \textit{Stochastic Composition Gradient Descent} (SCGD) method \cite{SCGDPaper}, which was used for the comparison for composite objective functions in \cite[Section 4.2]{CSGPart1}, are not capable of solving \eqref{eq:ProblemSetting1}, due to the special structure of $J$.
\begin{figure}
  \begin{minipage}[c]{0.5\textwidth}
    \includegraphics[width=\textwidth]{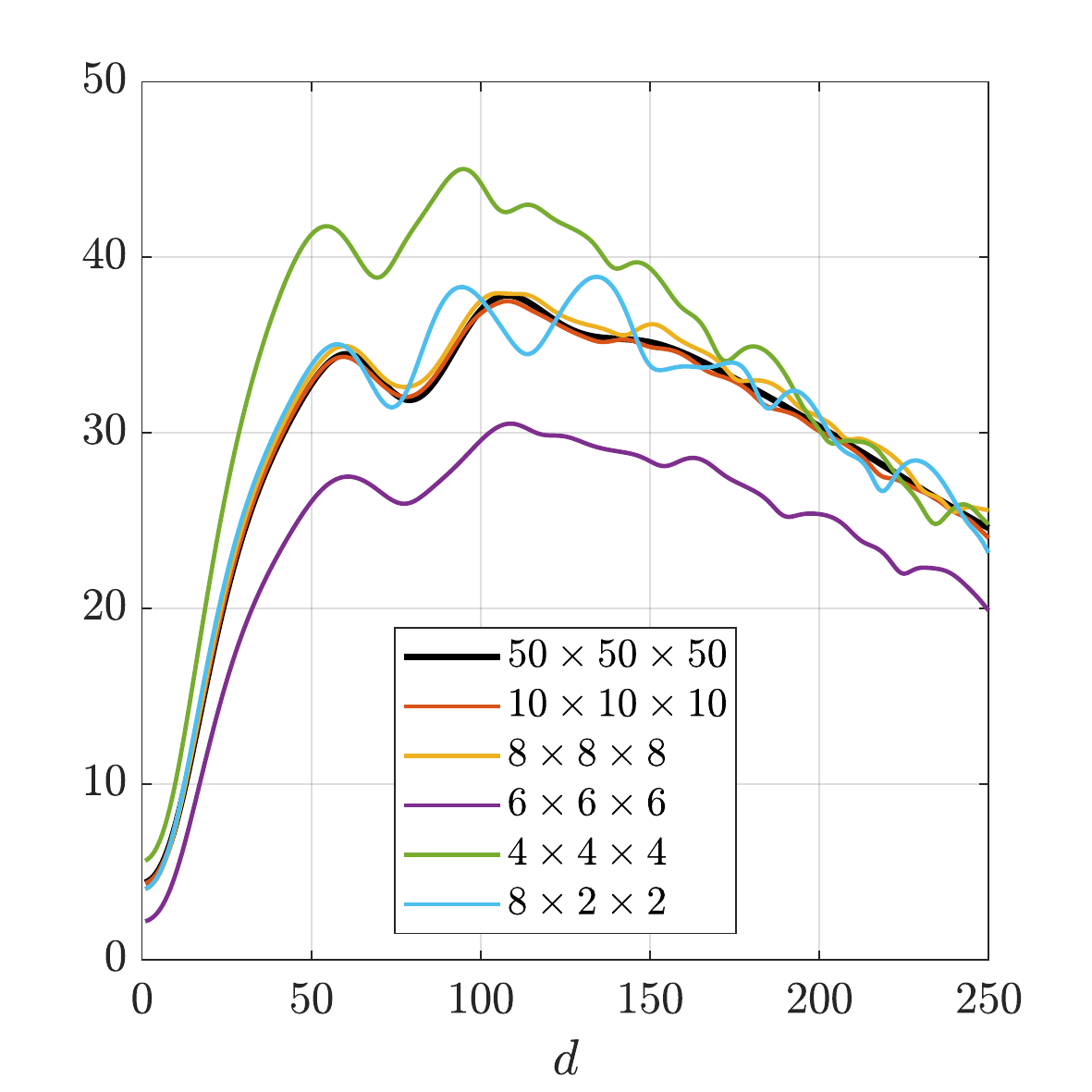}
  \end{minipage}\hfill
  \begin{minipage}[c]{0.45\textwidth}
    \caption{Objective function values for fixed core radius of 3nm. Different graphs correspond to different discretizations. The label of a curve shows into how many points the integrals over $\Lambda$, $\mathcal{R}$ and $\mathcal{D}$ have been split, respectively. Each of the discretizations introduces artificial stationary points into the objective function.}
    \label{fig:SliceSetting1}
  \end{minipage}
\end{figure}

\begin{figure}
    \centering
    \begin{minipage}[t][][c]{0.48\textwidth}
        \centering
        \includegraphics[width = \textwidth]{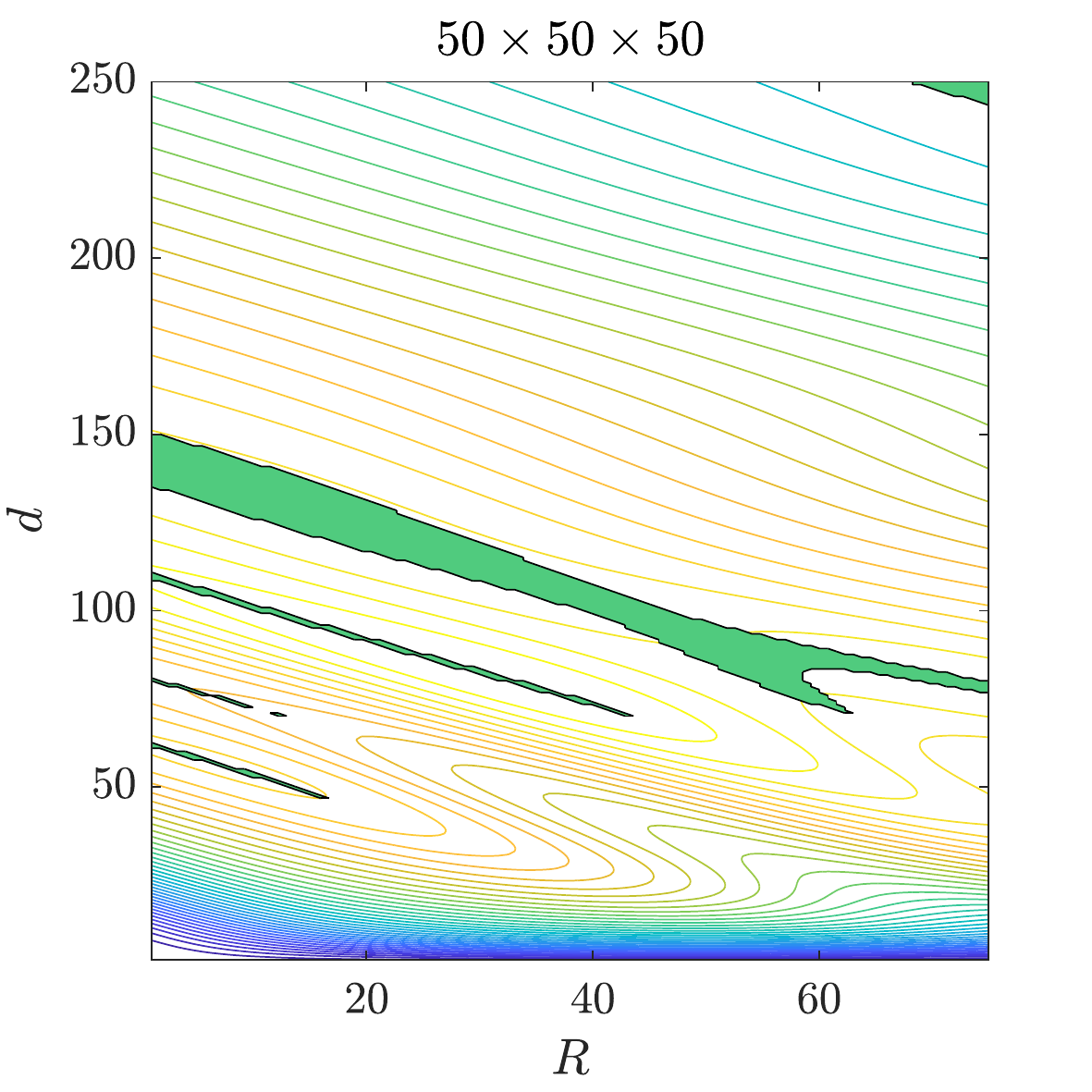}%
    \end{minipage}\hfill
    \begin{minipage}[t][][c]{0.48\textwidth}
        \centering
        \includegraphics[width = \textwidth]{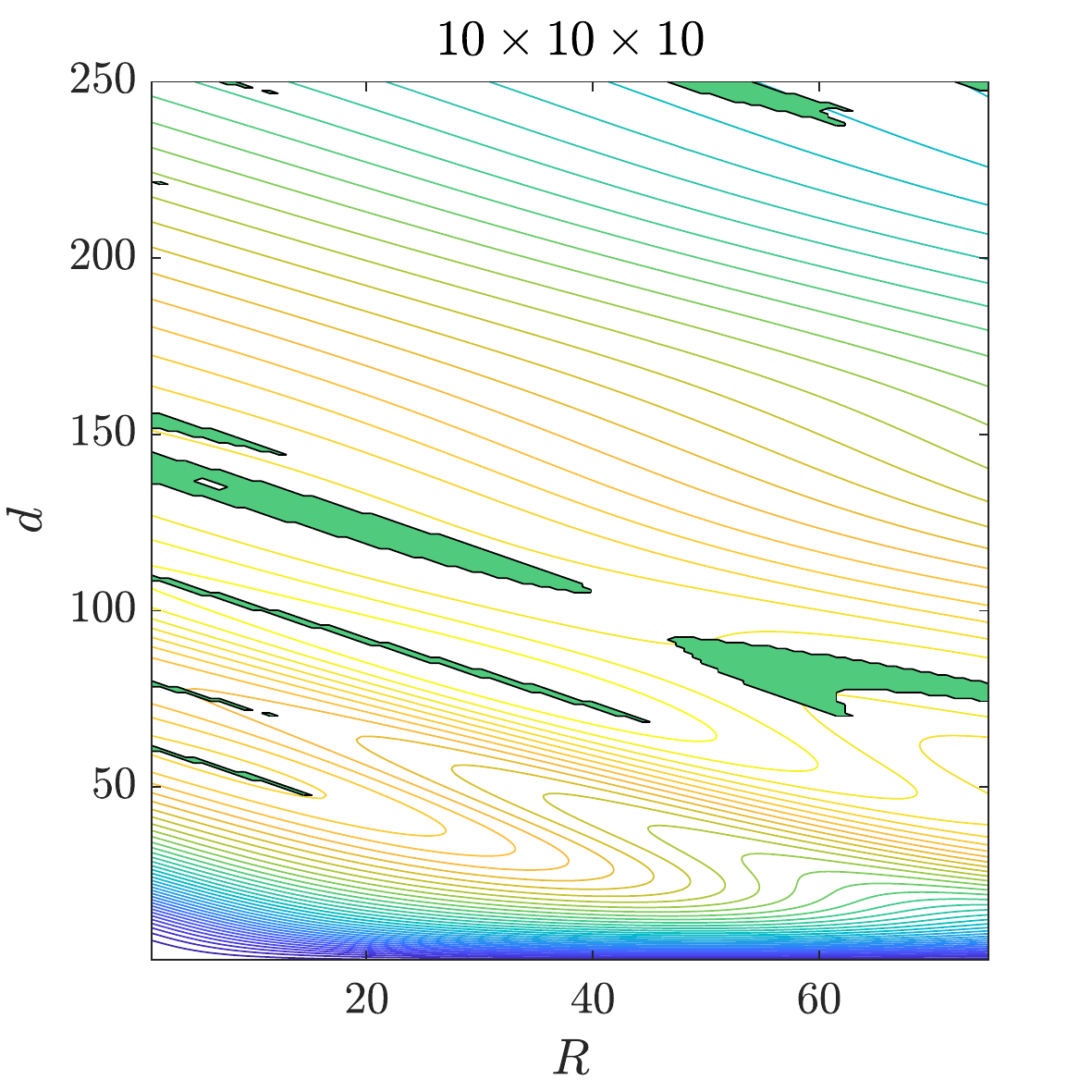}%
    \end{minipage}
    \begin{minipage}[t][][c]{0.48\textwidth}
        \centering
        \includegraphics[width = \textwidth]{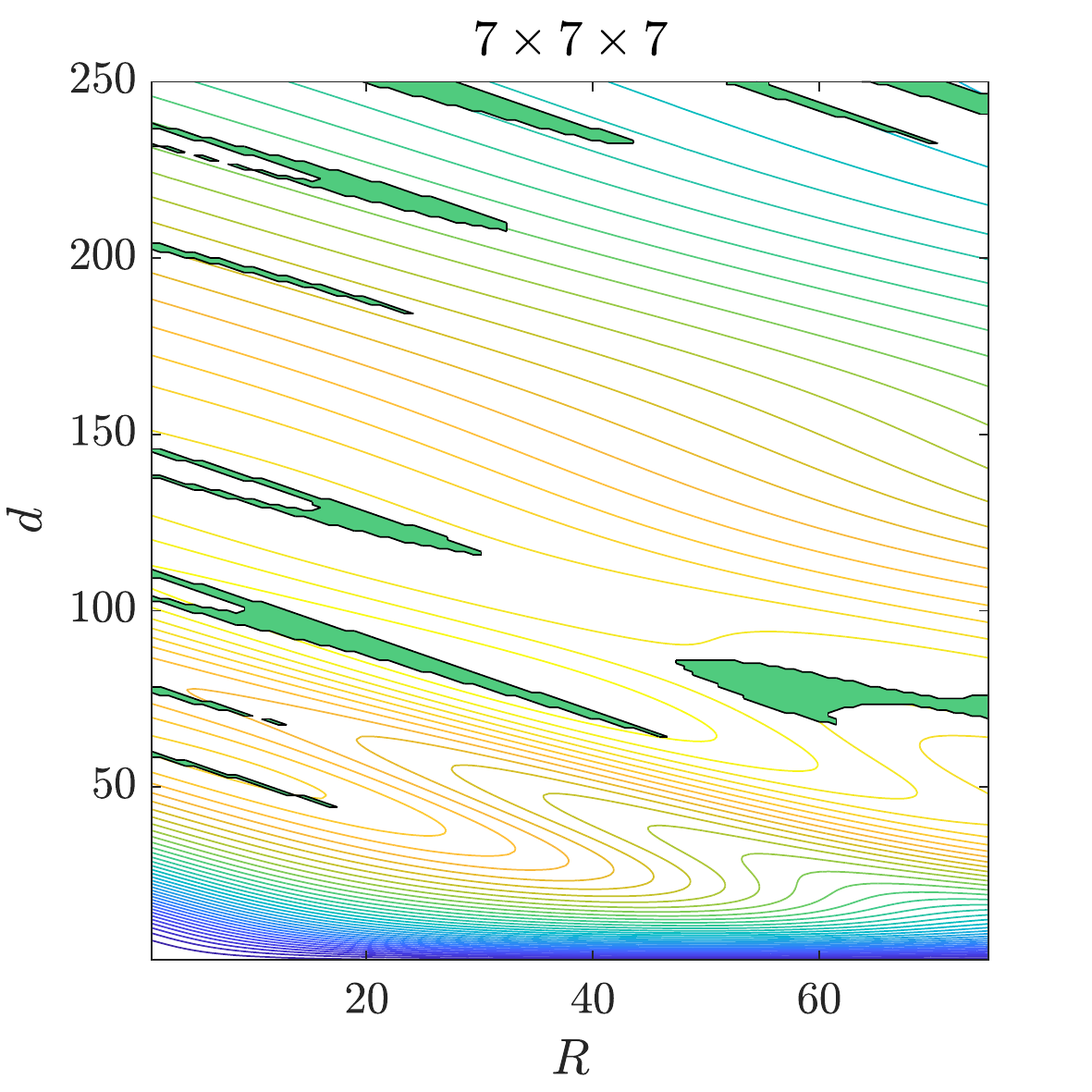}%
    \end{minipage}\hfill
    \begin{minipage}[t][][c]{0.48\textwidth}
        \centering
        \includegraphics[width = \textwidth]{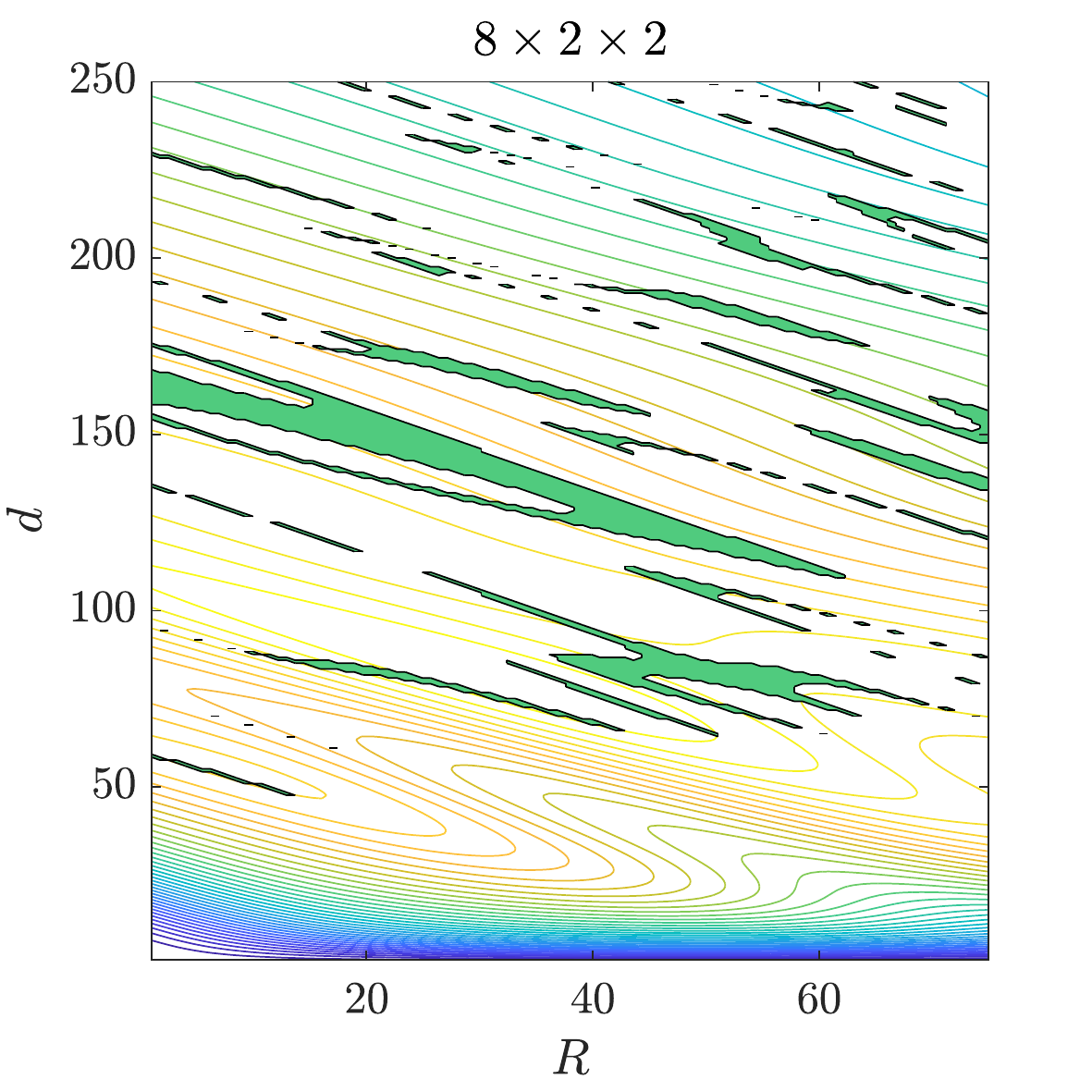}%
    \end{minipage}
    \caption{Flat regions in the discretizted objective functions. The underlying contour plot corresponds to the discretization of $\Lambda\times\mathcal{R}\times\mathcal{D}$ into $50\times50\times50$ points. For each figure, the green region consists of all points at which the euclidean norm of the gradient of the discretized objective function is smaller than 0.05. The discretizations of $\Lambda\times\mathcal{R}\times\mathcal{D}$ are given in the titles, respectively.}
    \label{fig:ApplicationFlatRegions}
\end{figure}

\subsection{Discretization}
For the reasons mentioned above, we will only compare the results obtained by CSG to generic deterministic optimization schemes for various choices of discretization. Since the integration over $\Lambda$ admits no special structure, we always choose an equidistant partition for this dimension of integration. However, for the integration over $\mathcal{R}\times\mathcal{D}$, we can use our knowledge of $\mu_u$ to achieve a better approximation to the true integral. Instead of dividing $\mathcal{R}\times\mathcal{D}$ into an equidistant grid, we utilize the fact that $\tilde{R}$ and $\tilde{d}$ are normal distributed independent from each other. Since, for a normal distribution, $99.7\%$ of all weight is concentrated in the $3\sigma$-interval around the mean value, we may only discretize this portion of the full domain in each step. 

Moreover, we know the precise density function for both $\tilde{R}$ and $\tilde{d}$. Thus, given a design $u_n=(R_n,d_n)$, we will partition $\left( R_n - \tfrac{3}{10}R_n, R_n + \tfrac{3}{10}R_n\right)$ and $\left( d_n - \tfrac{3}{10}d_n, d_n + \tfrac{3}{10}d_n\right)$ not into equidistant intervals, but instead in intervals of equal weight. This procedure is illustrated in \Cref{fig:CDF,fig:DFColor} and produces very good results even for a small number of sample points.

However, as we have already seen in \Cref{fig:SliceSetting1}, even this dedicated discretization scheme introduces additional propbelms into \eqref{eq:ProblemSetting1}. Furthermore, we want to emphasize that choosing a reasonable discretization is a challenge of its own. Not only is there no a priori indication for the general magnitude of the number of points needed, it is also unclear whether or not one should use the same number of points in each direction. 

\begin{figure} 
    \centering
    \begin{minipage}[b][][c]{0.48\textwidth}
        \centering
        \includegraphics[width = \textwidth]{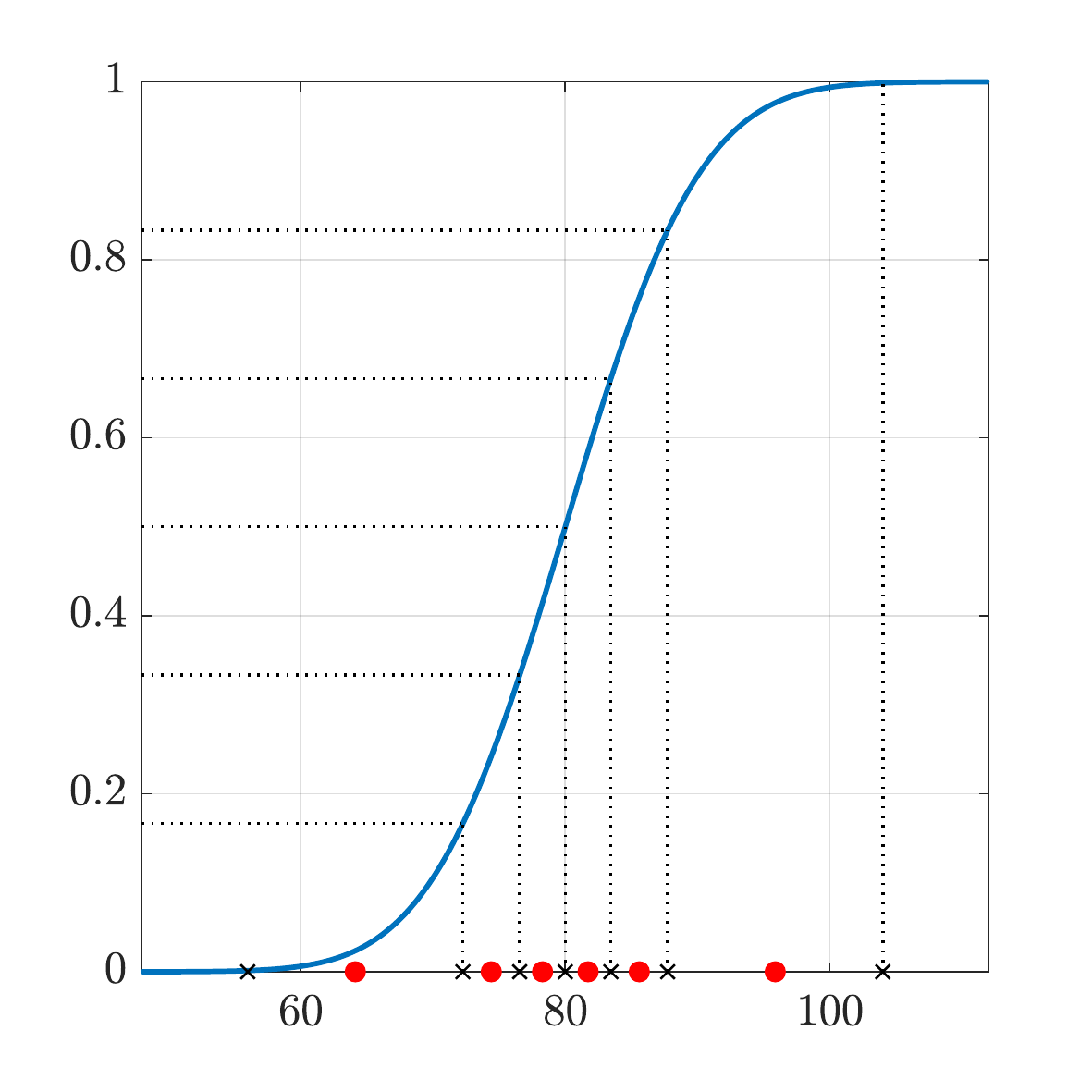}
    \end{minipage}\hfill
    \begin{minipage}[b][][c]{0.48\textwidth}
        \centering
        \includegraphics[width = \textwidth]{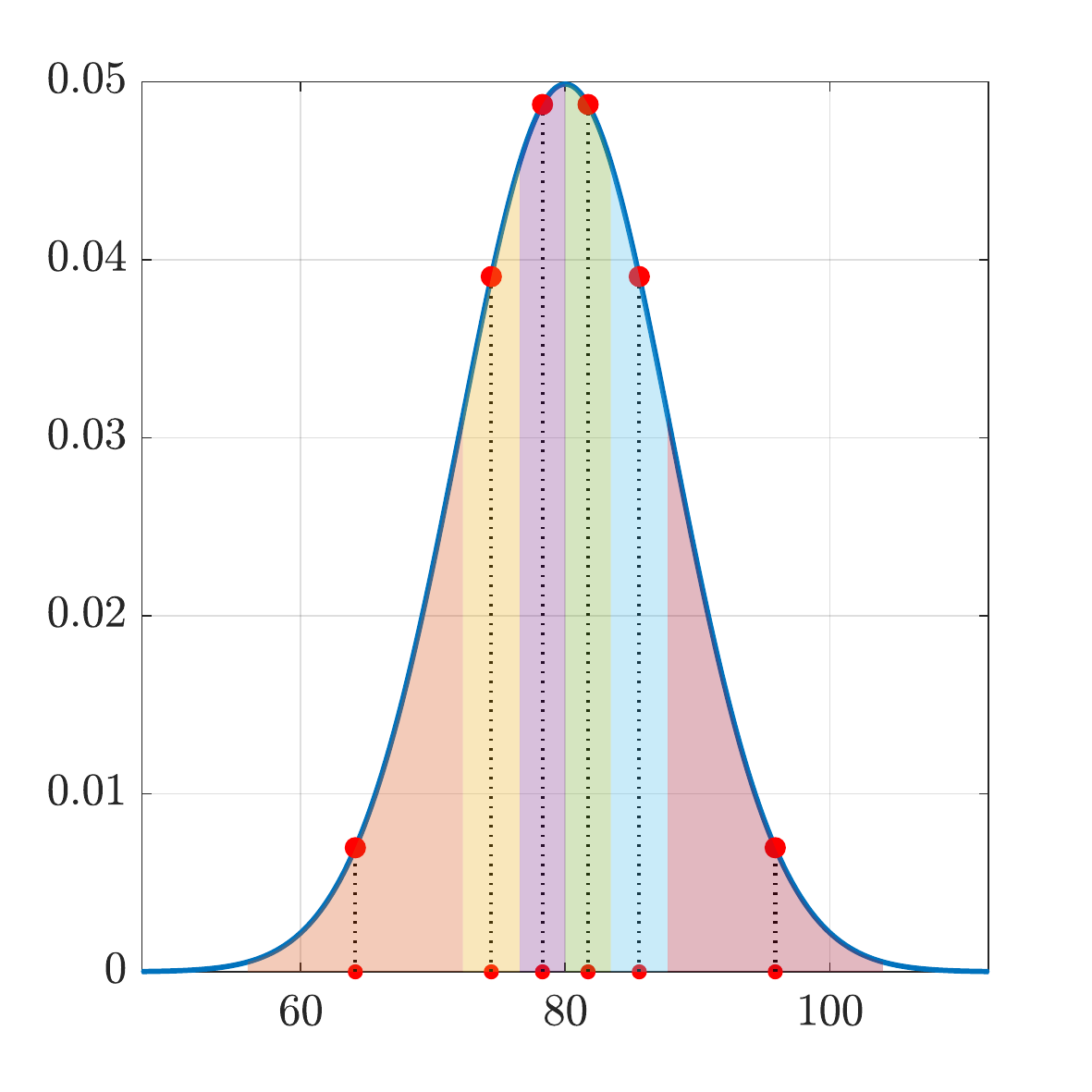}%
    \end{minipage}
    \par
    \begin{minipage}[t][][c]{0.48\textwidth}
        \caption{Cumulative density function for $\tilde{R}$ in the case $R=80$. The six integration points (red dots) are obtained by dividing $(0,1)$ in six intervals of equal size and calculating the midpoints of the resulting preimages (black crosses). Note that the preimages are first projected on the $3\sigma$-interval.}
        \label{fig:CDF}
    \end{minipage}\hfill
    \begin{minipage}[t][][c]{0.48\textwidth}
        \caption{Density function for $\tilde{R}$ in the case $R=80$. The red dots represent the six integration points as detailed in \Cref{fig:CDF}. By their special construction, each shaded region under the curve is of equal area.}
        \label{fig:DFColor}
    \end{minipage}
\end{figure}

\subsection{Numerical Results}\label{subsec:ApplicationsNumerics}
As mentioned above, the restriction to radially symmetric nanoparticles allows us to apply standard blackbox solvers to \eqref{eq:ProblemSetting1}, in order to have a comparison for the CSG results. In our case, we chose the \textit{fmincon} implementation of an interior point algorithm, integrated in MATLAB, as is it an easy-to-use blackbox algorithm that yields reproducible results. 

Specifically, we compared the results of SCIBL-CSG with empirical weights on $\mathcal{R}\times\mathcal{D}$ and exact hybrid weights on $\Lambda$ (cf. \cite[Section 3]{CSGPart1}) to the fmincon results for three different discretization schemes of $\Lambda\times\mathcal{R}\times\mathcal{D}$. Two of these are equal in each dimension ($10\times 10\times10$ and $7\times7\times7$), while the last one is asymmetric ($8\times2\times2$). Once again, we want to stress that finding an appropriate discretization scheme already requires a thorough analysis of \eqref{eq:ProblemSetting1}. The specific choices listed above represent three of the most promising candidates found during our investigation. 

As we consider this example to be a prototype for more advanced settings from topology optimization, e.g., switching the setting to the DDA model later, we compare the different approaches with respect to the number of inner gradient evaluations, since this is by far the most time-consuming step in these cases. To be precise, an evaluation represents the calculation of $\Abs$, $\Sca$, $\Geo$, $\nabla\Abs$, $\nabla\Sca$ and $\nabla\Geo$ for a single $(\lambda,\tilde{R},\tilde{d})\in\Lambda\times\mathcal{R}\times\mathcal{D}$.

Since the produced iterates depend on the initial design, we randomly selected 500 starting points in the whole design domain $\U=[1,75]\times[1,250]$. In each optimization run, the total number of evaluations was limited to 50.000 for fmincon and to 5.000 for SCIBL-CSG. To obtain an overview of the general performance of the different approaches, we take snapshots of all iterates after different amounts of evaluations. The results are given in \Cref{fig:ApplicationResults01} and \Cref{fig:ApplicationResults02} and yield a good impression on how fast each method tends to find solutions to \eqref{eq:ProblemSetting1}. Note that, for the sake of readability and better comparison, the final CSG iterates after 5.000 evaluations are shown in all graphs labeled with a higher number of total evaluations.

By comparing \Cref{fig:ApplicationResults01} and \Cref{fig:ApplicationResults02} with \Cref{fig:ApplicationFlatRegions}, we observe that the artificial flat regions discussed earlier indeed slow down the optimization progress for all choices of prediscretization. Furthermore, we note that only the highest resolution $10\times10\times10$ overcomes this approximation error, at the cost of the largest amount of evaluations needed. In contrast, the resolutions $7\times7\times7$ and $8\times2\times2$ converge much faster, but some of the final designs are no stationary points of \eqref{eq:ProblemSetting1}. Out of the 500 optimization runs we performed, $7\times7\times7$ converged to a wrong design, i.e., artificial local minimum, 16 times (3.2\%). For $8\times2\times2$, a wrong design was found in 218 (43.6\%) instances, see \Cref{fig:ApplicationResults02}. 

Lastly, we are interested in the performance of each method with respect to $J(u_n)$ over the course of the iterations. Since each local solution to \eqref{eq:ProblemSetting1} admits a different objective function value, we focus only on the global maximum. For all approaches, we selected all runs whose final designs are closer to the global maximum of \eqref{eq:ProblemSetting1} than to any other stationary point. The results are shown in \Cref{fig:MedianSetting1} and \Cref{fig:QuantilesSetting1}.

\begin{figure} 
    \centering
    \begin{minipage}[b][][c]{0.48\textwidth}
        \centering
        \includegraphics[width = \textwidth]{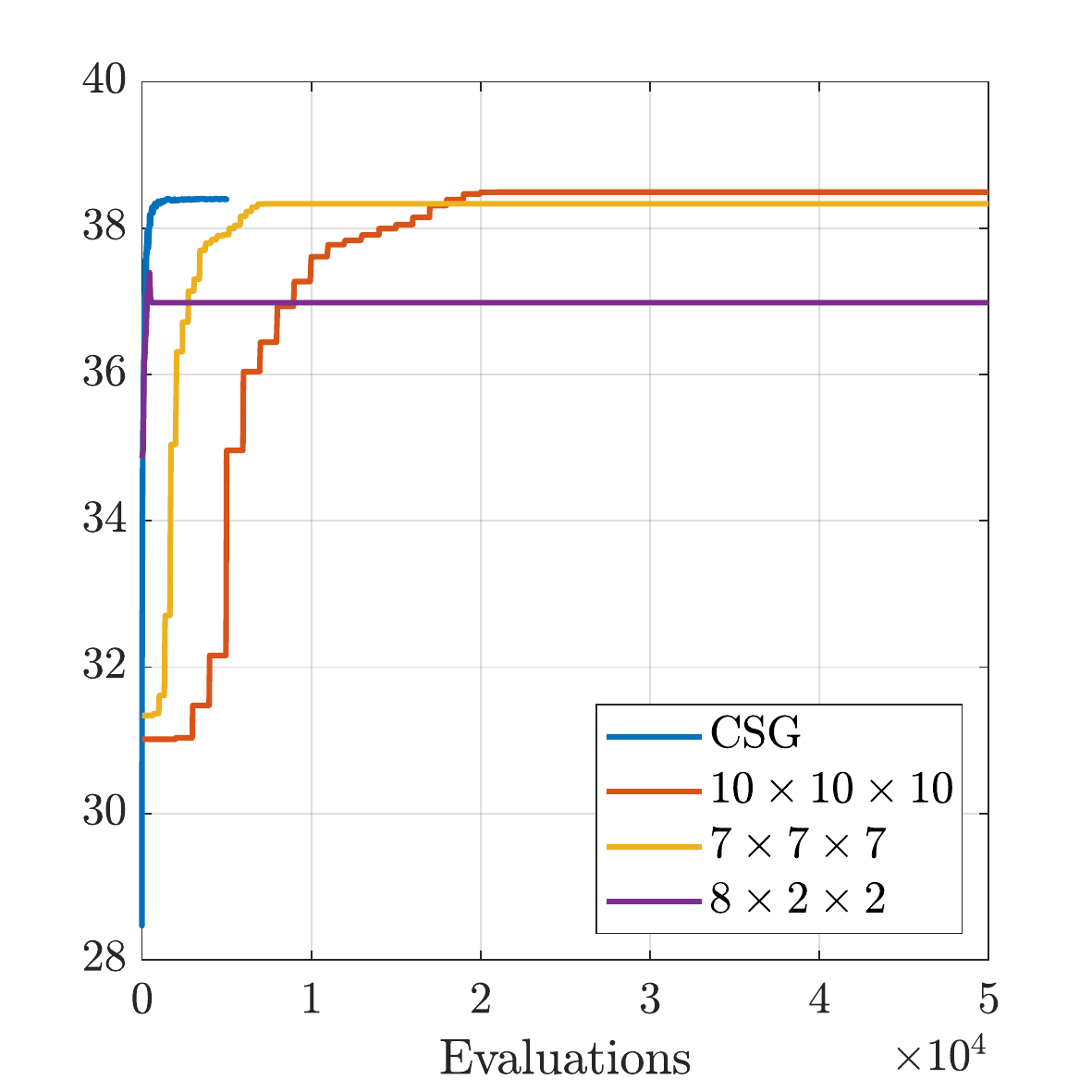}
    \end{minipage}\hfill
    \begin{minipage}[b][][c]{0.48\textwidth}
        \centering
        \includegraphics[width = \textwidth]{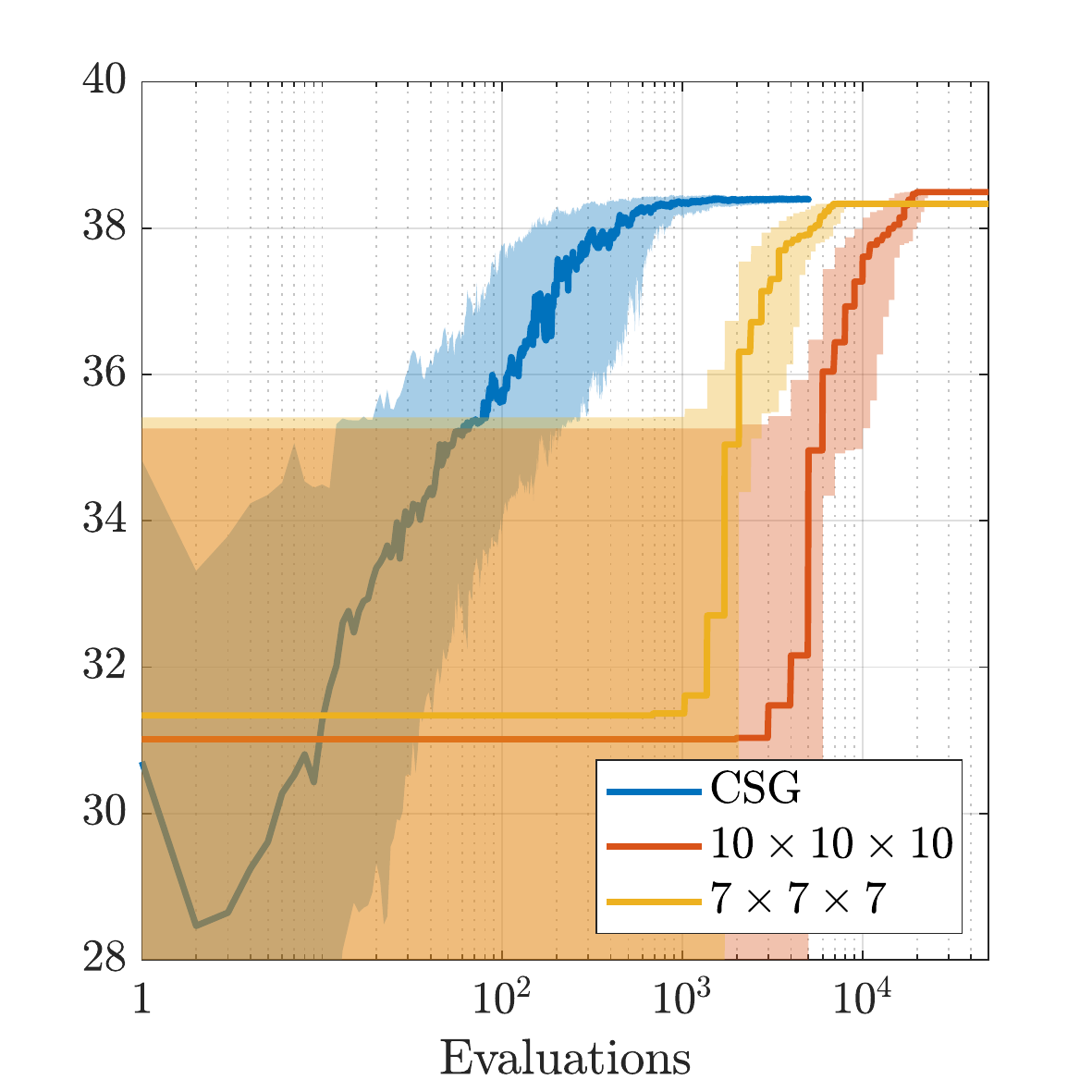}%
    \end{minipage}
    \par
    \begin{minipage}[t][][c]{0.48\textwidth}
        \caption{Median objective function value of all optimization runs in which the final design was closer to the global maximum of \eqref{eq:ProblemSetting1} than to any other stationary point. The values were obtained using a discretization into $50\times50\times50$ points.}
        \label{fig:MedianSetting1}
    \end{minipage}\hfill
    \begin{minipage}[t][][c]{0.48\textwidth}
        \caption{The medians presented in \Cref{fig:MedianSetting1} (solid lines) and the corresponding quantiles $P_{0.25,0.75}$, indicated by the shaded areas. For better visibility, the number of evaluations is scaled logarithmically and the discretization $8\times2\times2$ was discarded.}
        \label{fig:QuantilesSetting1}
    \end{minipage}
\end{figure}

\begin{figure}
    \centering
    \includegraphics[scale=0.6,clip,trim=70 15 45 0]{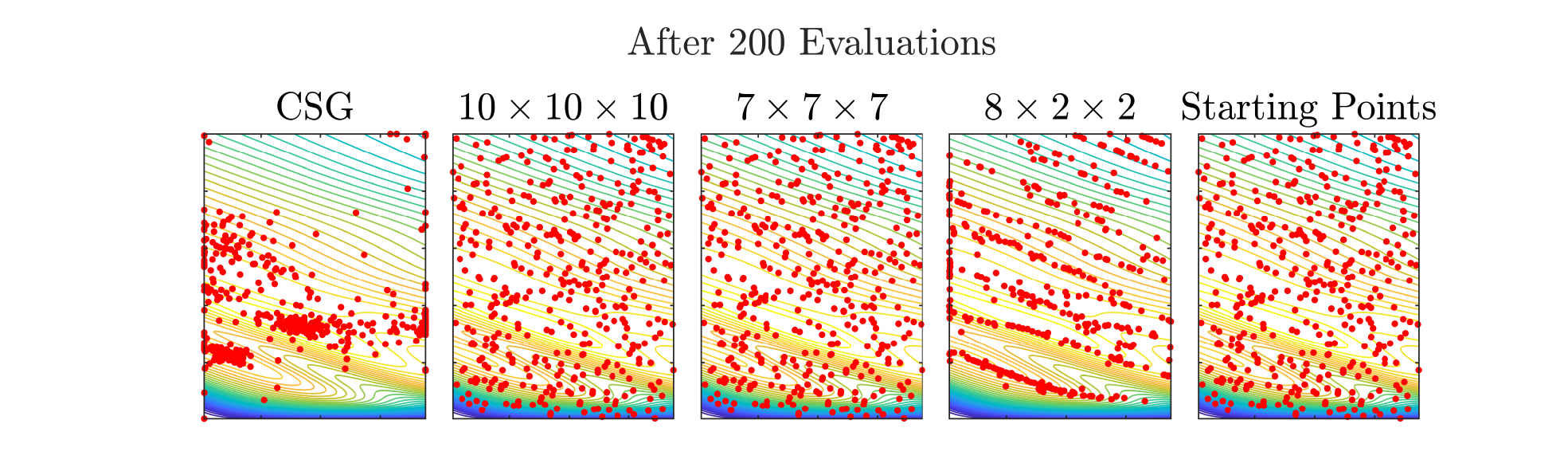}
    \includegraphics[scale=0.6,clip,trim=70 15 45 0]{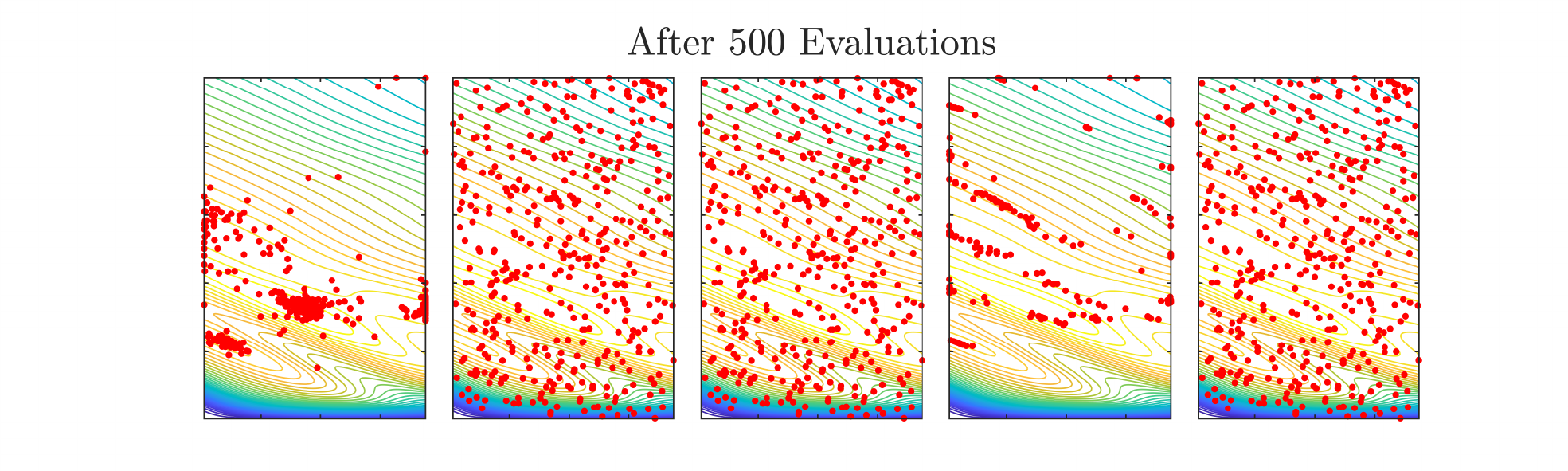}
    \includegraphics[scale=0.6,clip,trim=70 15 45 0]{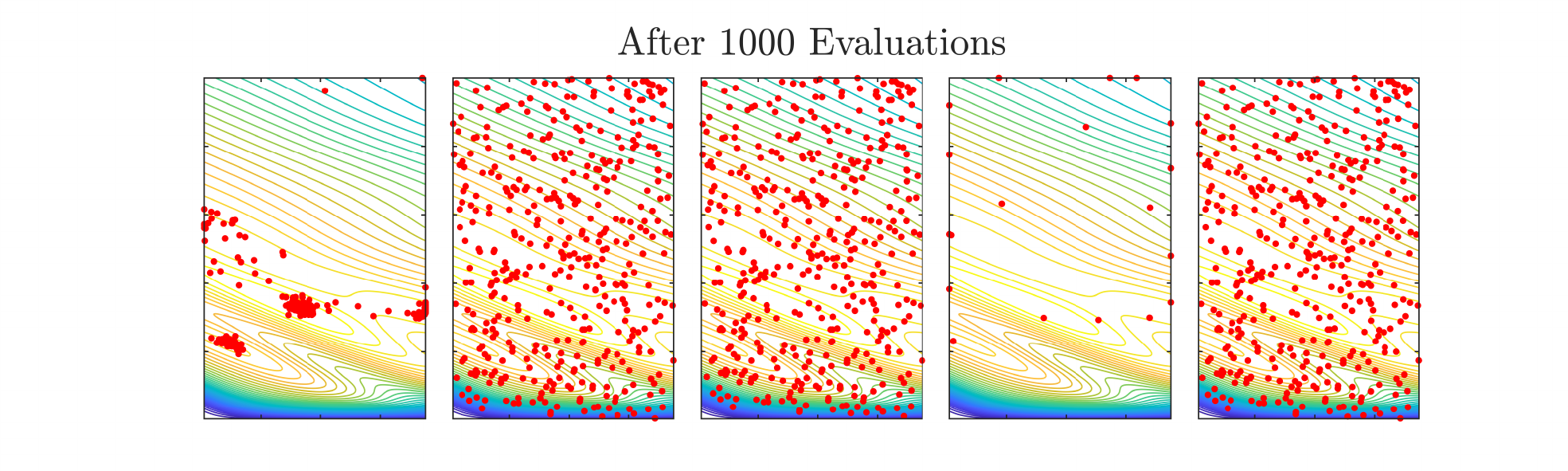}
    \includegraphics[scale=0.6,clip,trim=70 15 45 0]{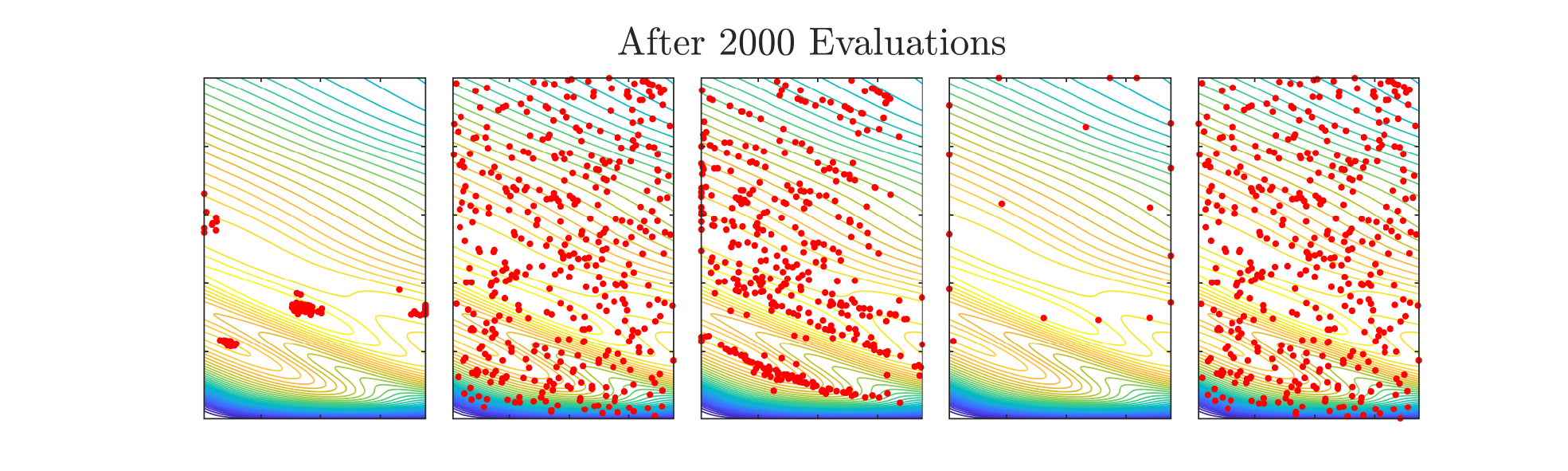}
    \includegraphics[scale=0.6,clip,trim=70 15 45 0]{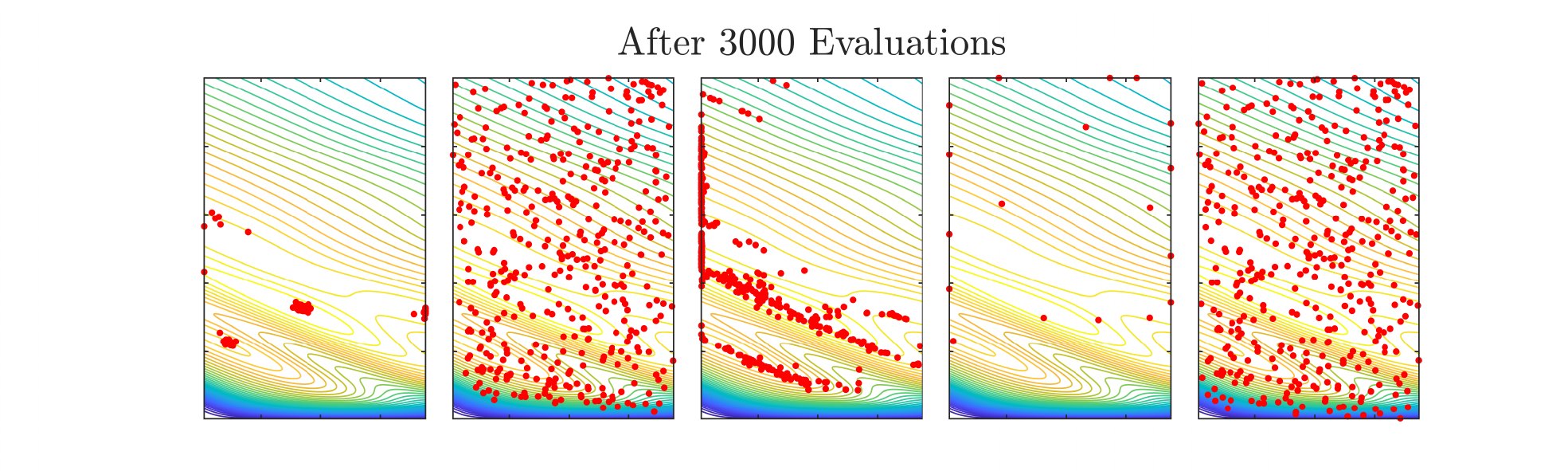}
    \caption{Iterates of the different optimization approaches for \eqref{eq:ProblemSetting1} in the whole design domain $\U=[1,75]\times[1,250]$. For fmincon, the discretization of $\Lambda\times\mathcal{R}\times\mathcal{D}$ is given in the titles, respectively. To measure the progress, the starting points are also shown. As mentioned above, an evaluation corresponds to the calculation of $\Abs$, $\Sca$, $\Geo$, $\nabla\Abs$, $\nabla\Sca$ and $\nabla\Geo$ for one combination $(\lambda,\tilde{R},\tilde{d})\in\Lambda\times\mathcal{R}\times\mathcal{D}$. Again, the underlying contours are obtained by discretizing $\Lambda\times\mathcal{R}\times\mathcal{R}$ into $50\times50\times50$ points.}
    \label{fig:ApplicationResults01}
\end{figure}

\begin{figure}
    \centering
    \includegraphics[scale=0.6,clip,trim=70 15 45 0]{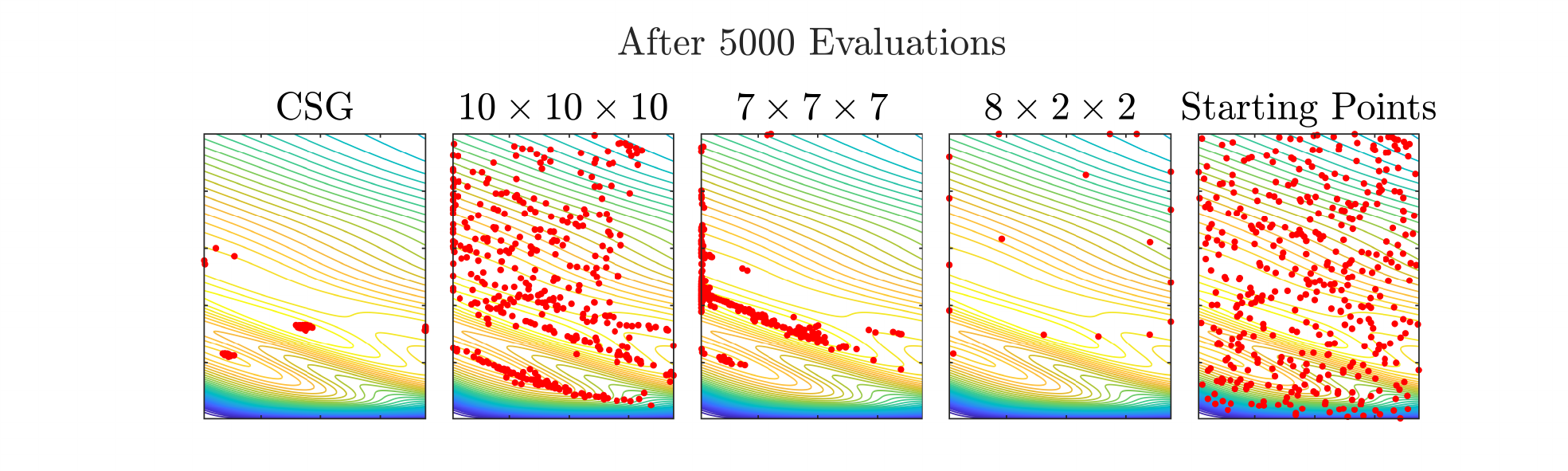}
    
    {\transparent{0.4}{\includegraphics[scale=0.6,clip,trim=70 15 410 0]{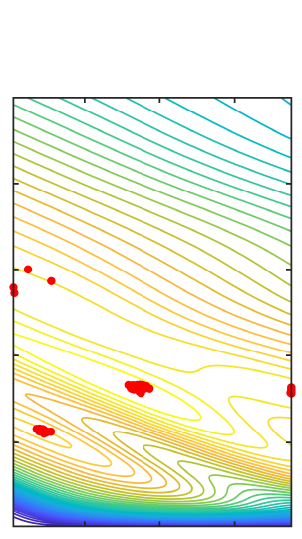}}}
    {\includegraphics[scale=0.6,clip,trim=162 15 45 0]{Figures/RobustPaint/ResultsFigures/Application7.pdf}}
    
    {\transparent{0.4}{\includegraphics[scale=0.6,clip,trim=70 15 410 0]{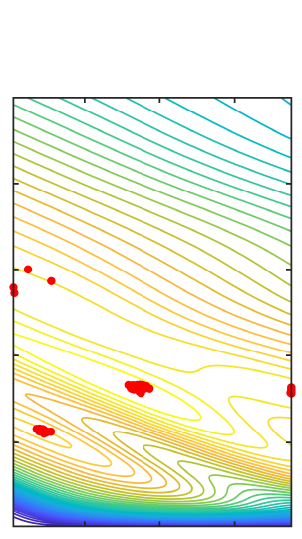}}}
    {\includegraphics[scale=0.6,clip,trim=162 15 45 0]{Figures/RobustPaint/ResultsFigures/Application8.pdf}}
    
    {\transparent{0.4}{\includegraphics[scale=0.6,clip,trim=70 15 410 0]{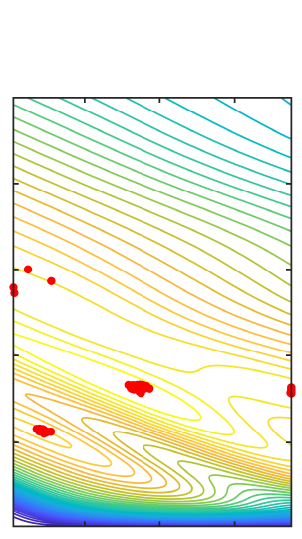}}}
    {\includegraphics[scale=0.6,clip,trim=162 15 45 0]{Figures/RobustPaint/ResultsFigures/Application9.pdf}}
    
    {\transparent{0.4}{\includegraphics[scale=0.6,clip,trim=70 15 410 0]{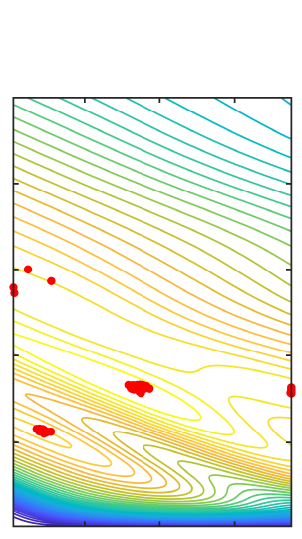}}}
    {\includegraphics[scale=0.6,clip,trim=162 15 45 0]{Figures/RobustPaint/ResultsFigures/Application10.pdf}}
    \caption{Continuation of the results for \eqref{eq:ProblemSetting1} presented in \Cref{fig:ApplicationResults01}. Since CSG was stopped after 5.000 evaluations, the iterates do not change afterwards, but are still shown as a point of reference. In the last row, final designs obtained by $7\times7\times7$ and $8\times2\times2$, which do not correspond to stationary points of \eqref{eq:ProblemSetting1}, are highlighted in blue.}
    \label{fig:ApplicationResults02}
\end{figure}

\subsection{Optimization in the DDA Model}\label{sec:DDAOptim}
As a final example from application, we drop the restriction to core shell particles and consider hematite nanoparticles of arbitrary shape with the DDA model. While the setting is very similar to the setting analyzed above, there are some minor differences.

First, we slightly change the weights appearing in the objective function:
\begin{equation}
    \max_{u\in\U}\quad \tfrac{1}{2}\,\CIEL(u) + \tfrac{1}{2}\,\CIEa(u).\label{eq:ProblemSetting2}
\end{equation}
This change was made purely for aesthetics, as the weights in \eqref{eq:ProblemSetting1} favour radially symmetric solutions, while \eqref{eq:ProblemSetting2} admits local solutions with a more interesting design structure. 

Furthermore, we do not assume a particle design distribution anymore, since it is unclear, how such a general shape distribution should look like. However, as the particles are no longer radially symmetric, we now have to consider the orientation of the particle with respect to the incoming light ray instead. Therefore, the $K$ and $S$ values explained in the introduction of this setting need to be averaged over all possible orientations, i.e.,
\begin{align*}
    K(u,\lambda) &= \frac{1}{\left\vert\mathbb{S}^2\right\vert}\iint_{\mathbb{S}^2}\Abs(u,\lambda,\nu)\mathrm{d}\nu 
\intertext{and}
    S(u,\lambda) &= \frac{1}{\left\vert\mathbb{S}^2\right\vert}\iint_{\mathbb{S}^2} \Sca(u,\lambda,\nu)\big(1-\Geo(u,\lambda,\nu)\big)\mathrm{d}\nu.
\end{align*}
Here, $\mathbb{S}^2$ denotes the unit sphere and the particle orientation $\nu$ is assumed to be distributed uniformly random over all possible directions.

The design domain is a ball of 300nm diameter, discretized into $n_0=65752$ dipole cells. The design $u\in[0,1]^{n_0}$ gives the relative amount of hematite to water in each cell. The optical properties of intermediate (grey) material $u^{(i)}\in(0,1)$ are generated by linear interpolation between the respective properties of water and hematite.

Generally, one would combine filtering techniques and greyness penalization to obtain a smooth final design without intermediate material (see, e.g., \cite{sigmund2007morphology}). However, we explicitly refrain from doing so to present a clear analysis of the CSG performance, without interference from secondary layers of smoothing techniques.

As mentioned above, the change to the DDA model significantly increases the computational cost of evaluating $\Sca$, $\Abs$ and $\Geo$ for a given $(u,\lambda,\nu)\in\U\times\Lambda\times\mathbb{S}^2$. Thus, the deterministic approaches used in the previous setting are no longer computationally feasible.

Furthermore, we want to use this example to analyze the impact of the chosen norm on $\U\times\Lambda\times\mathbb{S}^2$, appearing in the nearest neighbor calculation, which was already mentioned in \cite[Section 3.5]{CSGPart1}. To be precise, calculating the CSG integration weights requires the definition of an outer norm 
\begin{equation*}
    \big\Vert(u^\ast,\lambda^\ast,\nu^\ast)\big\Vert_{\text{Out}} = c_u\Vert u^\ast \Vert_\sU + c_\lambda\Vert \lambda^\ast\Vert_{_\Lambda} + c_\nu\Vert \nu^\ast\Vert_{_{\mathbb{S}^2}},
\end{equation*}
where $\Vert\cdot\Vert_\sU$, $\Vert\cdot\Vert_{_\Lambda}$ and $\Vert\cdot\Vert_{_{\mathbb{S}^2}}$ denote norms on the corresponding inner spaces and $c_u,c_\lambda,c_\nu>0$. In this application, we choose the euclidean norm $\Vert\cdot\Vert_{_2}$ for each inner space. Additionally, we fix $c_u = 1$, but consider different coefficients $c_\lambda$ and $c_\nu$. 

\begin{figure}
    \centering
    \includegraphics[scale=0.07]{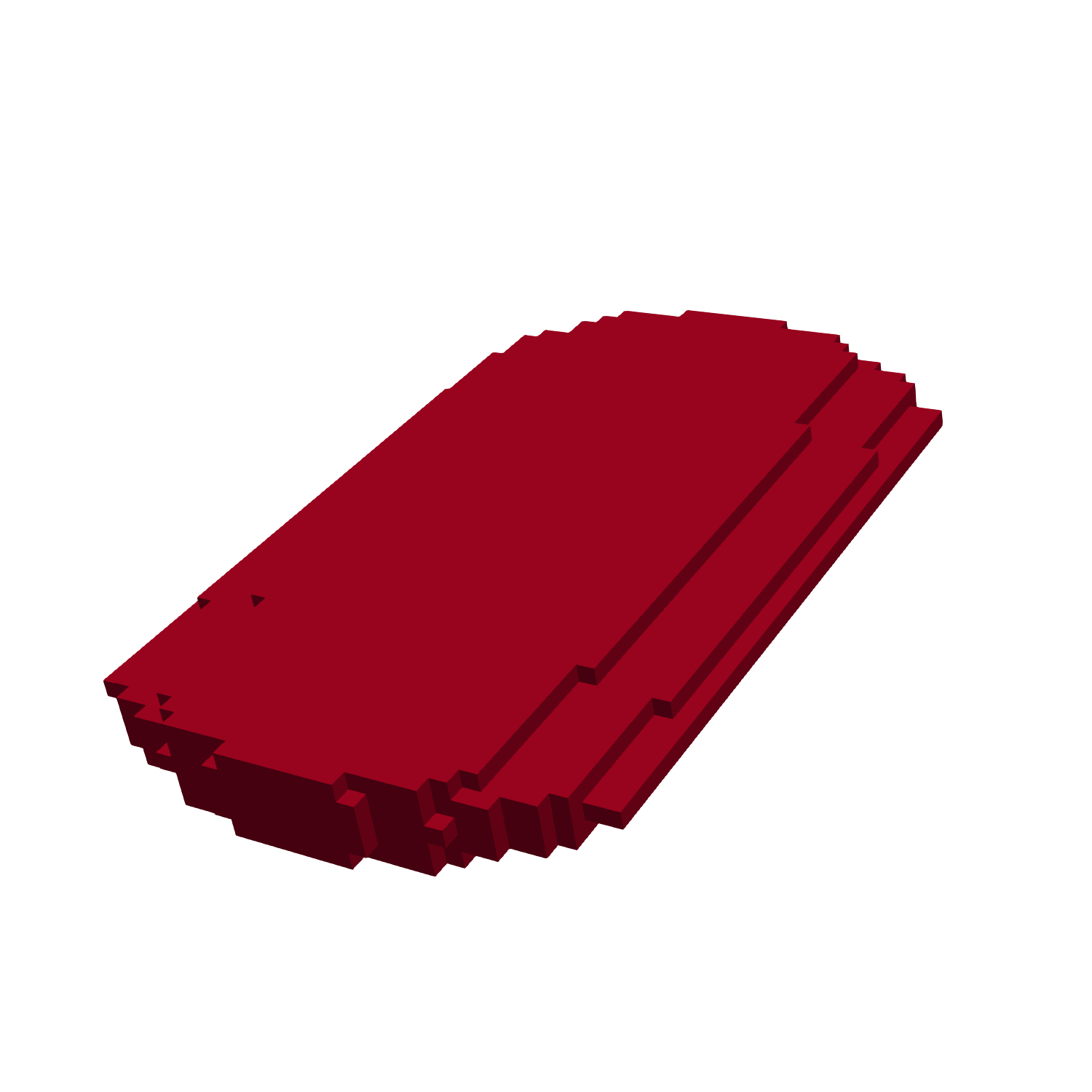}
    \includegraphics[scale=0.07]{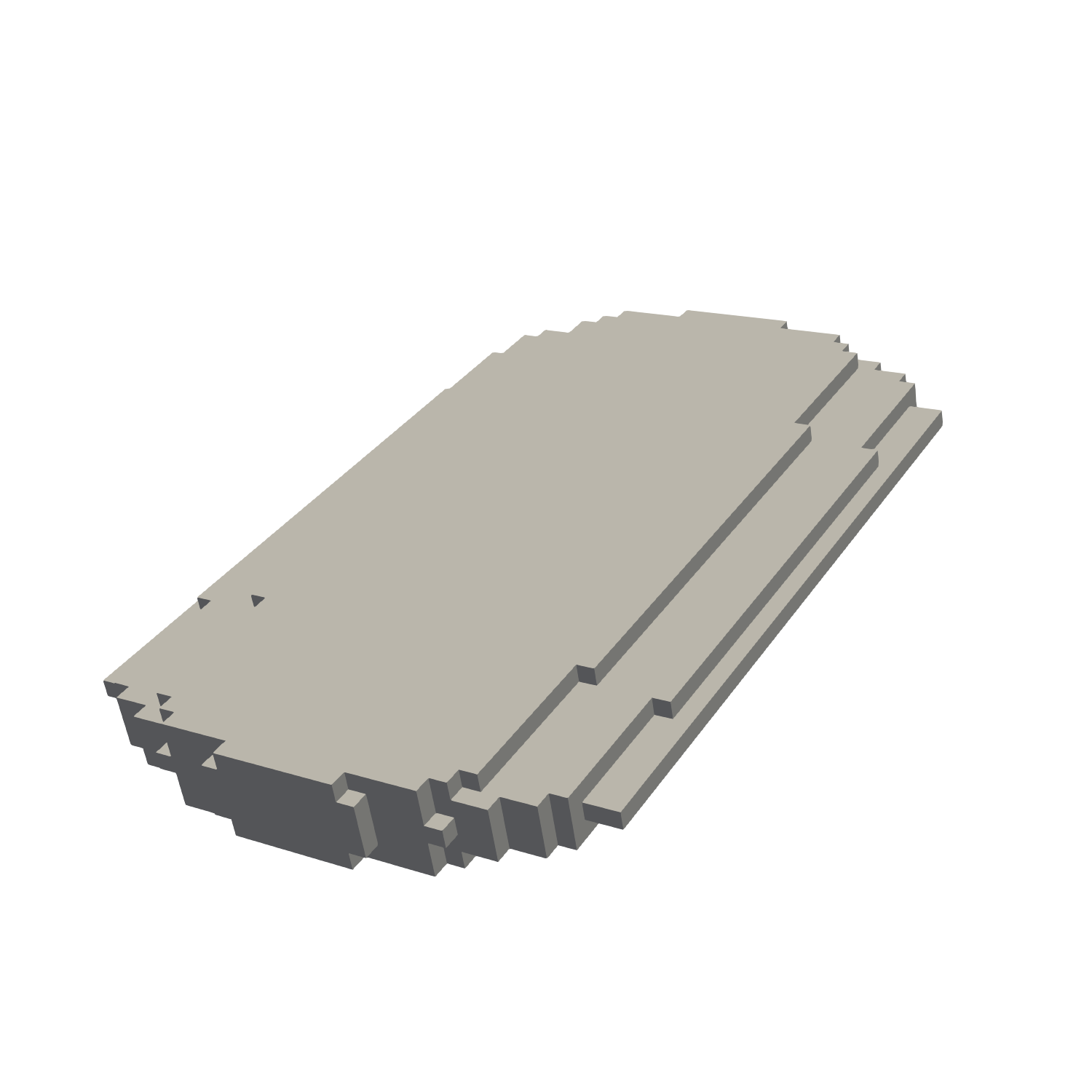}
    \includegraphics[scale=0.07]{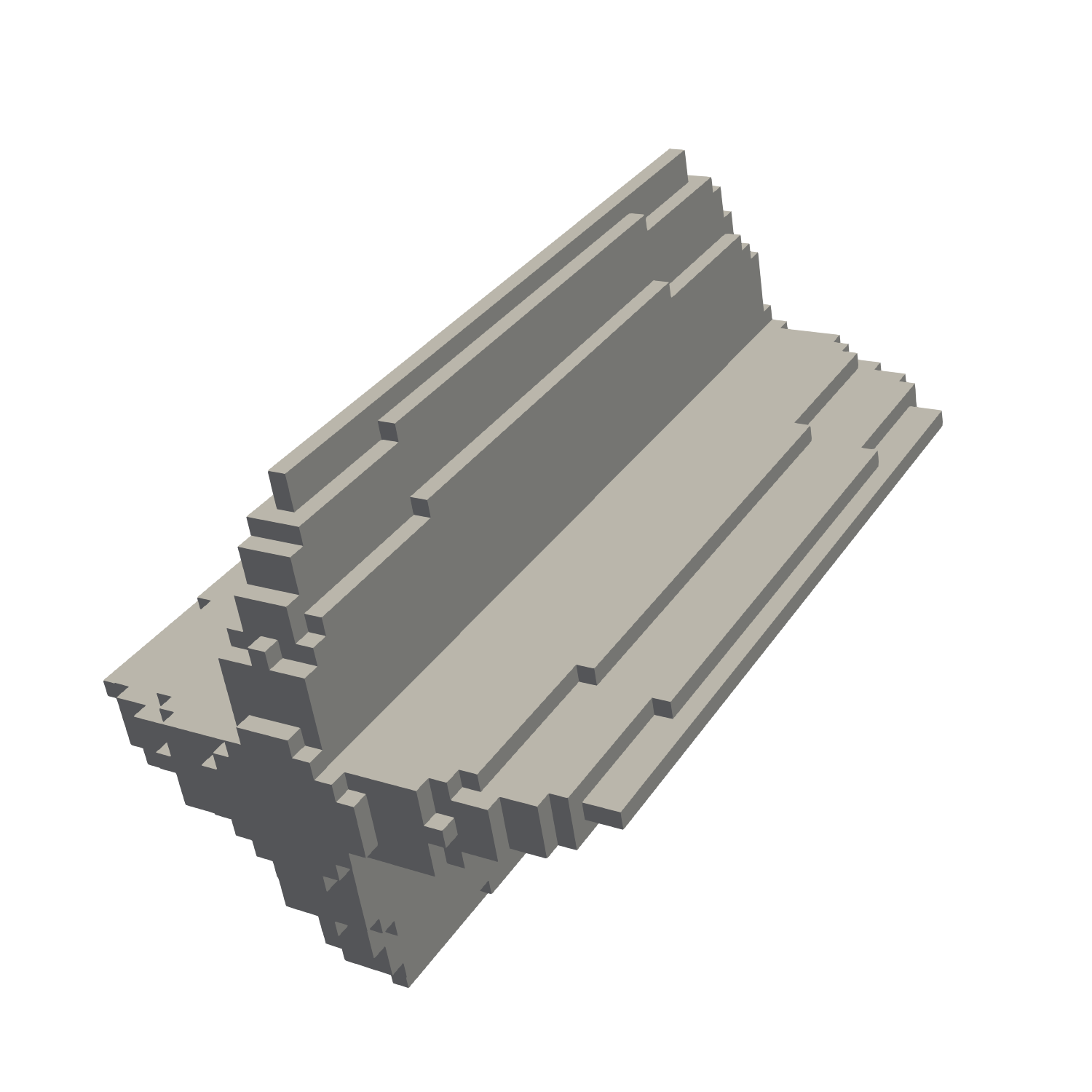}

    \includegraphics[scale=0.07]{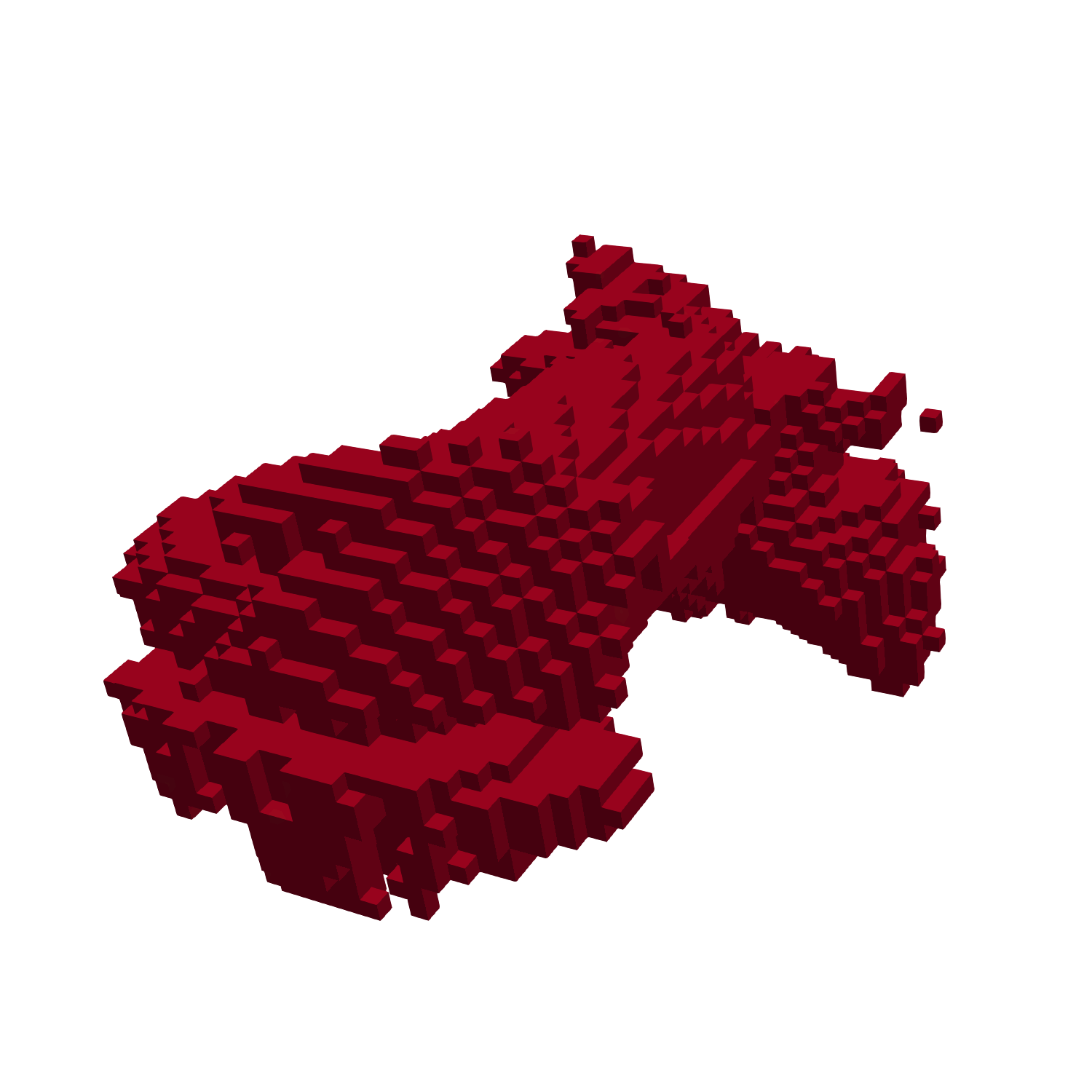}
    \includegraphics[scale=0.07]{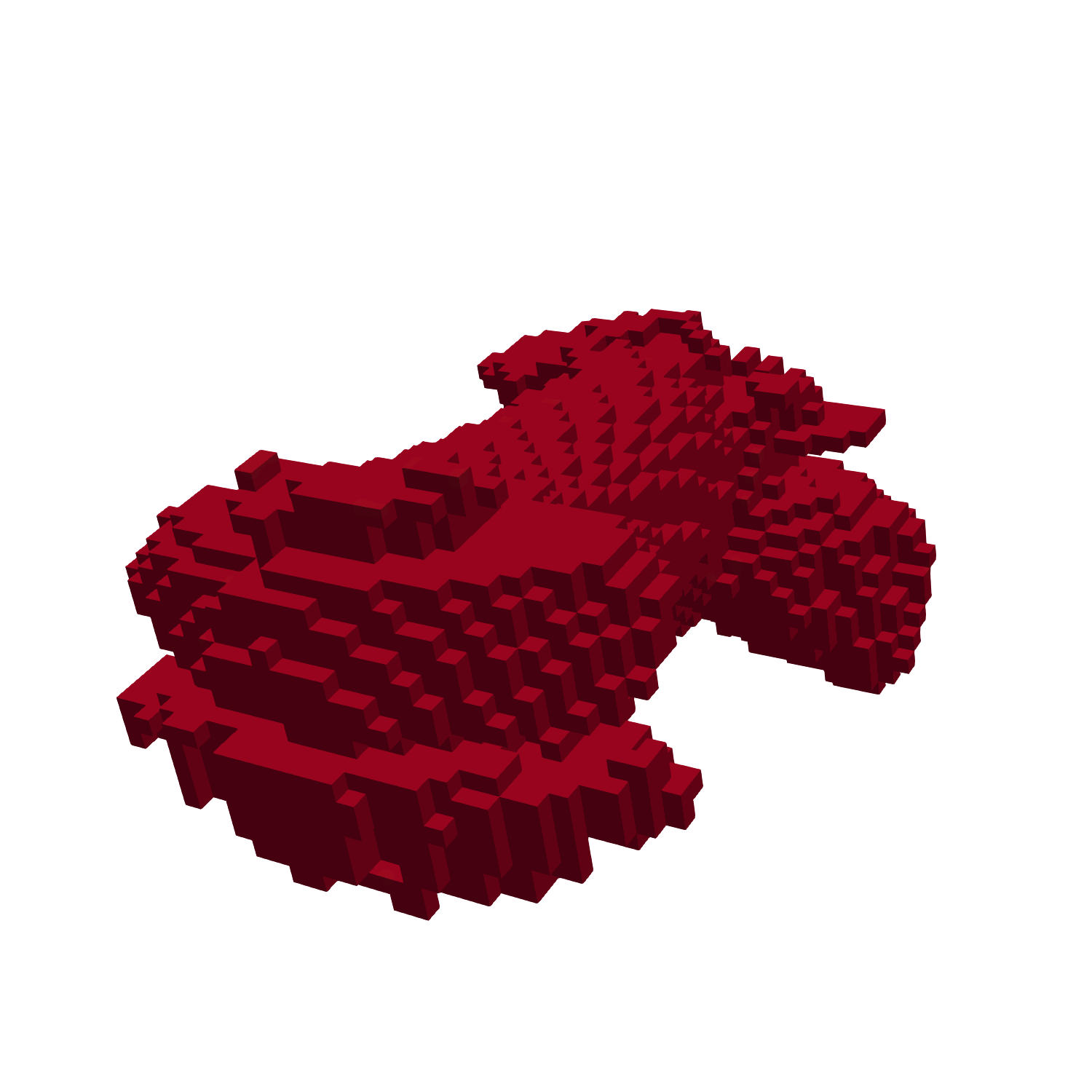}
    \includegraphics[scale=0.07]{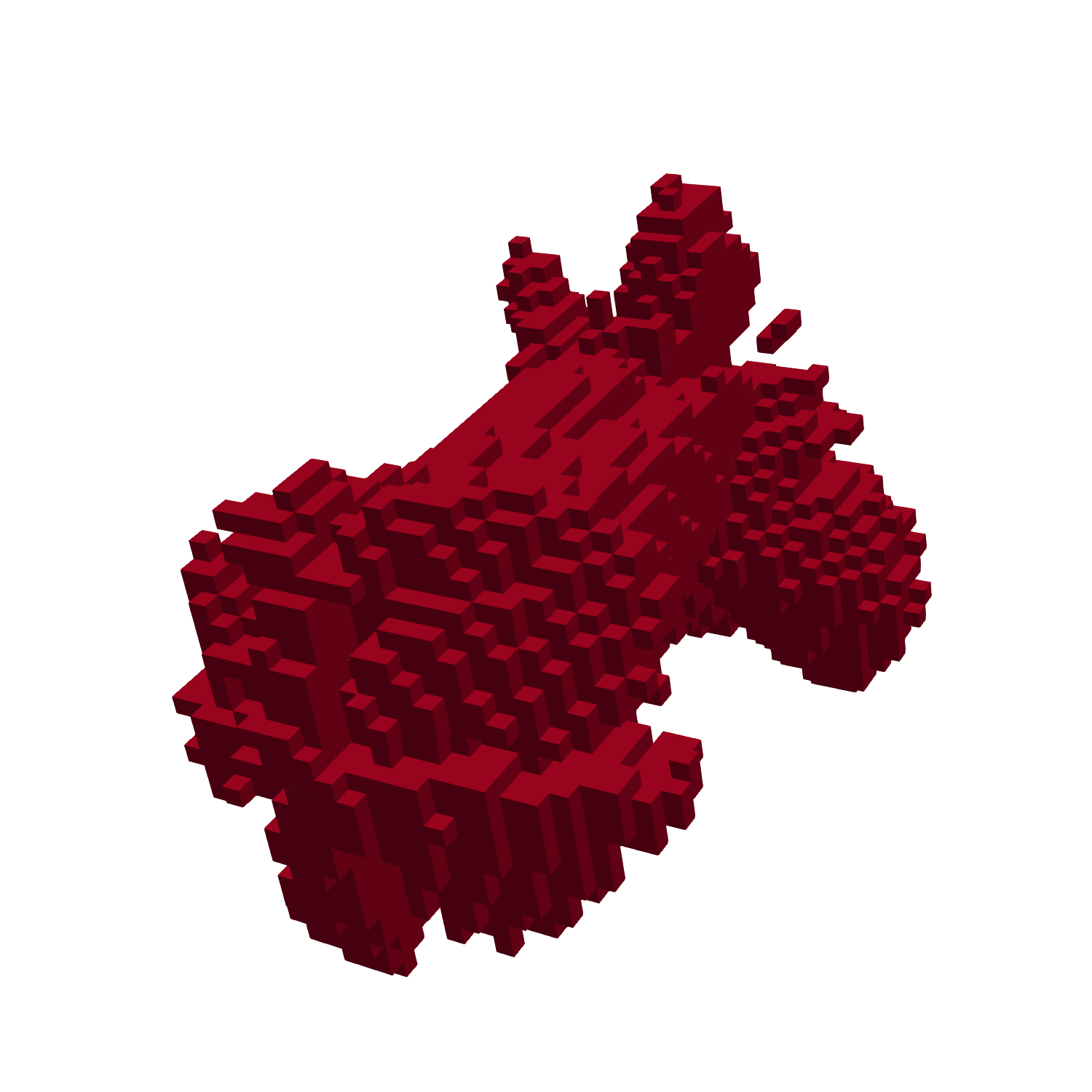}
    
    \caption{Representation of the initial designs (top row). Red boxes correspond to cells consisting purely of hematite, while grey boxes indicate an artificial intermediate material, consisting of 50\% hematite and 50\% water. For later references, we denote the initial designs by \textit{plate (100\%)}, \textit{plate (50\%)} and \textit{screwdriver (50\%)}, respectively.\\ The different final designs, obtained by 5.000 iterations of SCIBL-CSG with outer norm (a) are shown in the bottom row. For better visibility, cells with less than $50\%$ hematite are considered as pure water and left out of the visualization. For each final design, the amount of cells discarded in this fashion is less than 100 (less than $0.15\%$ of all cells).}
    \label{fig:DDA_initial_final_desgin}
\end{figure}

For the optimization, we consider three different initial designs, which are shown in \Cref{fig:DDA_initial_final_desgin}, top row. The objective function value as well as the values of $\CIEL$, $\CIEa$ and $\CIEb$ for these designs were computed using the CSG method with fixed design, i.e., with constant step size $\tau=0$, and verified by Monte Carlo (see, e.g., \cite{caflisch1998monte}) integration. For one of the initial designs, the objective function value approximation of CSG and Monte Carlo integration with respect to the number of evaluations and different choices of $\Vert\cdot\Vert_{_{\text{Out}}}$ is shown in \Cref{fig:DDAInitComp}. 
\begin{figure}
  \begin{minipage}[c]{0.5\textwidth}
    \includegraphics[width=\textwidth]{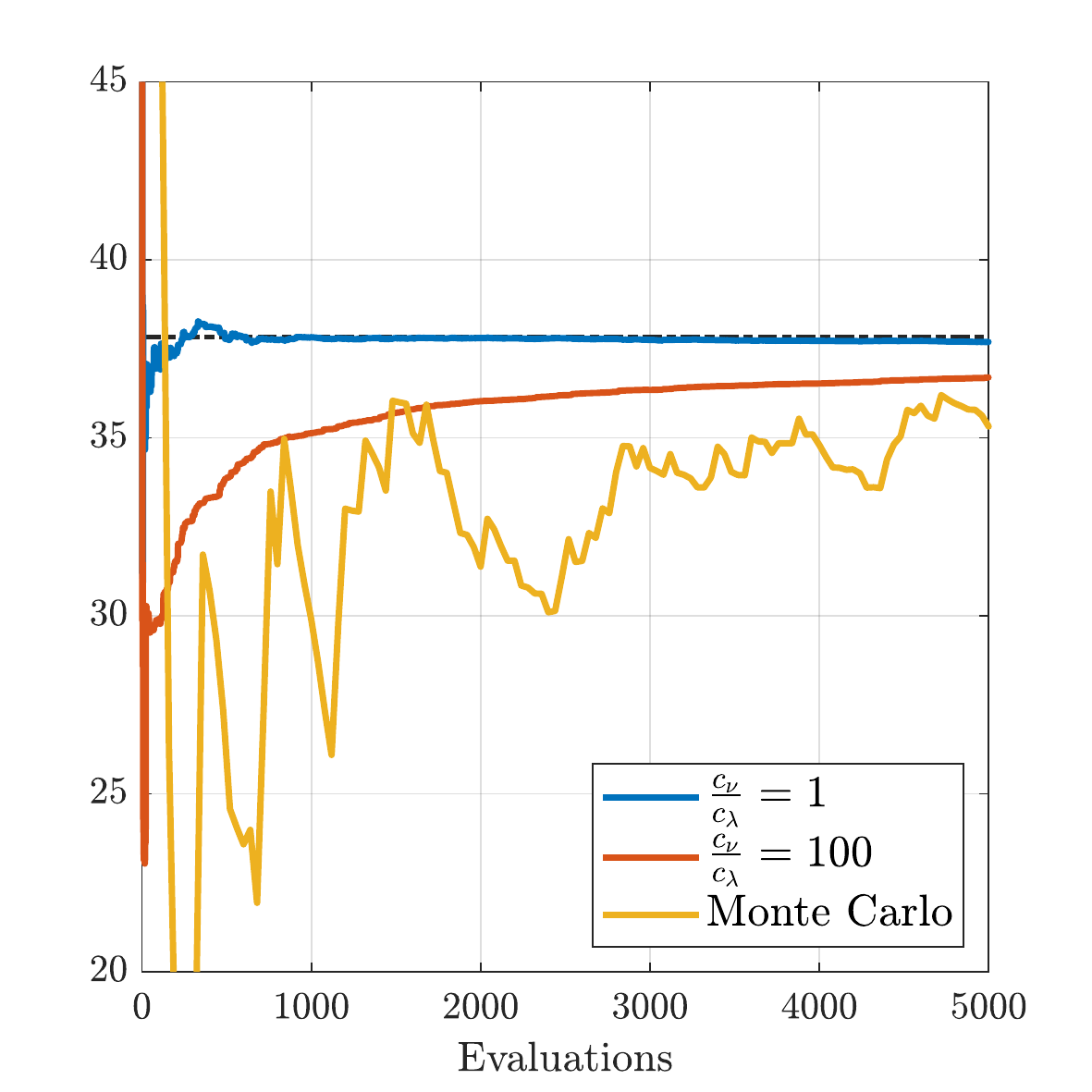}
  \end{minipage}\hfill
  \begin{minipage}[c]{0.45\textwidth}
    \caption{Objective function approximation for the \textit{screwdriver (50\%)} design. The blue and orange curve show the results for CSG with fixed step size ${\tau = 0}$ and different coefficients of the outer norm $\Vert\cdot\Vert_{_{\text{Out}}}$. For Monte Carlo, each inner integral over $\mathbb{S}^2$ was approximated using 40 random directions. The true objective function value ${J^\ast\approx 37.84}$ is indicated by the dashed line. The Monte Carlo results are truncated for the sake of readability, as it requires over 8.000 evaluations to reach a good approximation to $J^\ast$.}
    \label{fig:DDAInitComp}
  \end{minipage}
\end{figure}

Each design was optimized with SCIBL-CSG, using inexact hybrid weights for the integration over $\mathbb{S}^2$ and exact hybrid weights for the integration over $\Lambda$. For $\Vert\cdot\Vert_{_{\text{Out}}}$, we considered four different choices of the parameters:
\begin{enumerate}
    \item[(a)] $c_u = 1$, $c_\lambda=100$ and $c_\nu = 100$
    \item[(b)] $c_u = 1$, $c_\lambda=1$ and $c_\nu = 1$
    \item[(c)] $c_u = 1$, $c_\lambda=\tfrac{1}{100}$ and $c_\nu = 1$
    \item[(d)] $c_u = 1$, $c_\lambda=\tfrac{1}{100}$ and $c_\nu=\tfrac{1}{100}$
\end{enumerate}
The results in case (a) for all three initial designs are presented in \Cref{fig:DDAOptimAllDesigns} and the respective design evolution for the initial design \textit{screwdriver (50\%)}, shown in \Cref{fig:DDA_initial_final_desgin} top row, is depicted in \Cref{fig:DDA_designs_ScHa}. The corresponding final designs, obtained after 5.000 SCIBL-CSG iterations, are presented in \Cref{fig:DDA_initial_final_desgin}, bottom row.
As a second measure for convergence in the design space, the evolution of the norm distance to the respective final designs are shown in \Cref{fig:DesingChangeNorm} for all three initial designs.

Comparing \Cref{fig:DDAInitComp} and \Cref{fig:DDAOptimAllDesigns}, we notice that CSG, using an appropriate outer norm, finds an optimized design almost as fast as it computes the objective function value for a given design. In other words: The full optimization process is only slightly more expensive that the simple evaluation of a single design. Moreover, CSG finds an optimal solution to \eqref{eq:ProblemSetting2} long before the Monte Carlo approximation to the initial objective function value is converged.
\begin{figure}
  \begin{minipage}[c]{0.5\textwidth}
    \includegraphics[width=\textwidth]{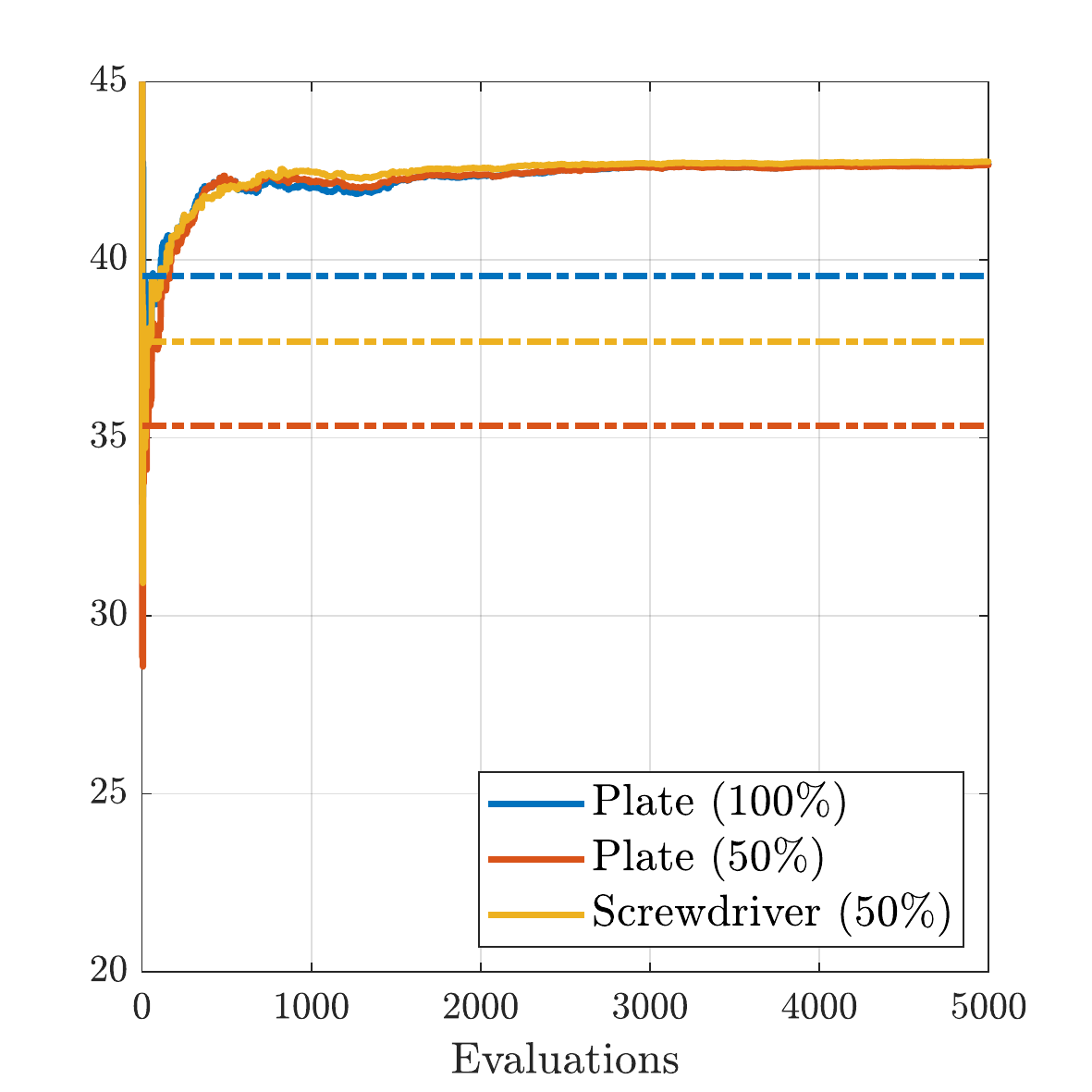}
  \end{minipage}\hfill
  \begin{minipage}[c]{0.45\textwidth}
    \caption{CSG objective function approximations during the optimization process for all initial designs and choice (a) for $\Vert\cdot\Vert_{_{\text{Out}}}$, i.e., $c_u=1$, $c_\lambda=100$ and $c_\nu=100$. The dashed lines indicate the objective function values of each initial design, respectively.}
    \label{fig:DDAOptimAllDesigns}
  \end{minipage}
\end{figure}

It should, of course, also be noted, that choosing $\Vert\cdot\Vert_{_\text{Out}}$ should be done with caution, as \Cref{fig:DDAPlateFull3} shows. While case (a) is, to the best of our knowledge, \textit{not} optimal by any means, cases (b) and (c) clearly show worse results. Choosing $\Vert\cdot\Vert_{_\text{Out}}$ extremely poorly, i.e., case (d), can even have devastating effects on the performance, see \Cref{fig:DDAPlateFull4}.

This, however, could also imply that the performance might be significantly improved, if problem specific inner and outer norms would be chosen. Especially in even more complex settings, techniques to obtain such norms a priori, or even during the optimization process itself, represent one of the most important points for further research.

\begin{figure}
    \centering
\includegraphics[scale=0.042]{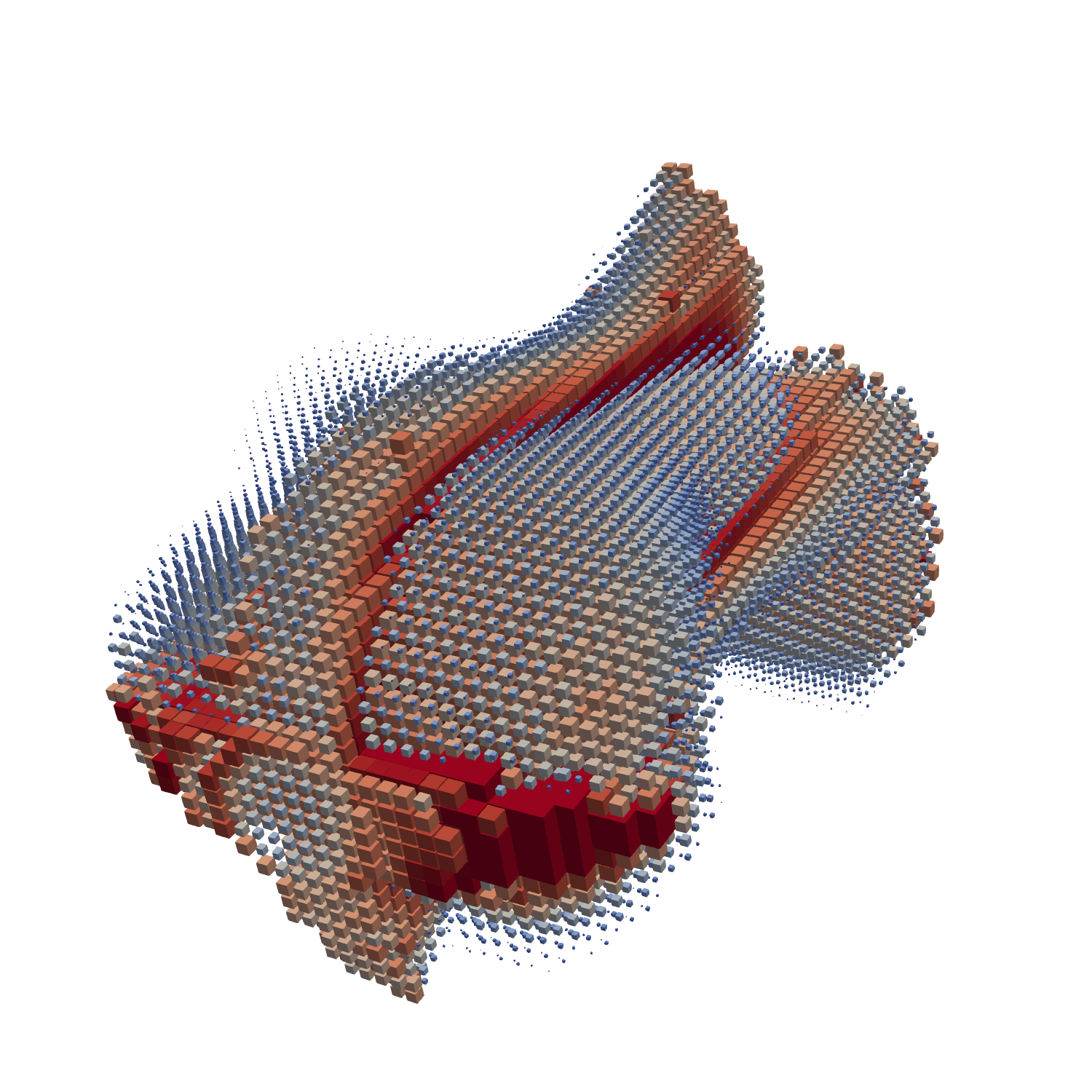}
\includegraphics[scale=0.042]{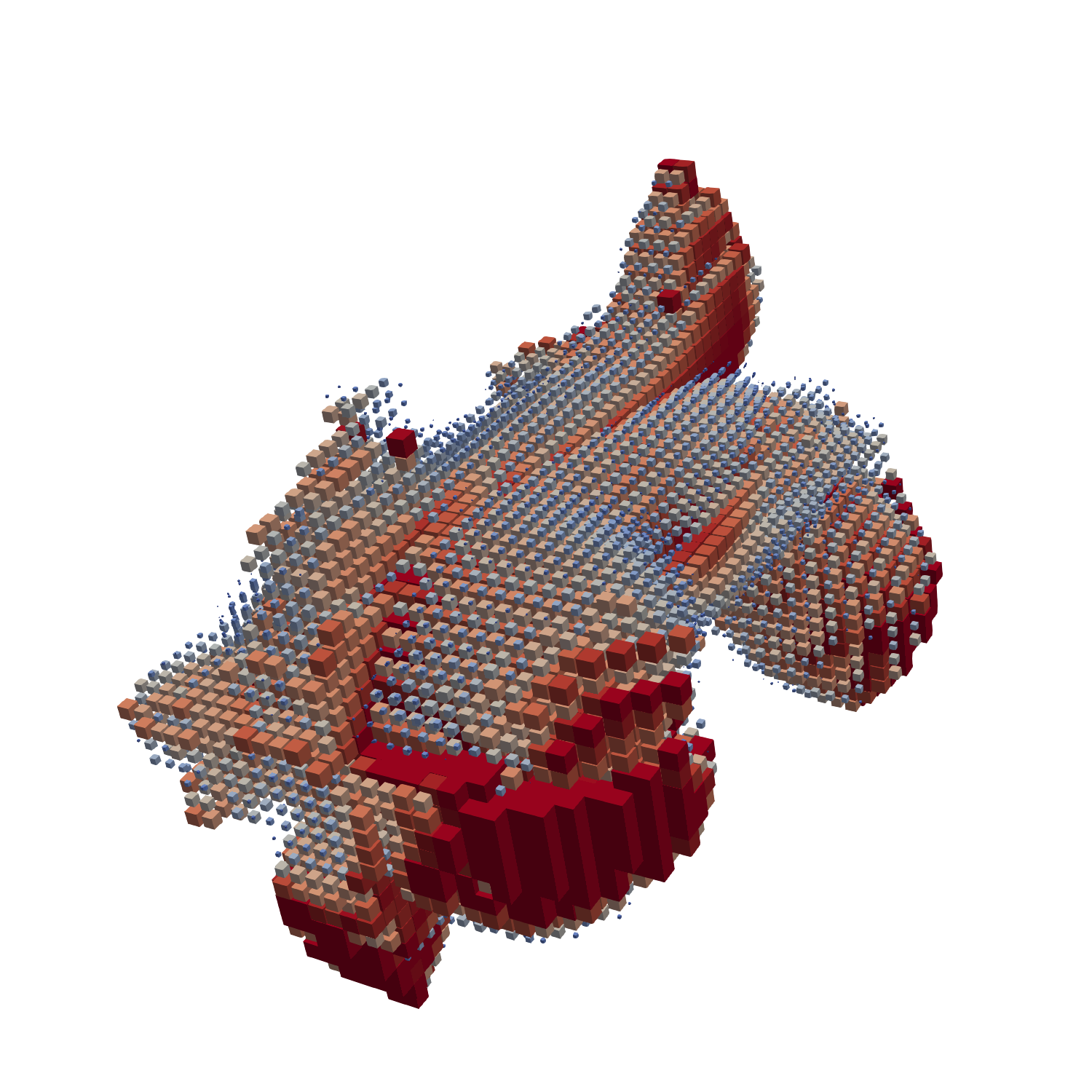}
\includegraphics[scale=0.042]{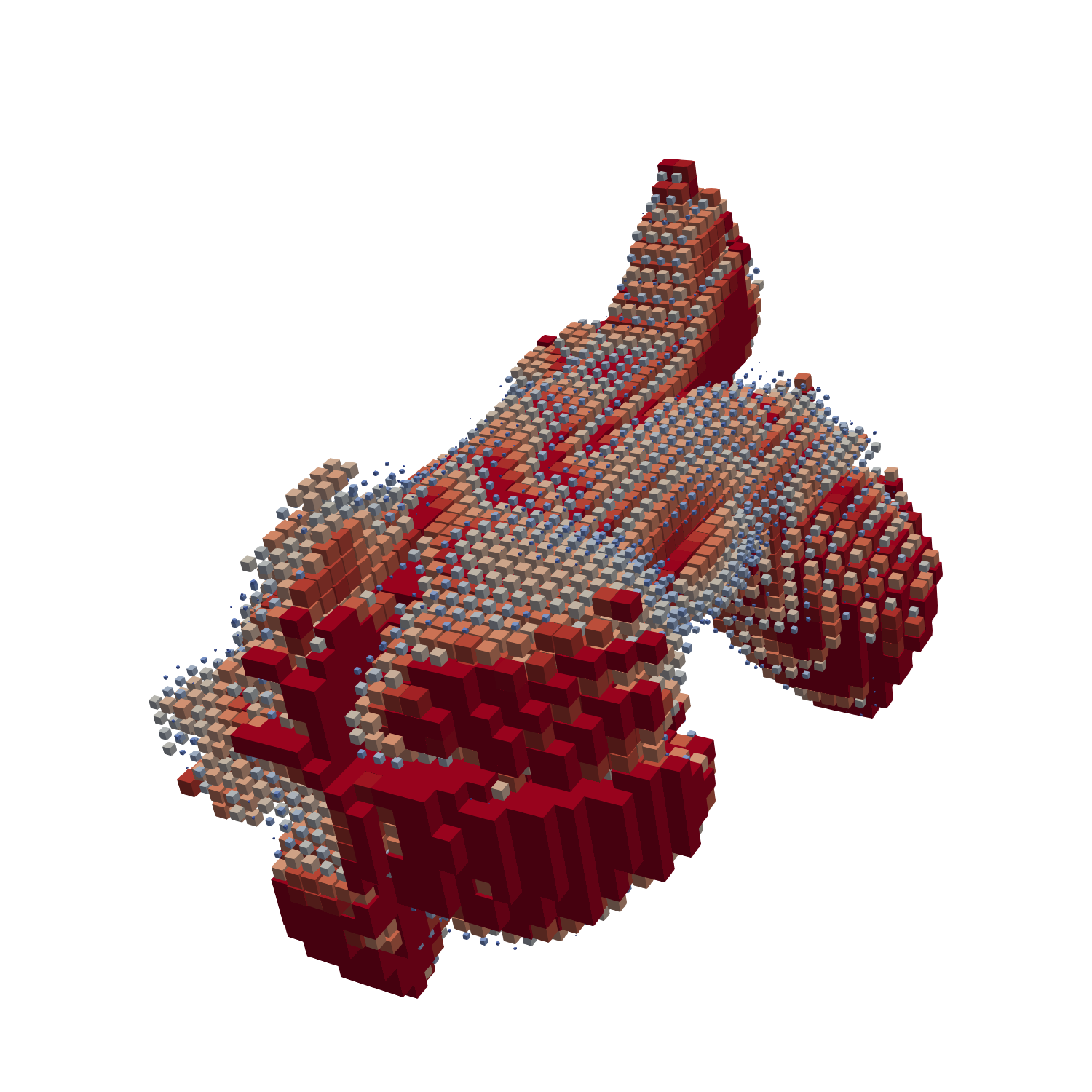}
\includegraphics[scale=0.042]{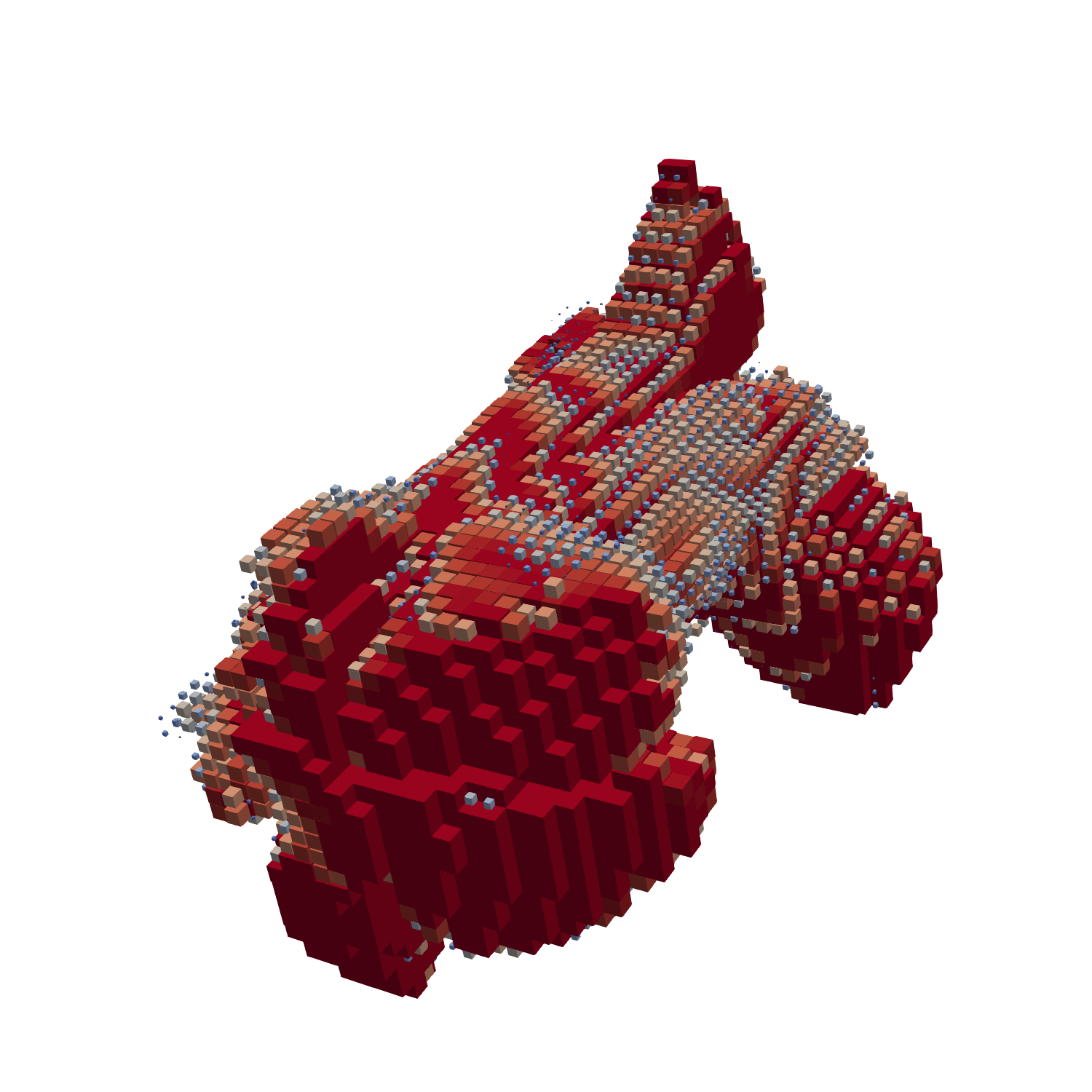}
\includegraphics[scale=0.042]{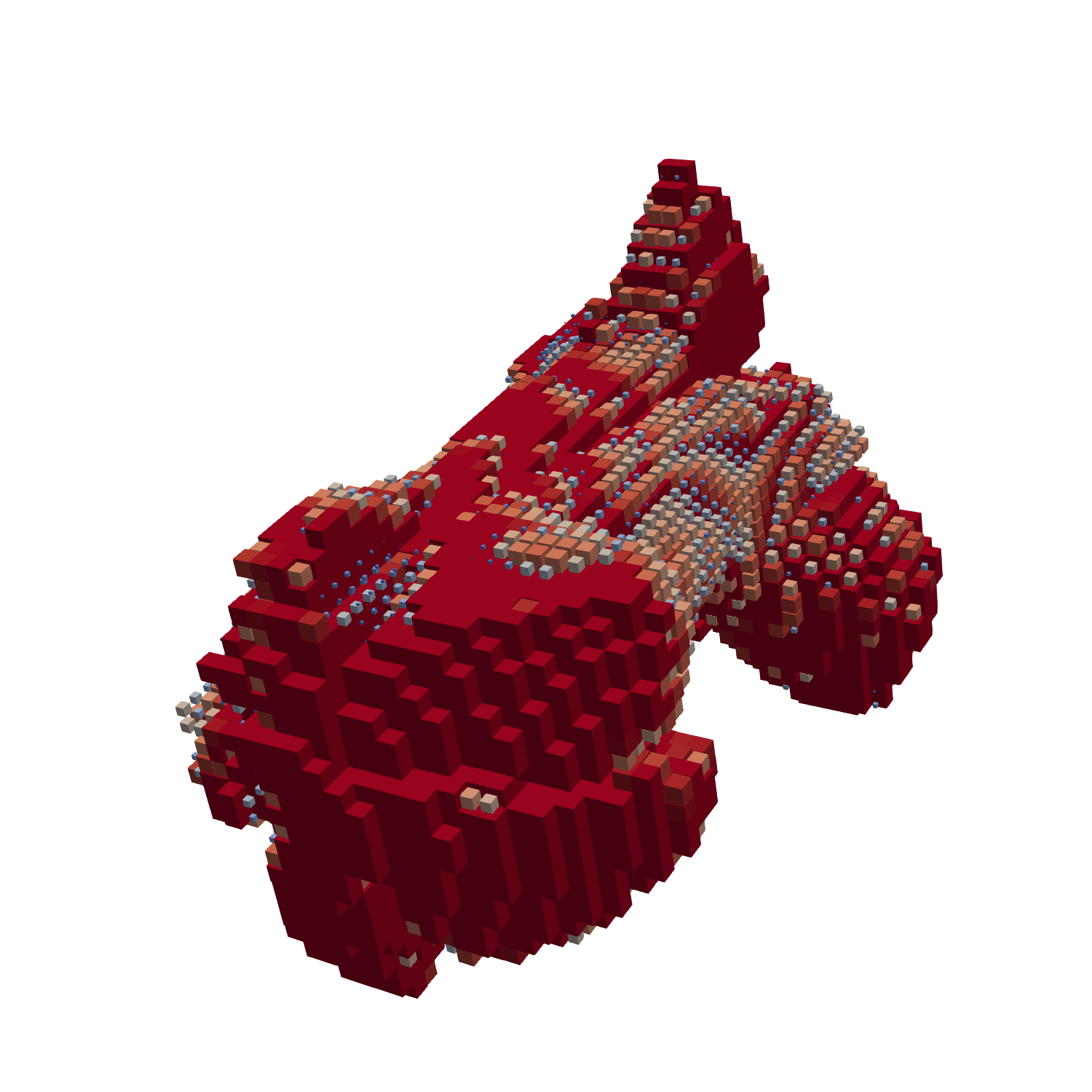}
                 
\includegraphics[scale=0.042]{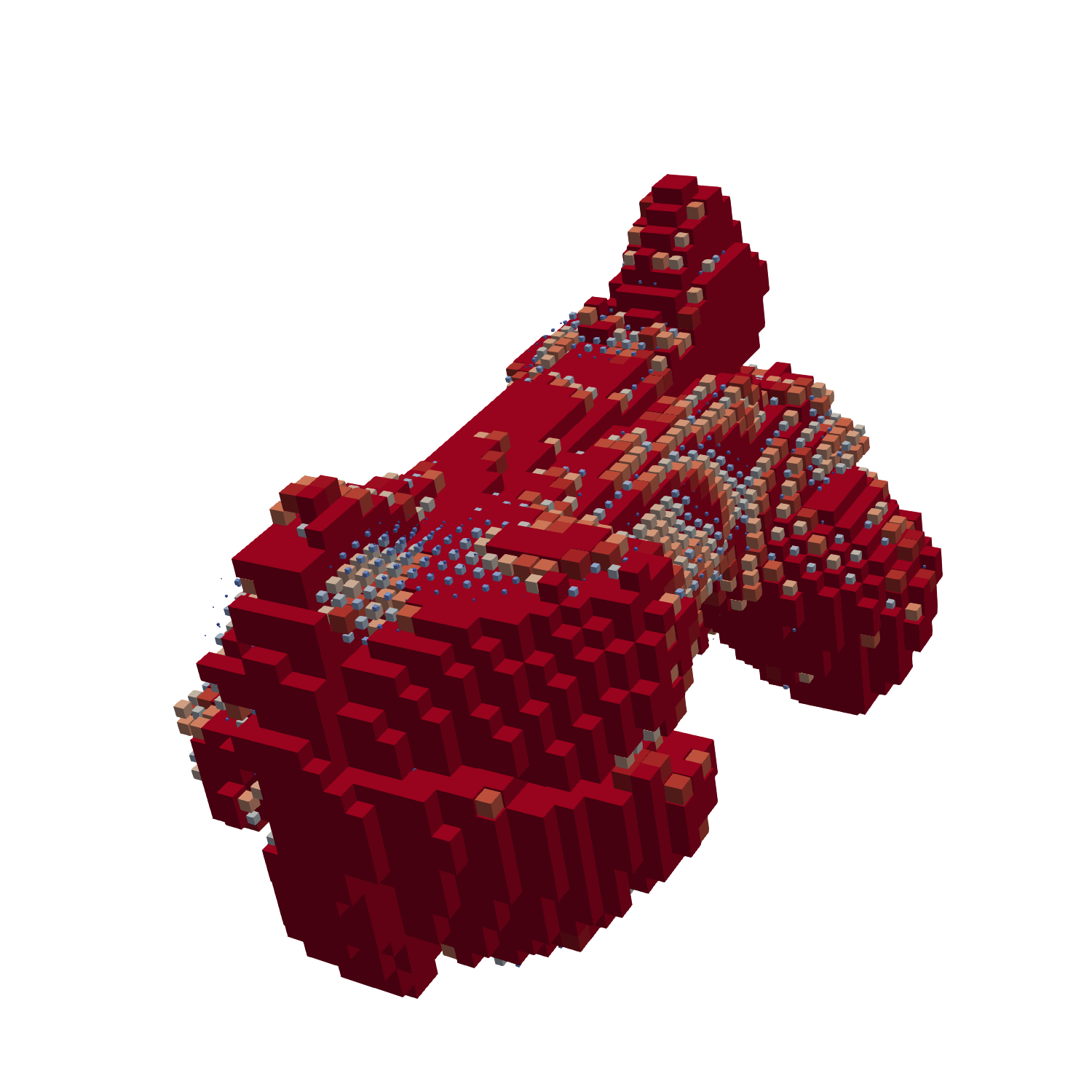}
\includegraphics[scale=0.042]{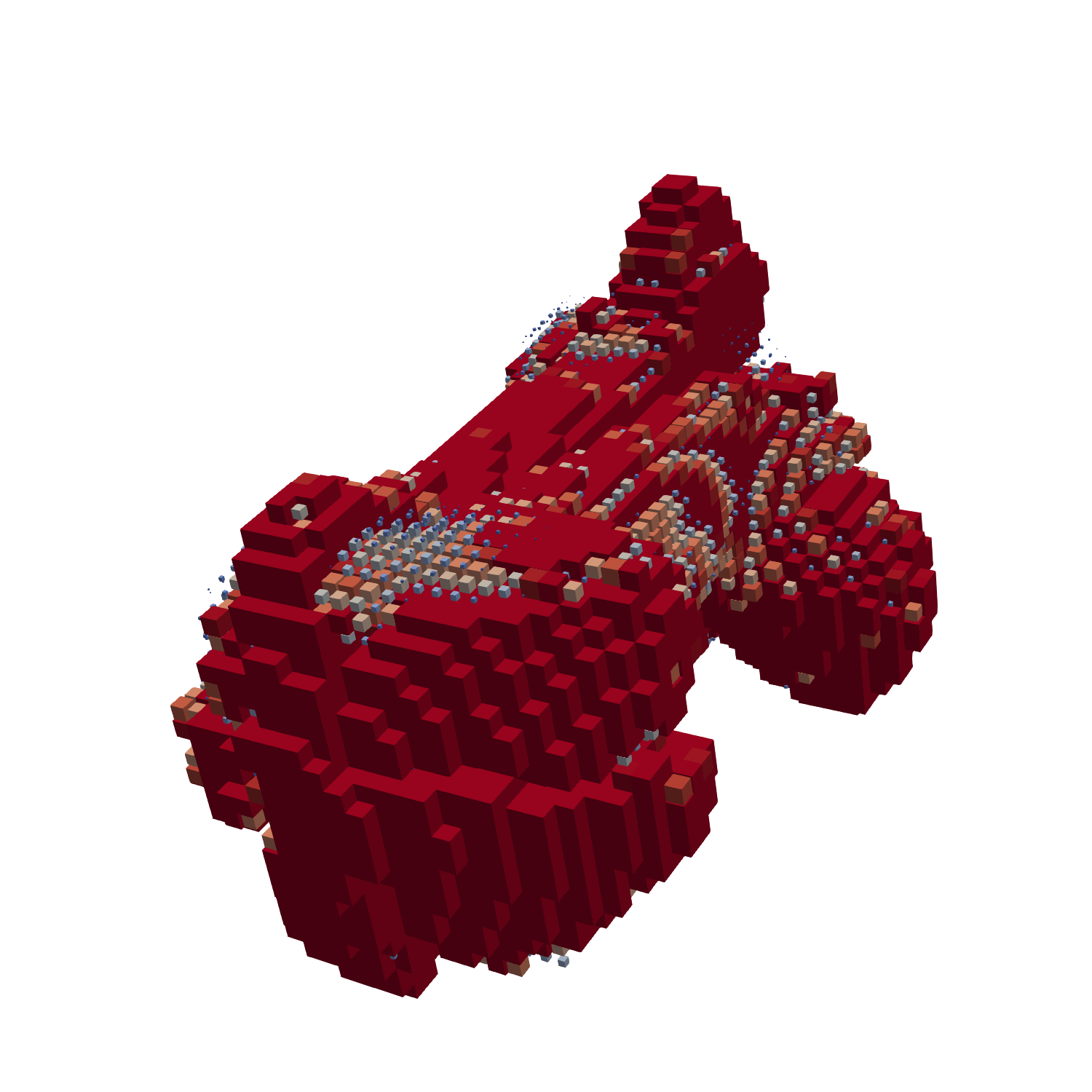}
\includegraphics[scale=0.042]{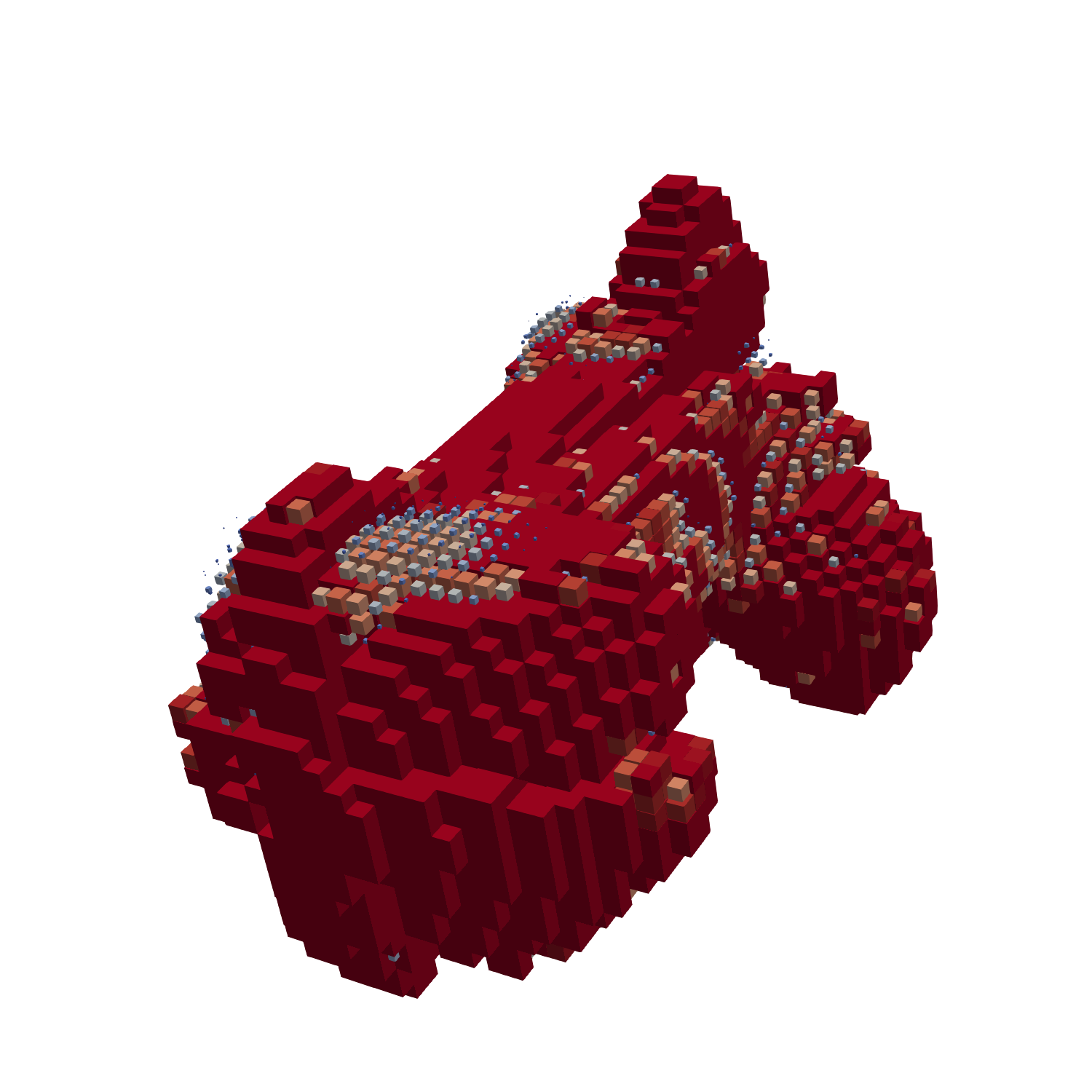}
\includegraphics[scale=0.042]{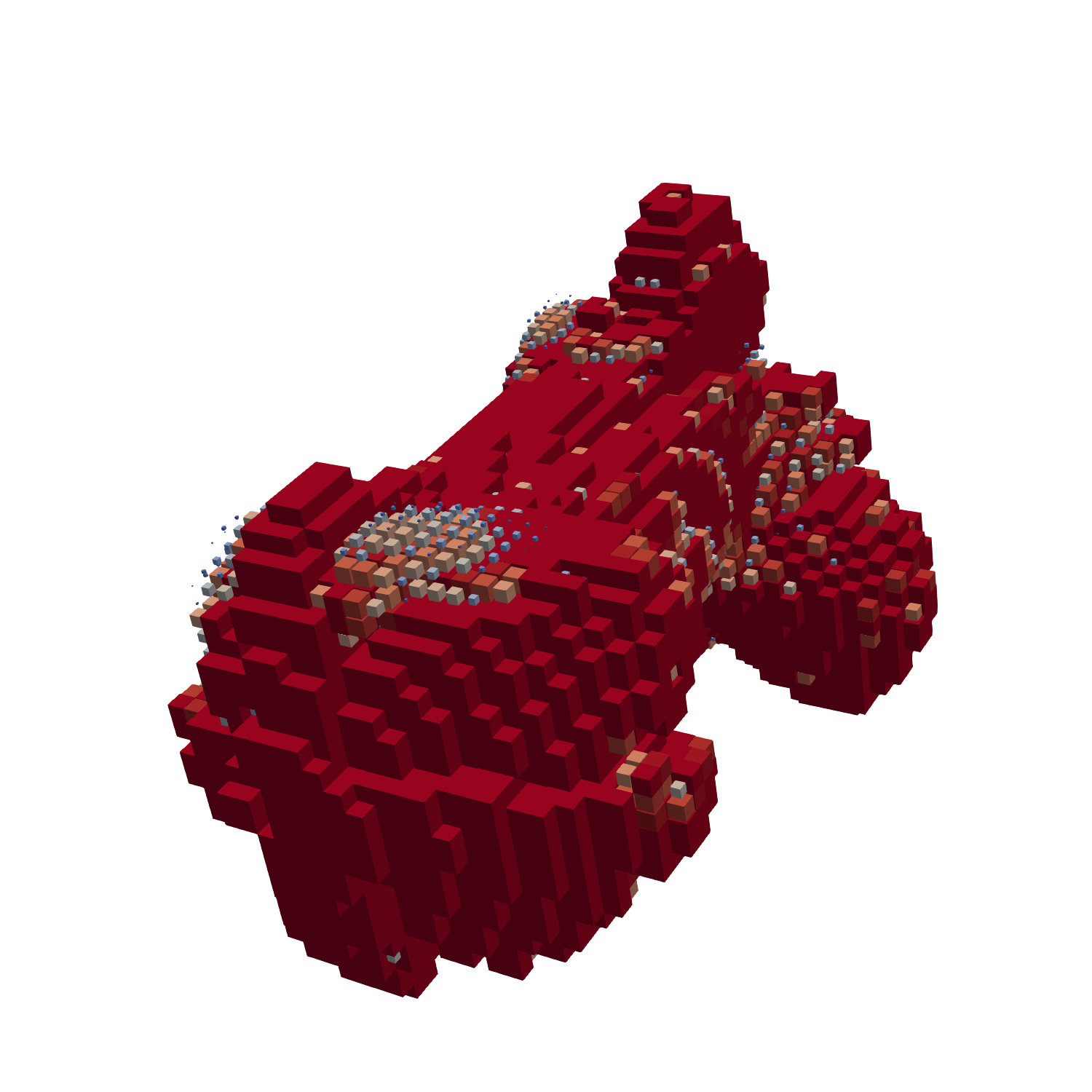}
\includegraphics[scale=0.042]{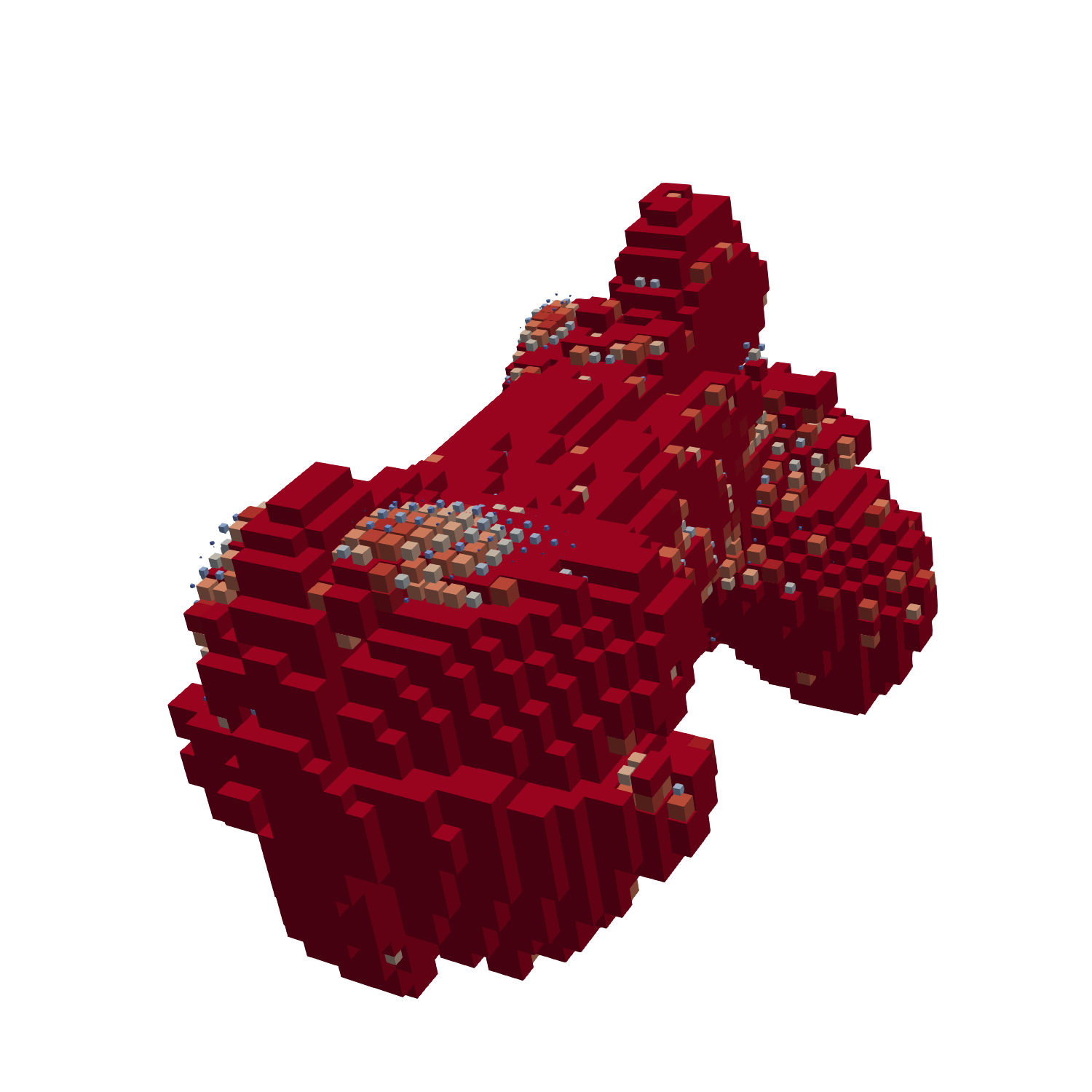}

\includegraphics[scale=0.042]{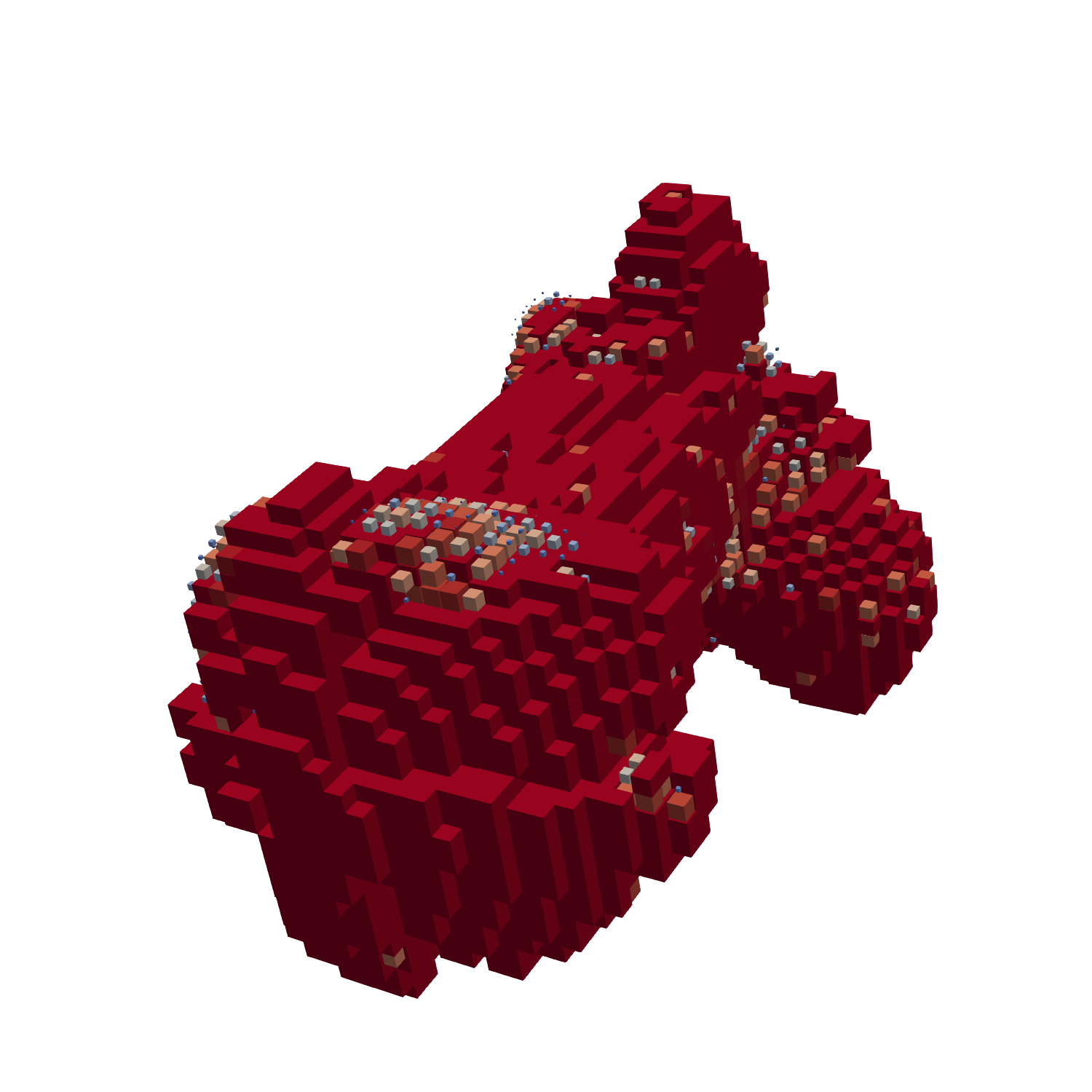}
\includegraphics[scale=0.042]{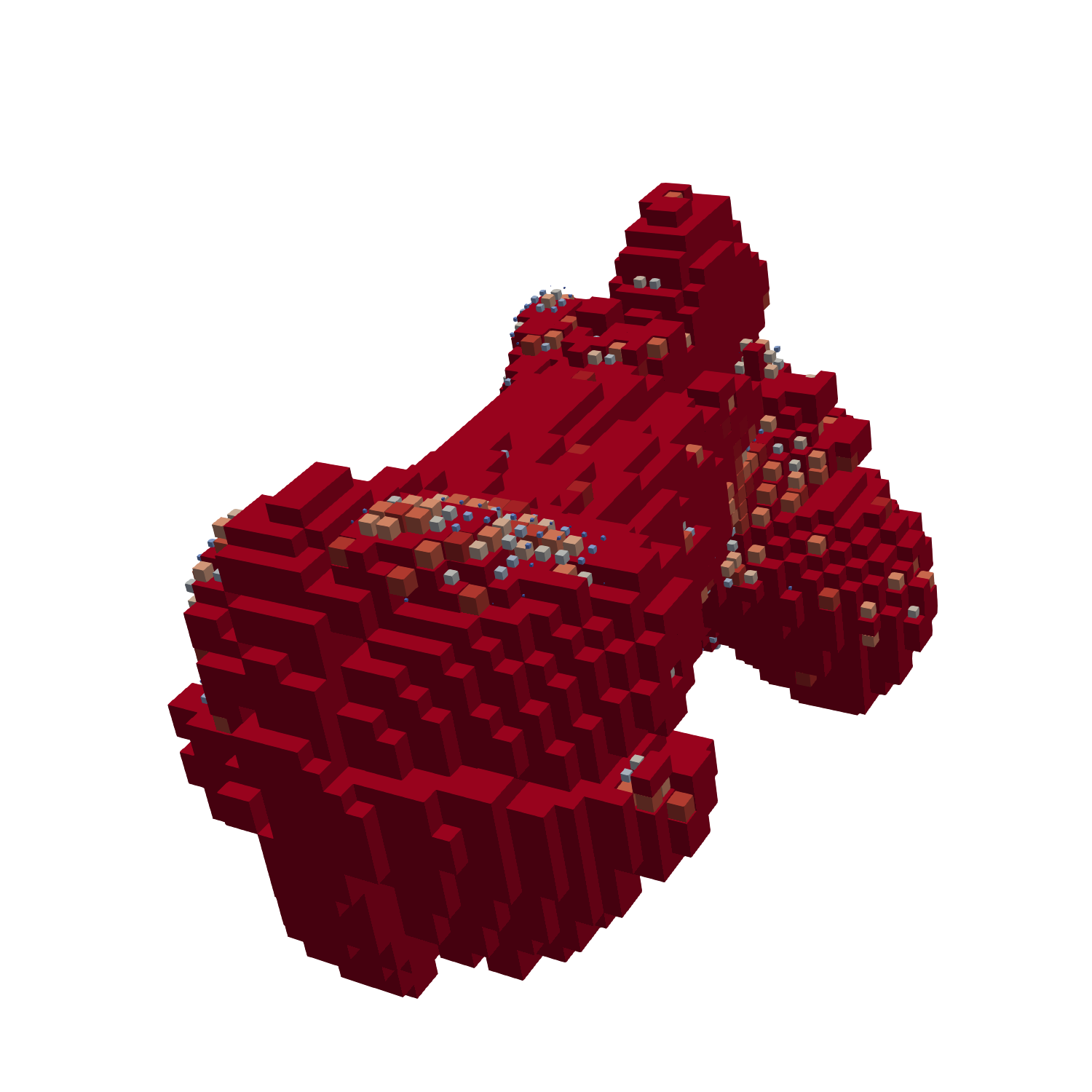}
\includegraphics[scale=0.042]{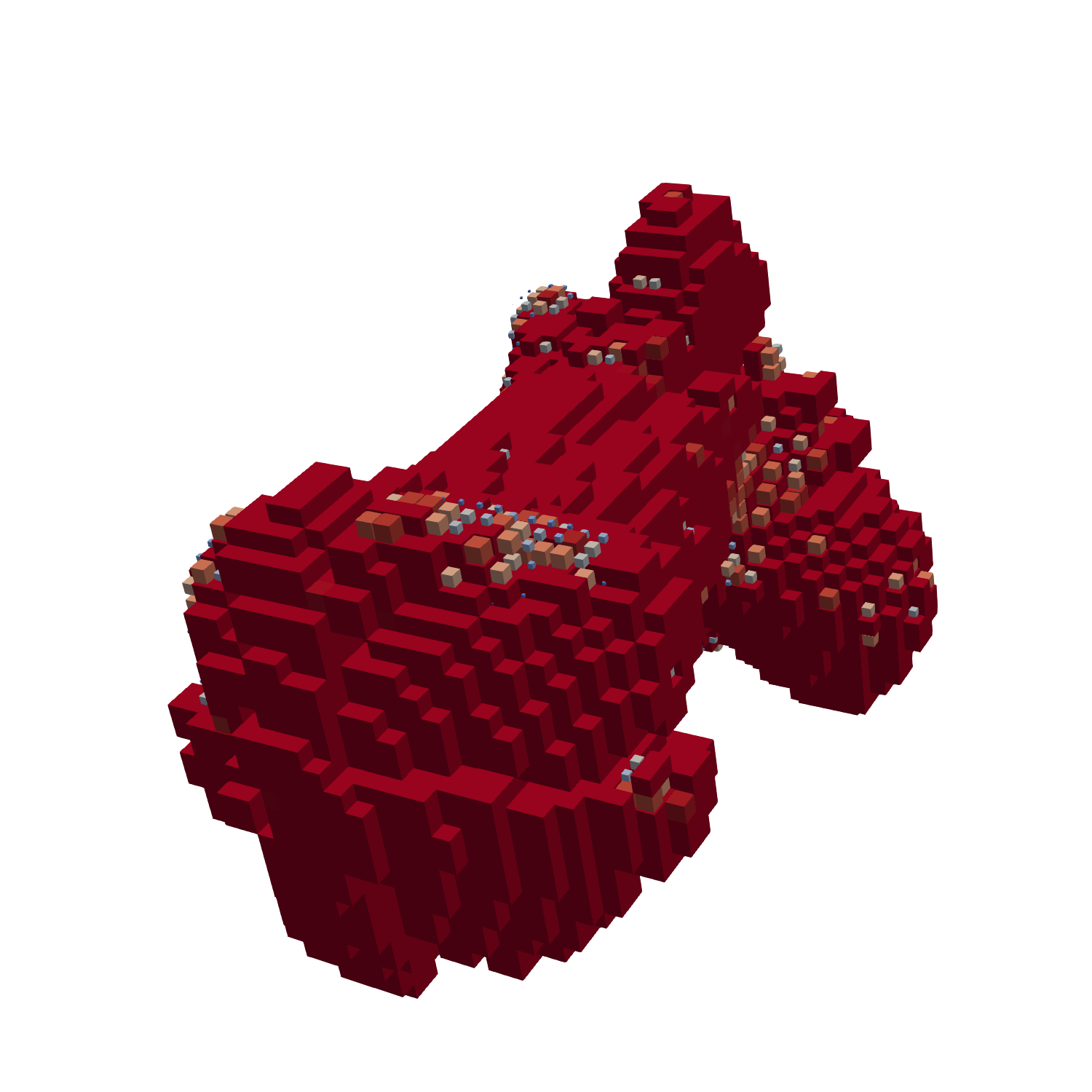}
\includegraphics[scale=0.042]{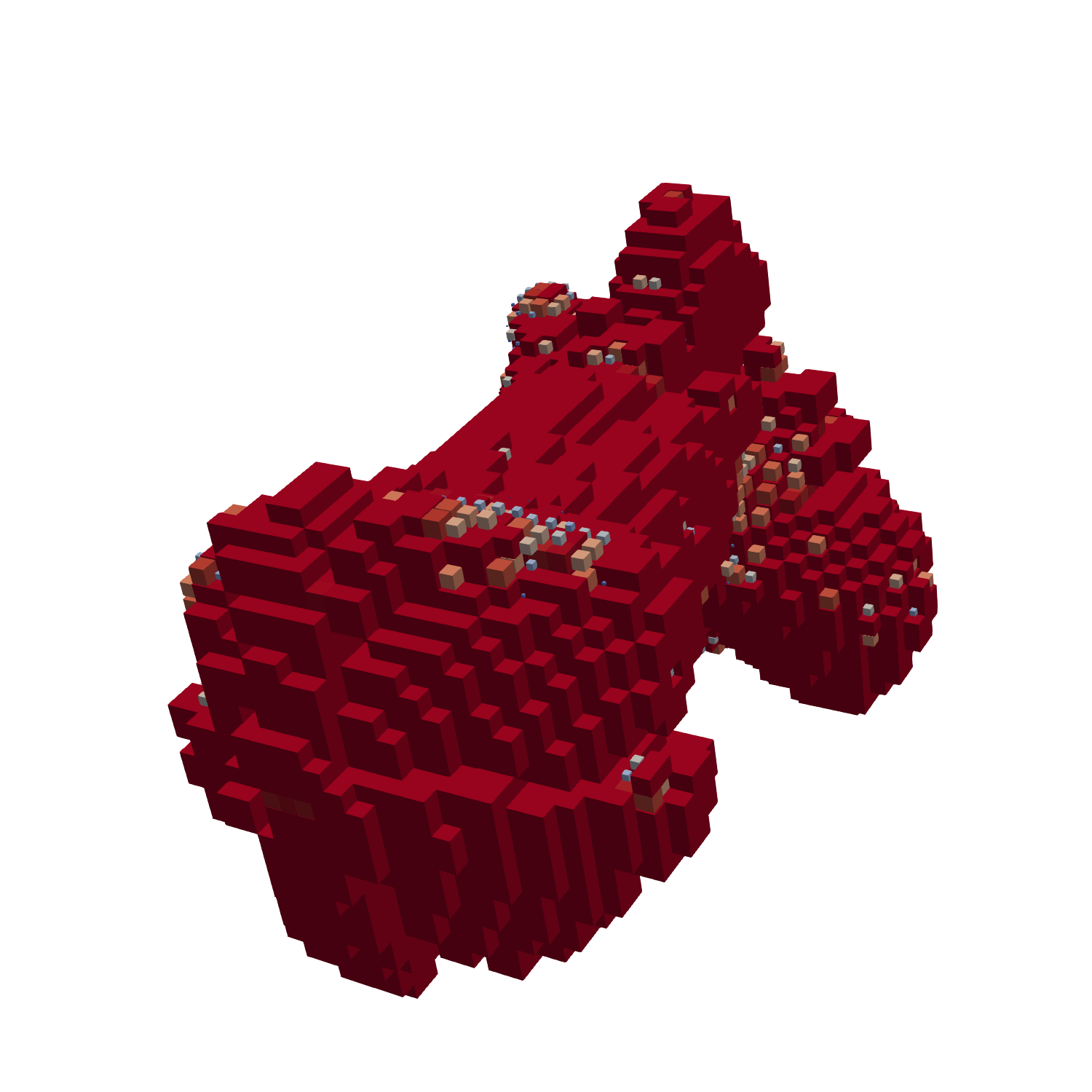}
\includegraphics[scale=0.042]{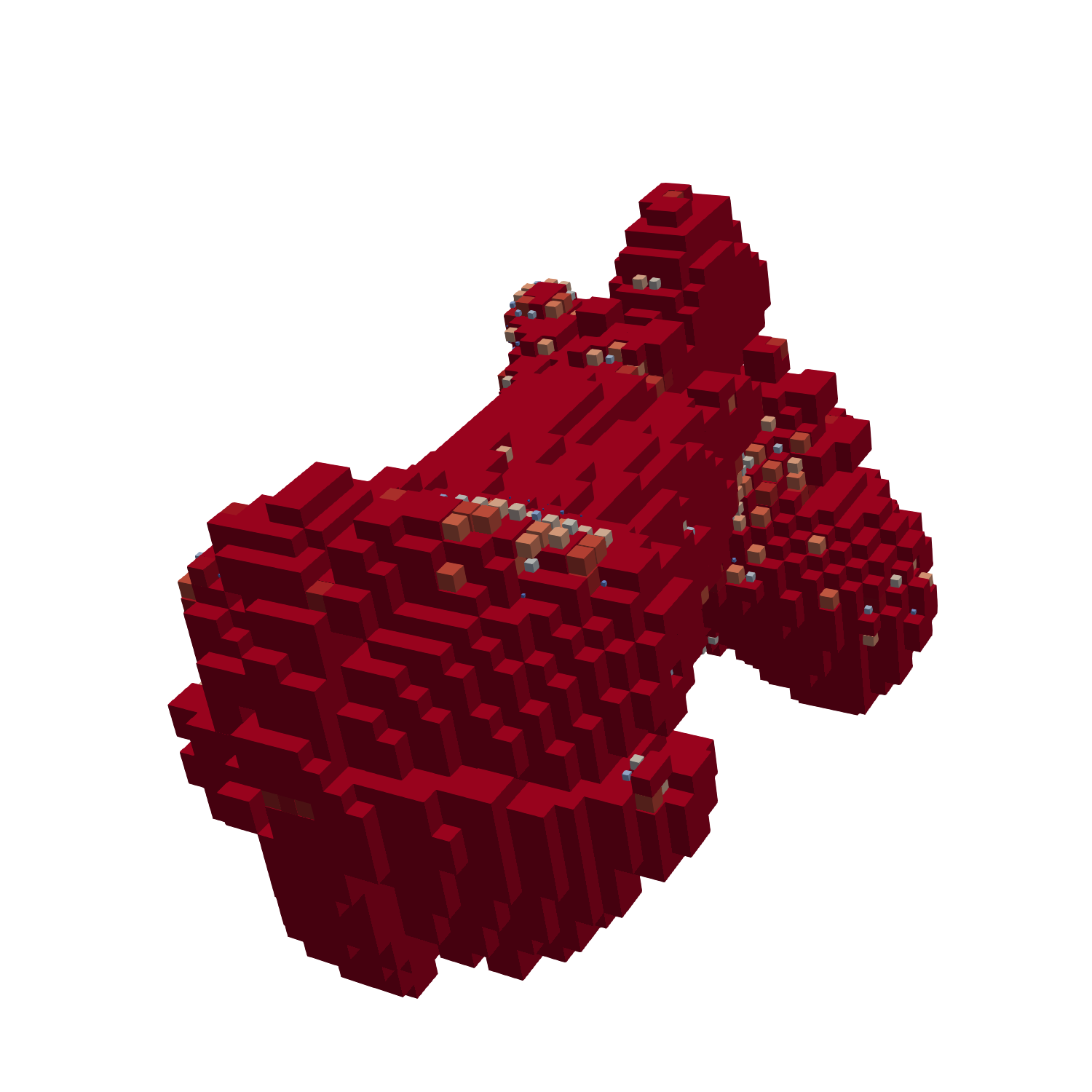}

\includegraphics[scale=0.042]{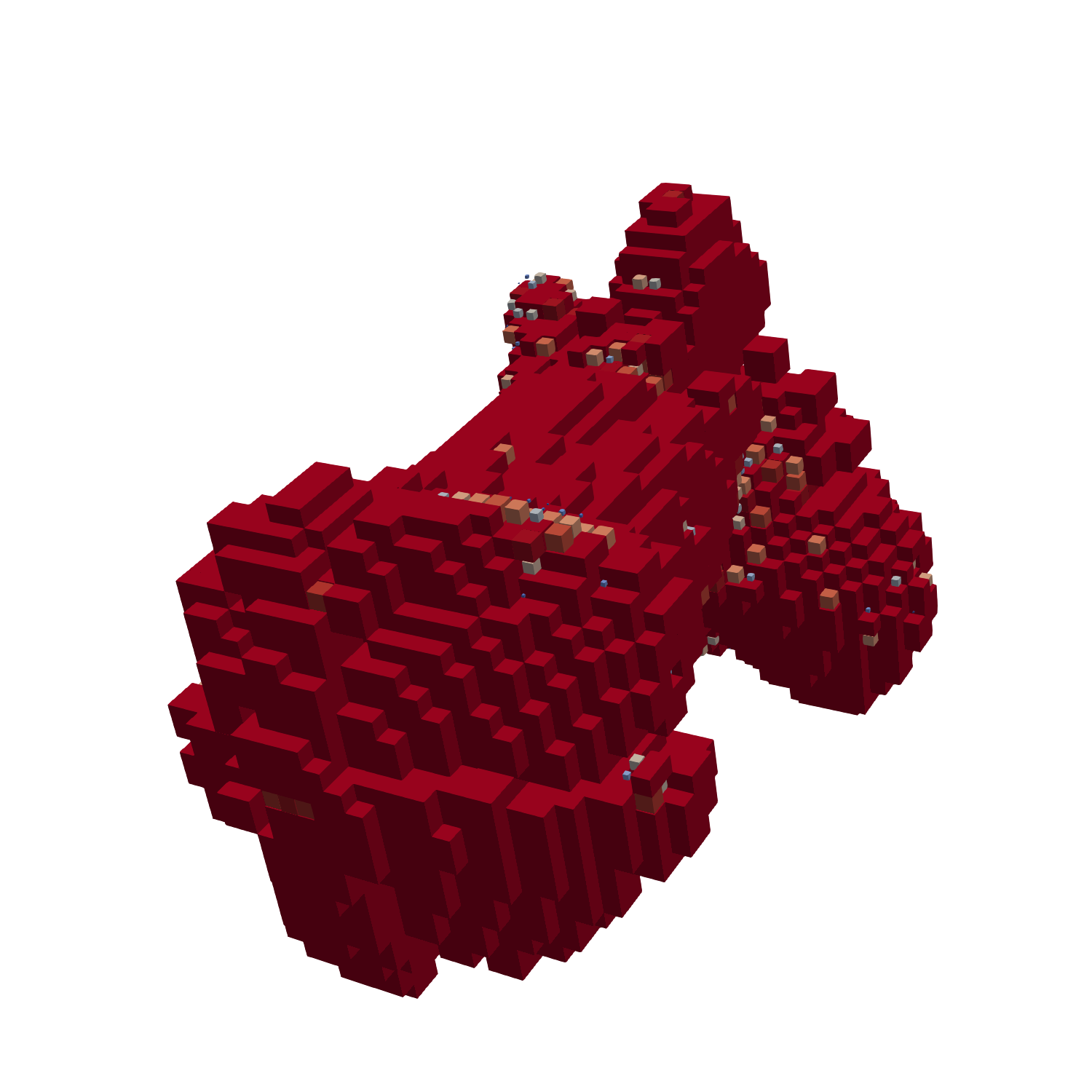}
\includegraphics[scale=0.042]{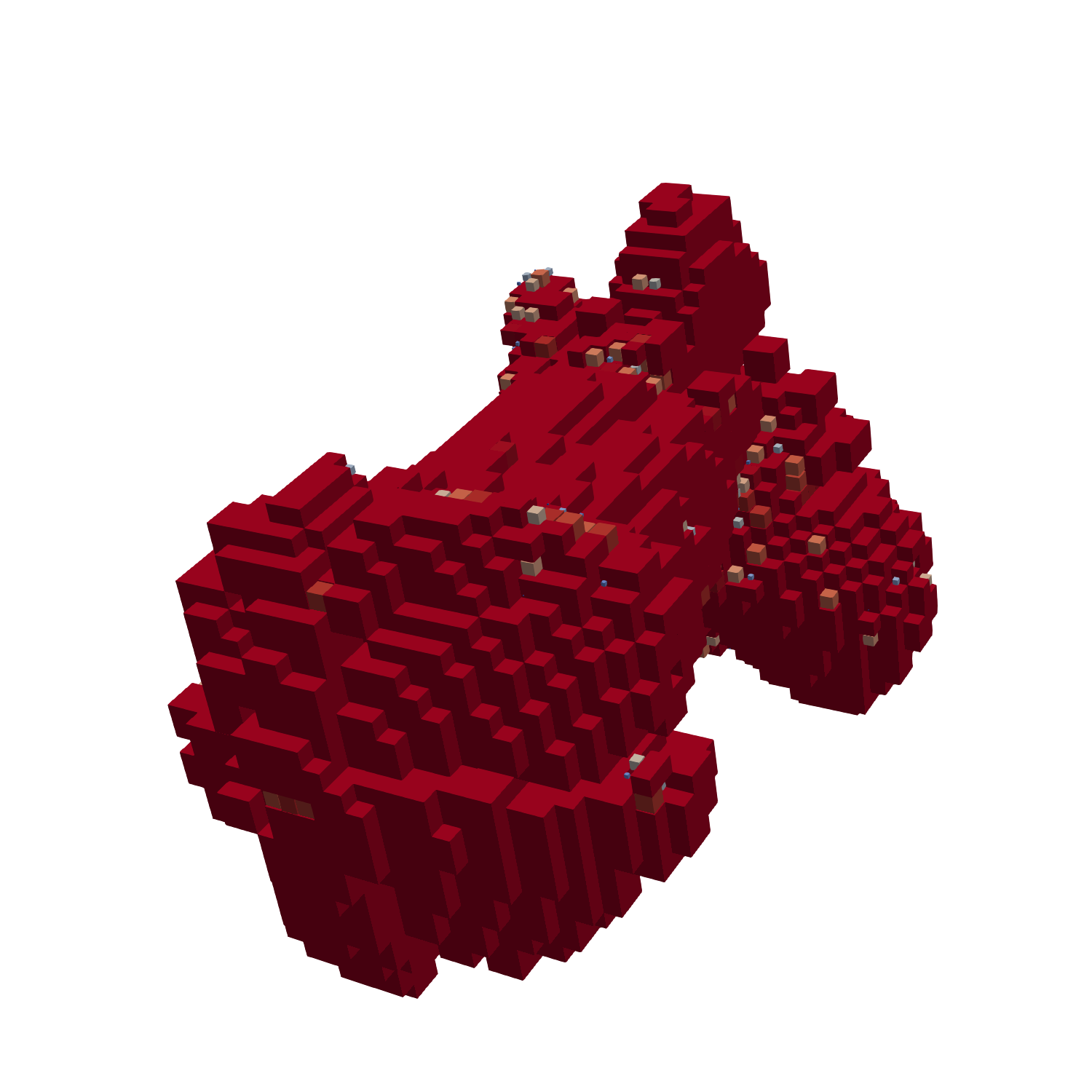}
\includegraphics[scale=0.042]{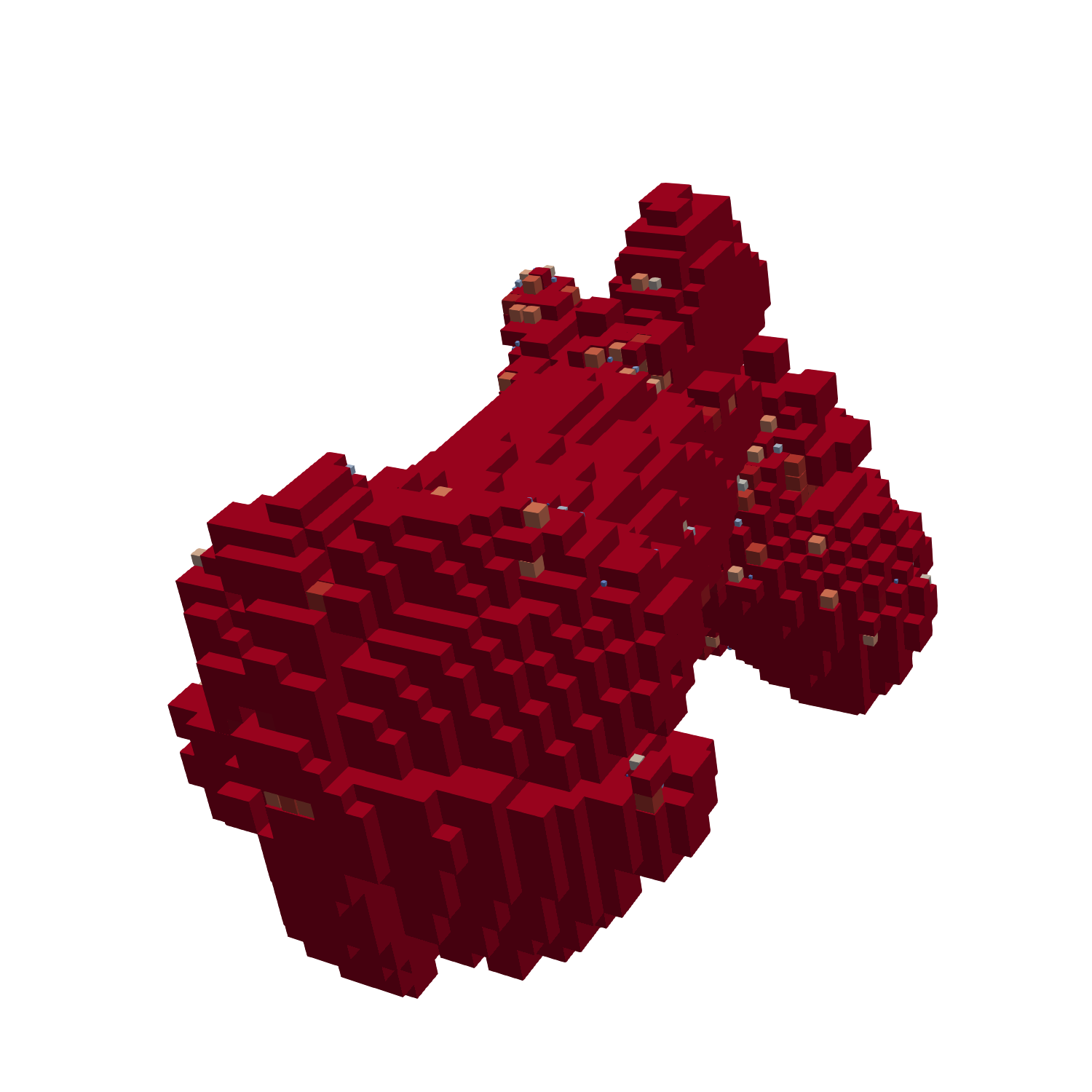}
\includegraphics[scale=0.042]{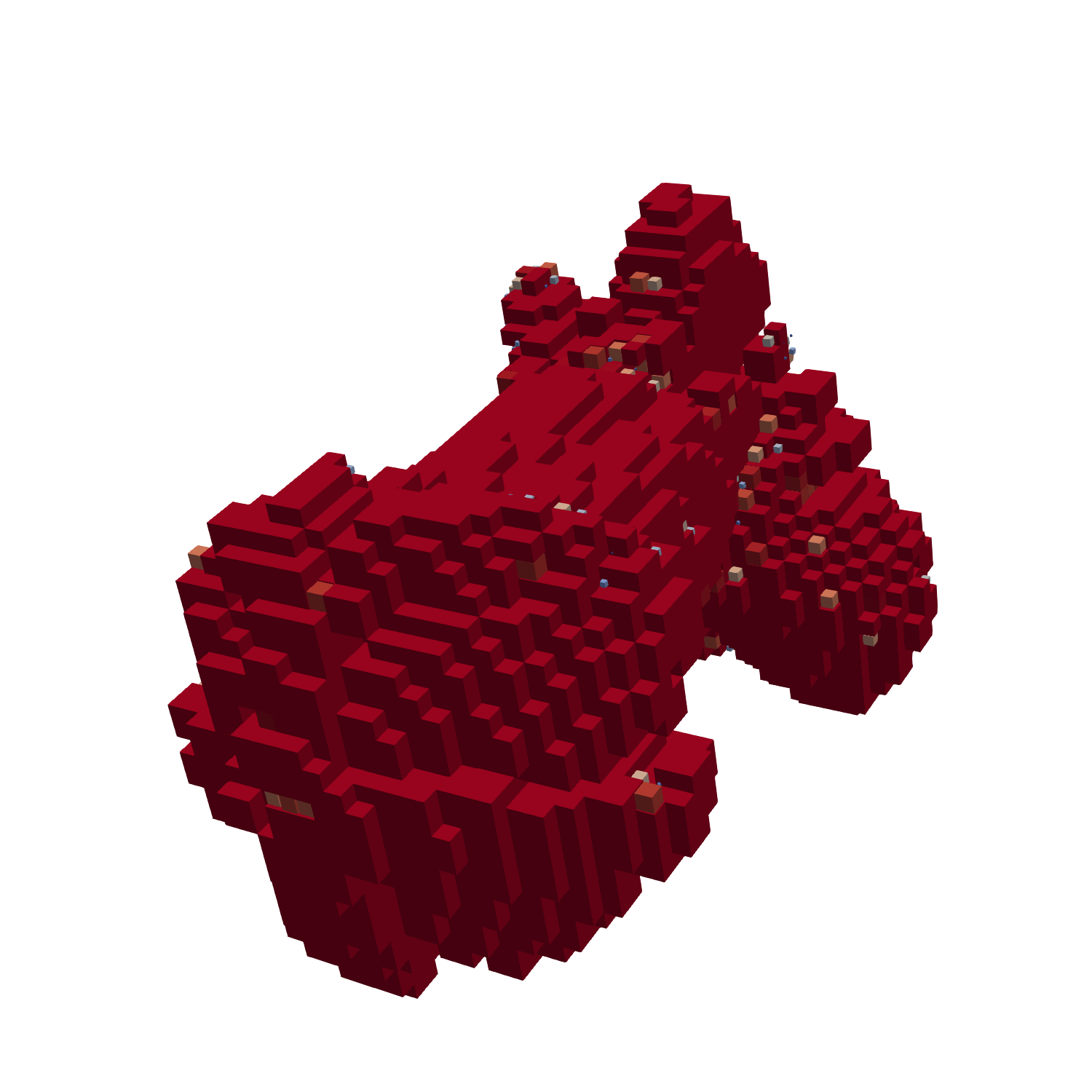}
\includegraphics[scale=0.042]{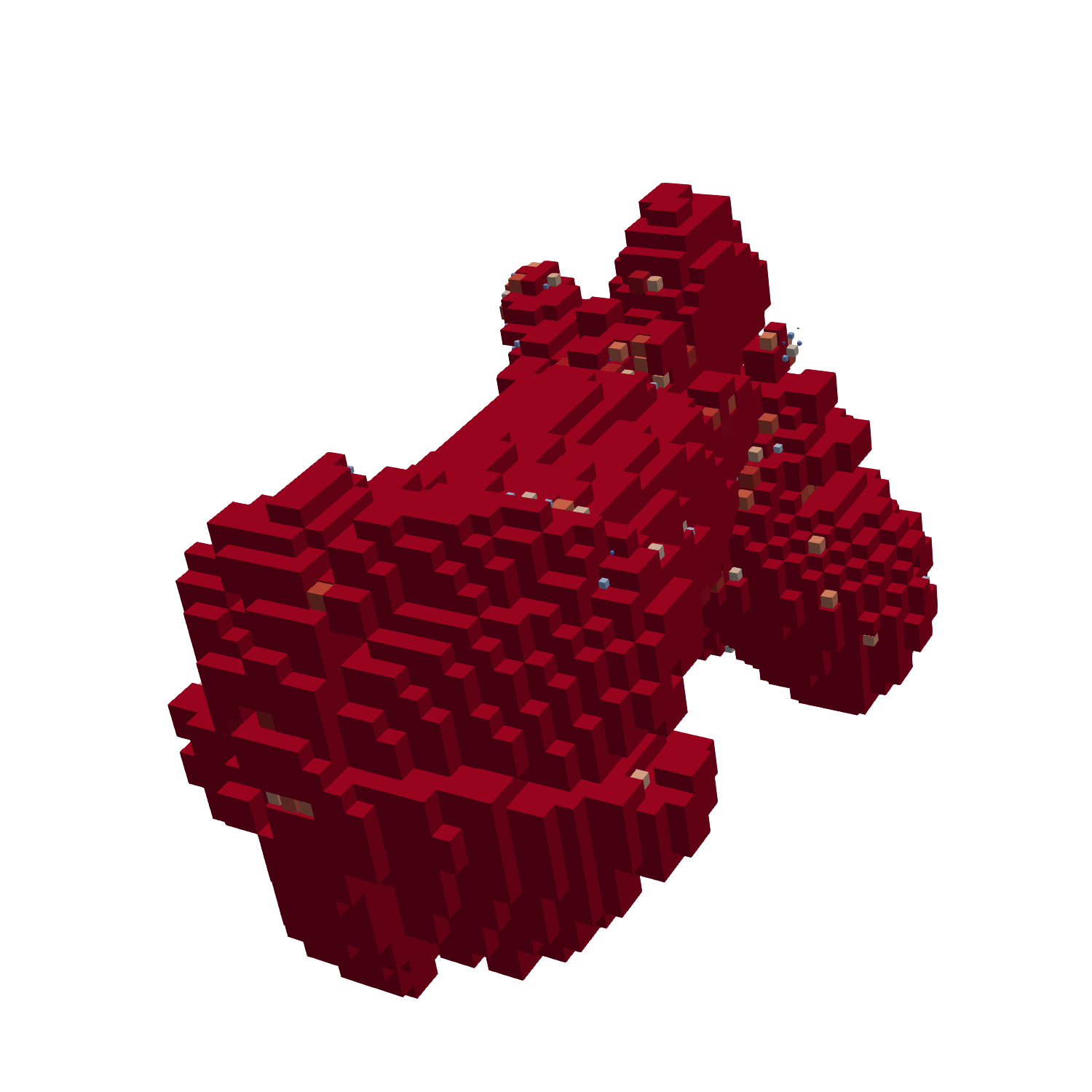}

\includegraphics[scale=0.042]{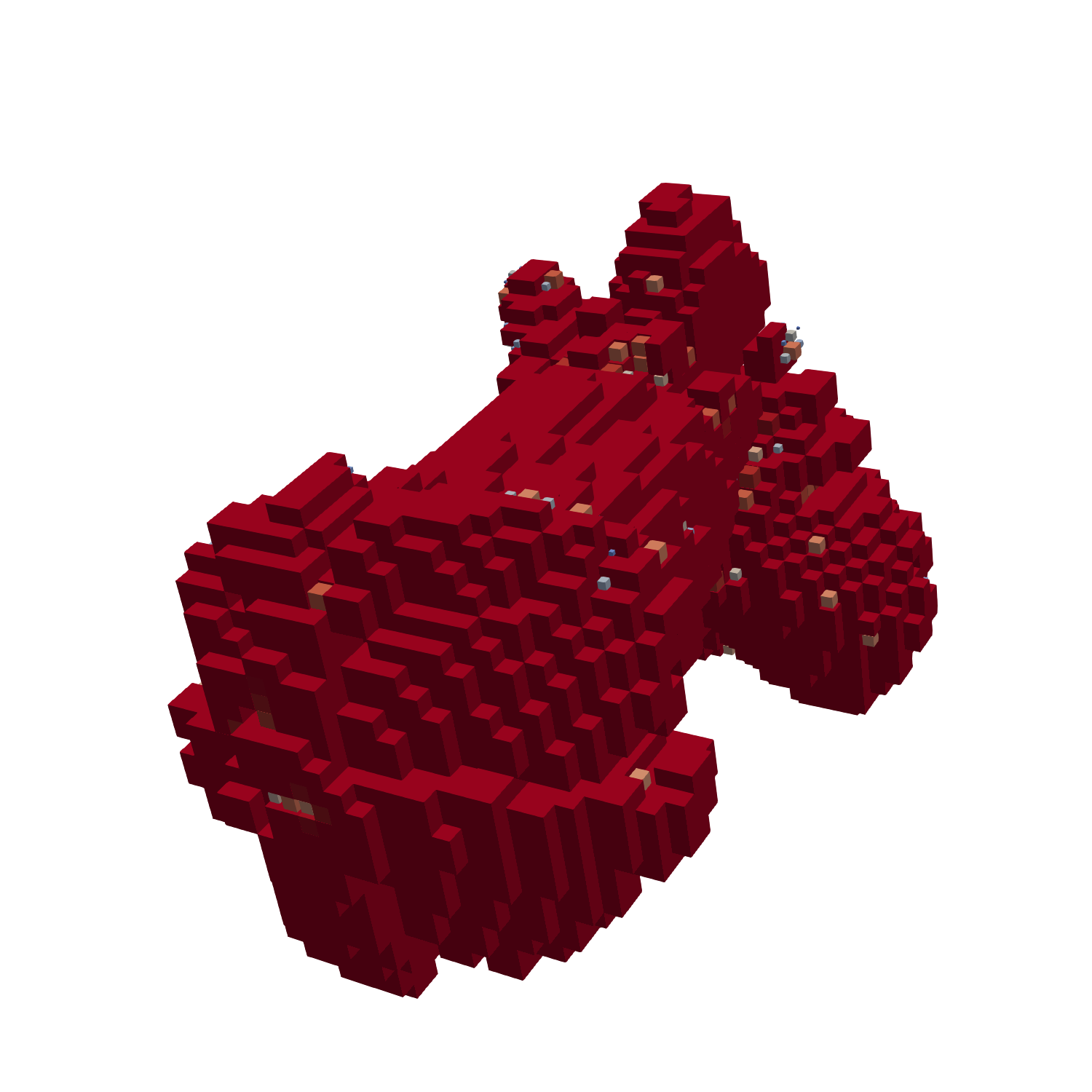}
\includegraphics[scale=0.042]{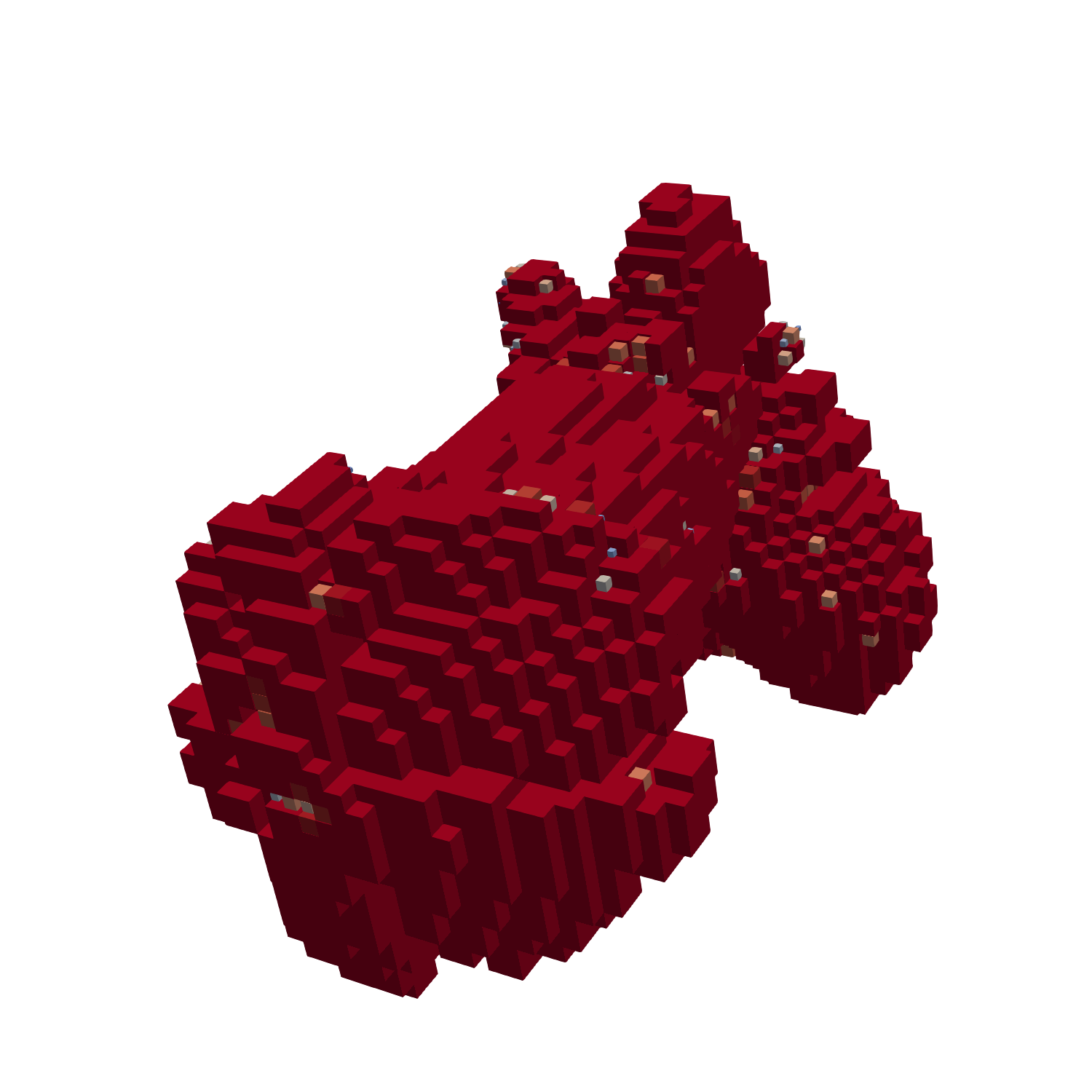}
\includegraphics[scale=0.042]{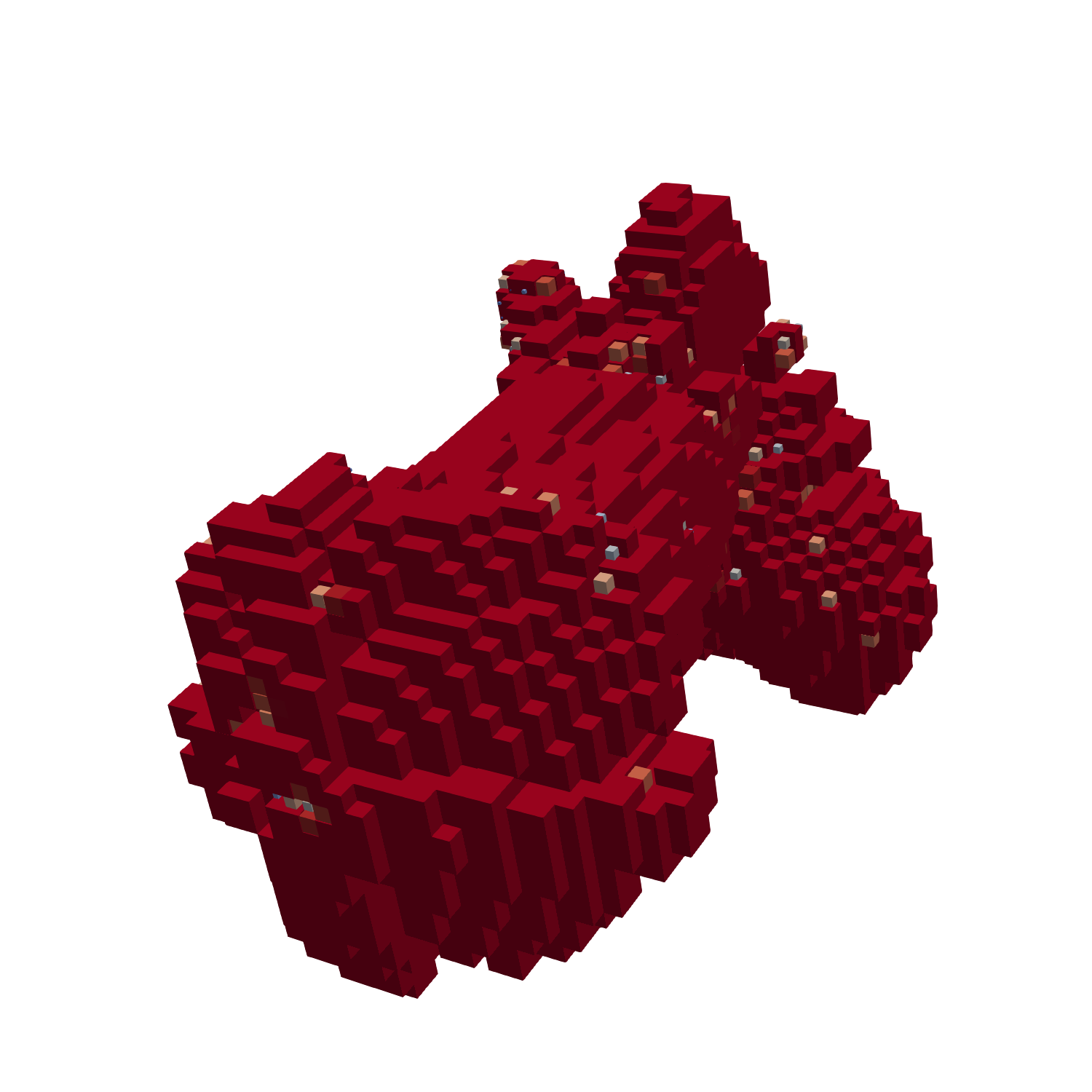}
\includegraphics[scale=0.042]{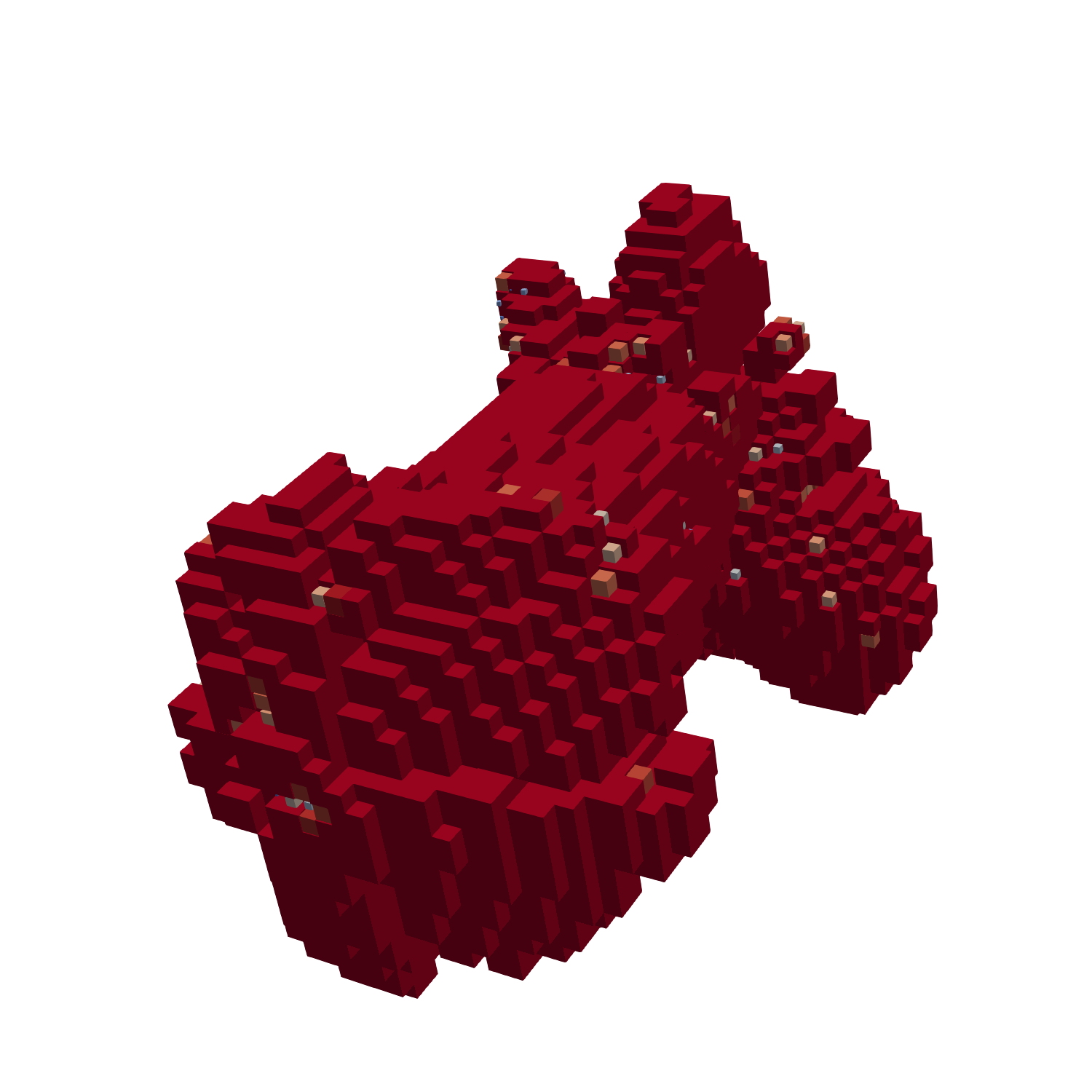}
\includegraphics[scale=0.042]{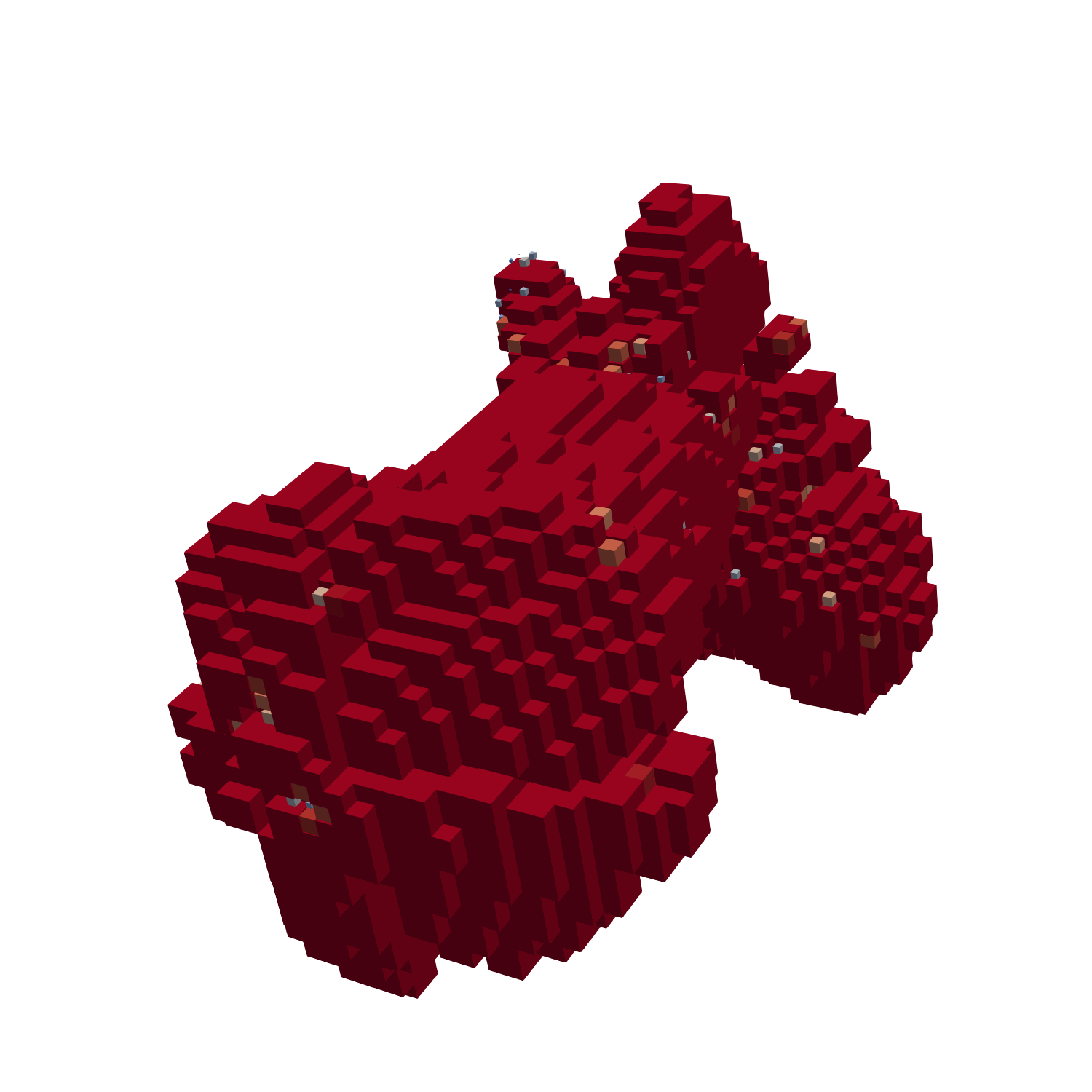}
    \caption{Top left to bottom right: Design evolution during the optimization process for the \textit{screwdriver (50\%)} initial design and outer norm (a). The design snapshots were taken every 200 iterations. Red boxes represent design cells consisting of pure hematite. Intermediate material is indicated via a color gradient, where a cell filled with $50\%$ water and $50\%$ hematite is colored grey. Based on this gradient, depending on the ratio of hematite and water in a cell, the cell color is shifted to red (more hematite) or blue (more water).}
    \label{fig:DDA_designs_ScHa}
\end{figure}
\begin{figure}
  \begin{minipage}[c]{0.5\textwidth}
    \includegraphics[width=\textwidth]{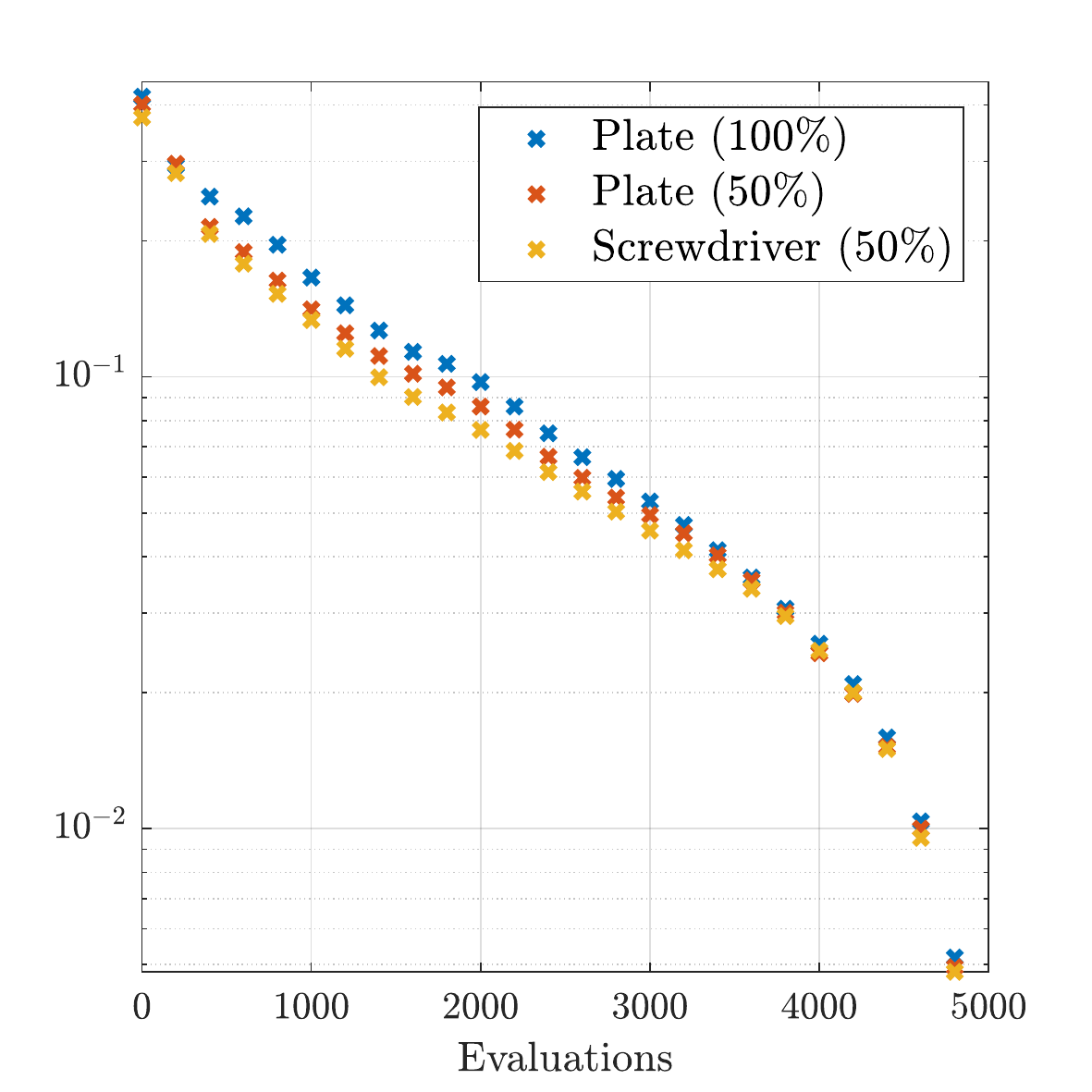}
  \end{minipage}\hfill
  \begin{minipage}[c]{0.45\textwidth}
    \caption{Euclidean distance (after dividing by $\sqrt{\dim(\U)}$ for scaling) between intermediate designs and the respective final design during the SCIBL-CSG optimization process, carried out with outer norm (a).}
    \label{fig:DesingChangeNorm}
  \end{minipage}
\end{figure}
\begin{figure} 
    \centering
    \begin{minipage}[b][][c]{0.48\textwidth}
        \centering
        \includegraphics[width = \textwidth]{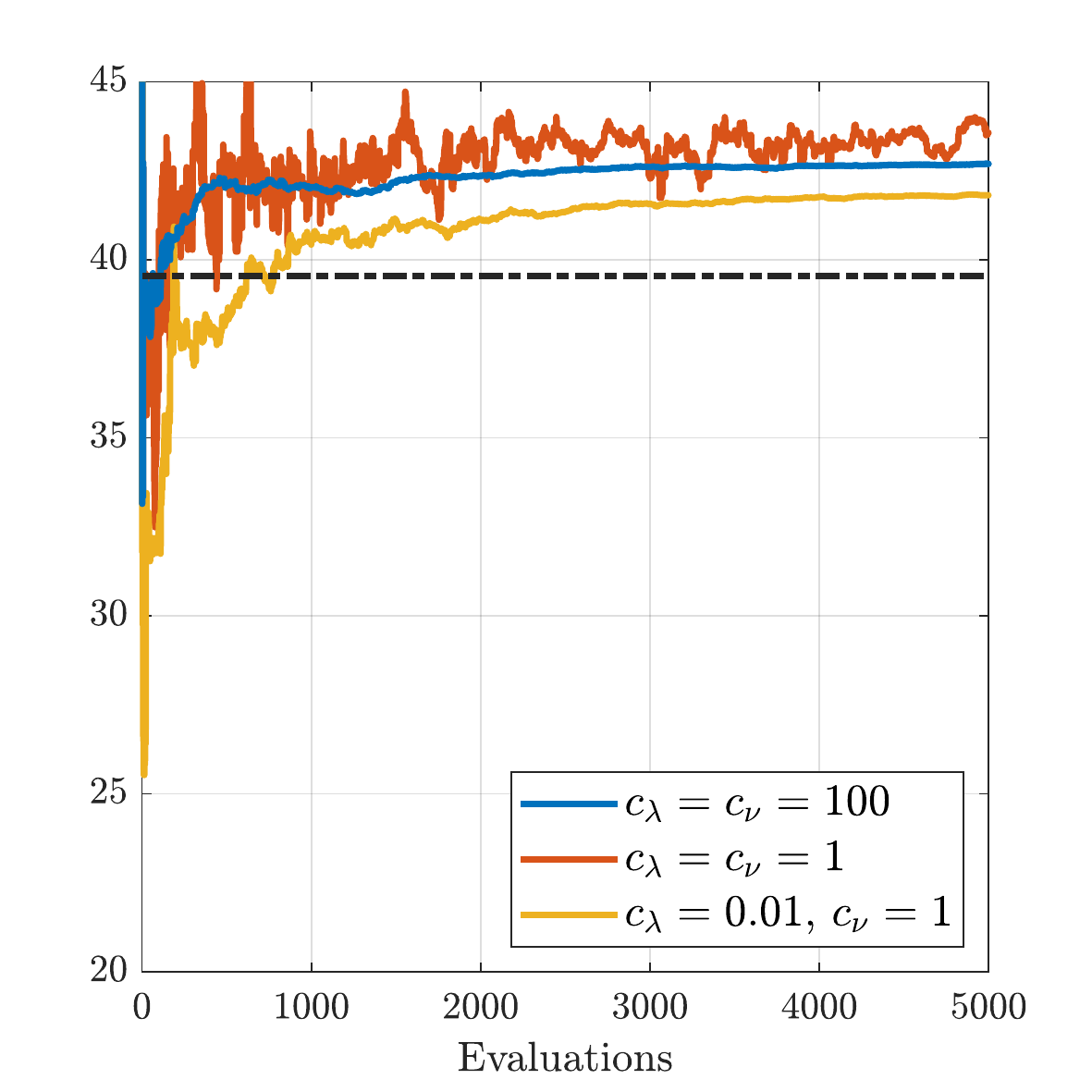}
    \end{minipage}\hfill
    \begin{minipage}[b][][c]{0.48\textwidth}
        \centering
        \includegraphics[width = \textwidth]{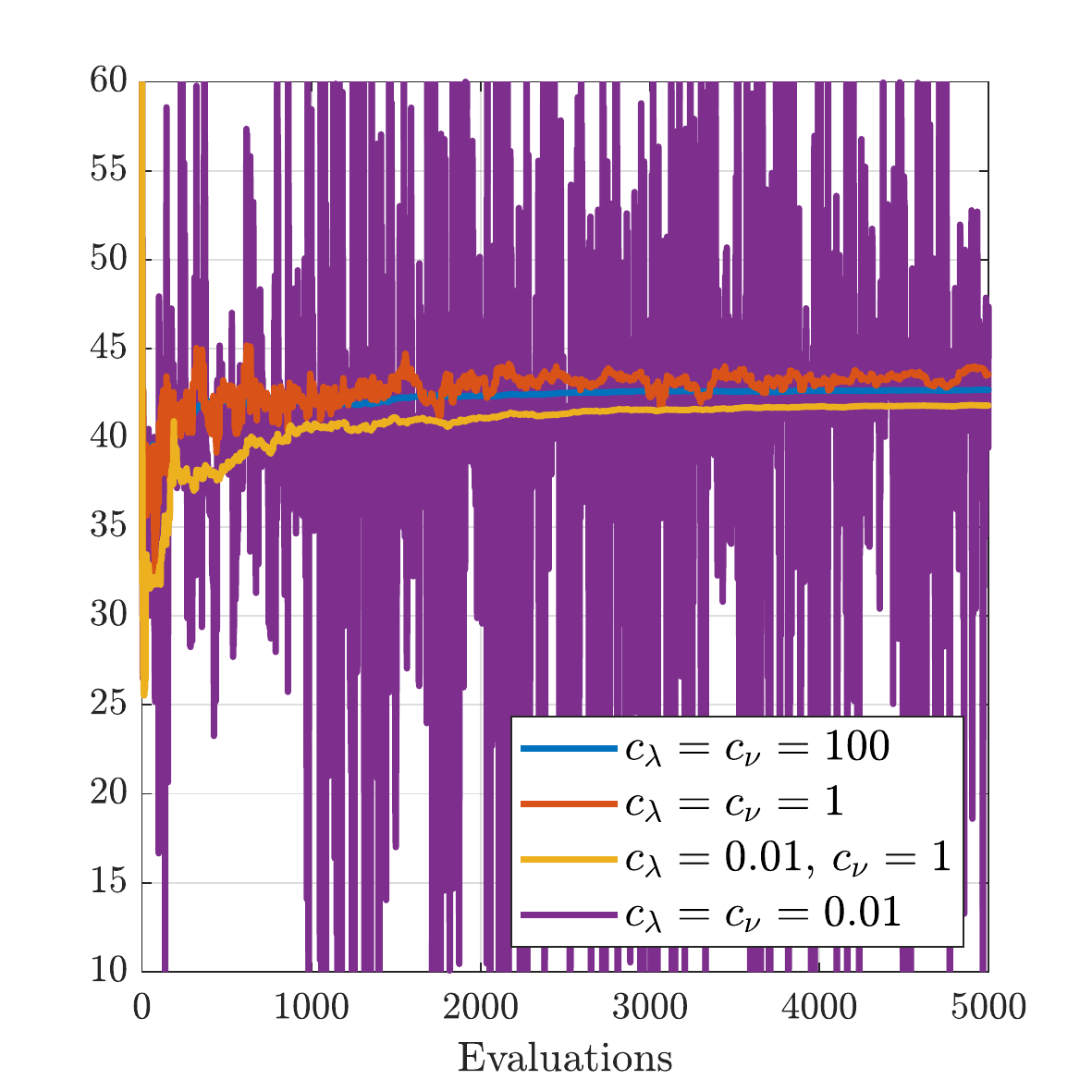}%
    \end{minipage}
    \par
    \begin{minipage}[t][][c]{0.48\textwidth}
        \caption{CSG objective function value approximation during the optimization process for the \textit{plate (100\%)} initial design. The dashed line shows the inital objective function value, whereas the different graphs correspond to the choices (a), (b) and (c) for $\Vert\cdot\Vert_{_\text{Out}}$.}
        \label{fig:DDAPlateFull3}
    \end{minipage}\hfill
    \begin{minipage}[t][][c]{0.48\textwidth}
        \caption{Results for the \textit{plate (100\%)} initial design presented in \Cref{fig:DDAPlateFull3}, augmented by the CSG objective function value approximation in the case that $\Vert\cdot\Vert_{_\text{Out}}$ was chosen according to (d).}
        \label{fig:DDAPlateFull4}
    \end{minipage}
\end{figure}
\FloatBarrier
\section{Online Error Estimation}\label{sec:ErrorEstimates}
Before we go into theoretical details, we first collect a few key properties and results concerning CSG, which were shown in \cite{CSGPart1}. In a first simple setting, we consider optimization problems of the form 
\begin{equation}\label{eq:SimpleSetting}
    \begin{aligned}
    \min \quad & J(u) \\
    \text{s.t.}\quad & u\in\U\subset\R^\ddes\text{ for some }\ddes\in\N.
    \end{aligned}
\end{equation}
Additionally, we assume that $\U$ is compact, and for some $\dran\in\N$, there exists an open an bounded set $\X\subset\R^\dran$ and a measure $\mu$ with $\supp(\mu)\subset\X$, such that $J$ can be written as $J(u) = \int_\X j(u,x)\mu(\mathrm{d}x)$. The detailed set of assumptions is given in \cite[Section 2]{CSGPart1}. For now, it is only important that $\nabla_1 j:\U\times\X\to\R$ is bounded and Lipschitz continuous with Lipschitz constant $L_j$. 

During the optimization process, CSG computes design dependent integration weights $\big(\alpha_k\big)_{k=1\ldots,n}$ (cf. \cite[Section 3]{CSGPart1}) to build an approximation $\hat{G}_n$ to the true objective function gradient, based on the available samples from previous iterations $\big(\nabla_1 j(u_k,x_k)\big)_{k=1,\ldots,n}$. To be precise, we have
\begin{equation*}
    \nabla J(u) = \int_\X \nabla_1 j(u,x) \mu(\mathrm{d}x) \approx \sum_{k=1}^n \alpha_k \nabla_1 j(u_k,x_k) =: \hat{G}_n.
\end{equation*}
It was shown in \cite[Lemma 4.7]{CSGPart1}, that 
\begin{equation*}
    \Vert \nabla J(u_n)-\hat{G}_n\Vert\to 0 \quad\text{for }n\to\infty \text{ almost surely}. 
\end{equation*}
Carefully investigating the methods to obtain the integration weights, we observe that
\begin{align*}
    \left\Vert \nabla J(u_n)-\hat{G}_n\right\Vert &= \left\Vert \int_\X \nabla_1 j(u_n,x)\mu(\mathrm{d}x) - \hat{G}_n\right\Vert \\
    &= \left\Vert \sum_{i=1}^n \int_{M_i} \nabla_1 j(u_n,x)\mu(\mathrm{d}x) - \sum_{i=1}^n \nabla_1 j(u_i,x_i)\nu_n(M_i)\right\Vert,
\end{align*}
where $\nu_n$ denotes the measure associated to one of the measures listed in \cite[Section 3.6]{CSGPart1}, depending on the choice of integration weights, and
\begin{align*}
    M_k := \big\{ x\in\X \, : \, \Vert u_n& - u_k \Vert_\sU + \Vert x - x_k\Vert_\sX  \\ &< \Vert u_n - u_j \Vert_\sU + \Vert x - x_j\Vert_\sX \text{ for all } j\in\{1,\ldots,n\}\setminus\{k\}\big\}.
\end{align*}
By construction, $M_k$ contains all points $x\in\X$, such that $(u_n,x)$ is closer to $(u_k,x_k)$ than to any other previous point we evaluated $\nabla_1 j$ at.
For exact integration weights, we have $\nu_n=\mu$ and thus
\begin{align*}
    \left\Vert \nabla J(u_n)-\hat{G}_n\right\Vert &= \left\Vert \sum_{i=1}^n \int_{M_i} \nabla_1 j(u_n,x)\mu(\mathrm{d}x) - \sum_{i=1}^n \int_{M_i} \nabla_1 j(u_i,x_i)\mu(\mathrm{d}x)\right\Vert\\
    &\le \sum_{i=1}^n  \int_{M_i} \left\Vert \nabla_1 j(u_n,x)-\nabla_1 j(u_i,x_i)\right\Vert\mu(\mathrm{d}x)\\
    &\le \sum_{i=1}^n \int_{M_i} L_j \cdot\left( \sup_{x\in M_i} Z_n(x) \right)\mu(\mathrm{d}x) \\
    &= L_j\sum_{i=1}^n \mu(M_i)\sup_{x\in M_i} Z_n(x)\\
    &\le L_j\sup_{x\in\X} Z_n(x).
\end{align*}
Here, $Z_n$ is given by
\begin{equation*}
    Z_n(x) := \min_{k\in\{1,\ldots,n\}}\big(\Vert u_n-u_k\Vert_\sU + \Vert x-x_k\Vert_\sX\big) .
\end{equation*}
In other words, the approximation error can be bounded in terms of the Lipschitz constant of $\nabla_1 j$ and the quantity $Z_n$, which relates to the size of Voronoi cells \cite{voronoi} with positive integration weights.

Both $L_j$ and $\sup_{x\in\X} Z_n(x)$ can be efficiently approximated during the optimization process, e.g. by finite differences of the samples $\big(\nabla_1 j(u_i,x_i)\big)_{i=1,\ldots,n}$ and by 
\begin{equation*}
    \sup_{x\in\X} Z_n(x) \approx \max_{k=1,\ldots,n} Z_n(x_k),
\end{equation*}
yielding an online error estimation. Such an approximation may, for example, be used in stopping criteria.
\section{Convergence Rates}\label{sec:ConvergenceRates}
Throughout this section, we assume \cite[Assumptions 2.2 - 2.8]{CSGPart1} to be satisfied.
\subsection{Theoretical Background}\label{sec:ConvergenceRatesTheory}
In the convergence analysis presented in \cite{CSGPart1}, we have already seen that the fashion in which the gradient approximation $\hat{G}_n$ is calculated in CSG is crucial for $\Vert \hat{G}_n-\nabla J(u_n)\Vert \to 0$ and that this property of CSG in turn is the key to all advantages CSG offers in comparison to classic stochastic optimization methods, like convergence for constant steps, backtracking, more involved optimization problems, etc.

The price we pay for this feature lies within the dependency of $\hat{G}_n$ on the past iterates. For comparison, the search direction $\hat{G}_n^{\text{SG}}$ in a stochastic gradient descent method is given by
\begin{equation*}
    \hat{G}_n^{\text{SG}} = \nabla_1 j(u_n,x_n).
\end{equation*}
Thus, it is independent of all previous steps and fulfills
\begin{equation*}
    \mathbb{E}_\X\left[\hat{G}_n^{\text{SG}}\right] = \mathbb{E}_\X\big[ \nabla_1 j(u_n,\cdot)\big] = \nabla J(u_n),
\end{equation*}
i.e., it is an unbiased sample of the full gradient. The combination of these properties allows for a straight-forward convergence rate analysis, see, e.g., \cite{Steps01}.

In contrast, $\hat{G}_n$ is in general \textit{not} an unbiased approximation to $\nabla J(u_n)$ and moreover \textit{not} independent of $\big(u_i,x_i)_{i=1,\ldots,n-1}$. The main problem in finding the convergence rate of $\Vert u_{n+1}-u_n\Vert\to0$ is, that this quantity depends on the approximation error $\Vert\hat{G}_n-\nabla J(u_n)\Vert$, which, as we have seen in \Cref{sec:ErrorEstimates}, depends on $Z_n$. Since $Z_n$ itself is deeply connected to $\min_k\Vert u_{n} - u_k\Vert$, we run into a circular argument.

Therefore, up to now, we are not able to proof convergence rates for the CSG iterates. We can, however, state a prediction to this rate and provide numerical evidence. 

\begin{claimnew}\label{rem:ProposedRates}
We claim that the CSG method, applied to problem \eqref{eq:SimpleSetting}, using a constant step size $\tau < \tfrac{2}{L}$ and empirical integration weights, fulfills
\begin{equation*}
    \Vert u_{n+1} - u_n\Vert = \mathcal{O}\left( \ln(n)\cdot n^{-\tfrac{1}{\max\{2,\dran\}}}\right).
\end{equation*}
\end{claimnew}
To motivate this claim, note that, in the proof of \cite[Lemma 4.7]{CSGPart1}, it was shown that there exists $C>0$ such that
\begin{equation*}
    \left\Vert\hat{G}_n-\nabla J(u_n)\right\Vert \le C \left(\int_\X Z_n(x)\mu(\mathrm{d}x) + d_W(\mu_n,\mu)\right),
\end{equation*}
where $d_W$ denotes the Wasserstein distance of the two measures $\mu_n$ and $\mu$. By \cite[Theorem 1]{WassersteinRate}, the empirical measure $\mu_n$ satisfies
\begin{equation*}
    \mathbb{E}\big[d_W(\mu_n,\mu)\big] \le C(\dran)\cdot \left(\int_\X \Vert x\Vert_\sX^3\mu(\mathrm{d}x)\right)^{\tfrac{1}{3}}\cdot \begin{cases} \tfrac{1}{\sqrt{n}} & \text{if }\dran = 1, \\  \tfrac{\ln(1+n)}{\sqrt{n}} & \text{if } \dran = 2, \\ n^{-\tfrac{1}{\dran}} & \text{if }\dran\ge 3.\end{cases}
\end{equation*}
This result is the main motivation for \Cref{rem:ProposedRates}. It can be shown that the rate $n^{-1/\dran}$ for $\dran\ge3$ is sharp if $\mu$ corresponds to a uniform distribution on $\X$. Thus, in this case, it is reasonable to assume a uniform distribution also corresponds to the worst-case rate of $\int_\X Z_n(s)\mu(\mathrm{d}x)\to 0$. Assuming that the difference in designs appearing in $Z_n$ is negligible due to the overall convergence of CSG, we obtain the rate
\begin{equation*}
    \sup_{x\in\X}\; Z_n(x) = \mathcal{O}\left( \ln(n)\cdot n^{-\tfrac{1}{\max\{2,\dran\}}}\right).
\end{equation*}
To see this, we fill $\X\subset\R^{\dran}$ with balls (w.r.t. the norm $\Vert\cdot\Vert_\sX$) of radius $\e >0$ and denote by $N(\e)\in\N$ the number of cells. Due to the dimension of $\X$, we have $\mathcal{O}\big(N(\e)\big)=\e^{-\dran}$. Now, to achieve $\sup_{x\in\X} Z_n(x) < \e$, we need each of these cells to contain at least one of the sample points $(x_i)_{i=1,\ldots,n}$. It is well-known that the expected number of samples we need to draw for this to happen is given by
\begin{equation*}
    N(\e)\sum_{k=1}^{N(\e)}\frac{1}{k} = \mathcal{O}\left(-\e^{-\dran}\ln(\e)\right),
\end{equation*}
where we used
\begin{equation*}
    \sum_{k=1}^n\frac{1}{k} = \mathcal{O}\big(\ln(n)\big) \quad\text{for }n\to\infty.
\end{equation*}
In other words, the convergence rates of $\int_\X Z_n(x)\mu(\mathrm{d}x)\to0$ and $d_W(\mu_n,\mu)\to 0$ are comparable.

Now that we motivated the rates claimed in \Cref{rem:ProposedRates} for the approximation error $\Vert \hat{G}_n - \nabla J(u_n)\Vert$, we use the following proposition to show that the rates of $\Vert u_{n+1}-u_n\Vert\to0$ can not be worse.
\begin{proposition}\label{prop:ProposedWeightsU}
Assume that the approximation error $\Vert \hat{G}_n-\nabla J(u_n)\Vert$ satisfies
\begin{equation*}
    \Vert \hat{G}_n - \nabla J(u_n)\Vert = \mathcal{O}\left( \ln(n)\cdot n^{-\tfrac{1}{\max\{2,\dran\}}}\right).
\end{equation*}
Then, under the assumptions of \Cref{rem:ProposedRates}, it holds 
\begin{equation*}
    \Vert u_{n+1}-u_n\Vert = \mathcal{O}\left( \ln(n)\cdot n^{-\tfrac{1}{\max\{2,\dran\}}}\right).
\end{equation*}
\end{proposition}
\begin{proof}
Assume for contradiction that this is not the case. Thus, there exists $N\in\N$ such that
\begin{equation}\label{eq:RatesAssum}
    \left\Vert\nabla J(u_n)-\hat{G}_n\right\Vert\le\tfrac{1}{2}\left(\tfrac{1}{\tau}-\tfrac{L}{2}\right)\Vert u_{n+1}-u_n\Vert_\sU \quad\text{for all }n\ge N.
\end{equation}
By the descent lemma \cite[Lemma 5.7]{DescentUndProjection}, the characteristic property of the projection operator \cite[Theorem 6.41]{DescentUndProjection} and the Cauchy-Schwarz inequality, we obtain
\begin{align*}
    J(u_{n+1})&-J(u_n) \\
        &\le \nabla J(u_n)^\top(u_{n+1}-u_n) + \tfrac{L}{2}\Vert u_{n+1}-u_n\Vert_\sU^2 \\
        &= \hat{G}_n^\top(u_{n+1}-u_n) + \tfrac{L}{2}\Vert u_{n+1}-u_n\Vert_\sU^2 + \left(\nabla J(u_n)-\hat{G}_n\right)^\top(u_{n+1}-u_n) \\
        &\le \left(\tfrac{L}{2}-\tfrac{1}{\tau}\right)\Vert u_{n+1}-u_n\Vert_\sU^2 + \left\Vert \nabla J(u_n)-\hat{G}_n\right\Vert\cdot\Vert u_{n+1}-u_n\Vert_\sU \\
        &= \left(\left(\tfrac{L}{2}-\tfrac{1}{\tau}\right)\Vert u_{n+1}-u_n\Vert_\sU + \left\Vert\nabla J(u_n)-\hat{G}_n\right\Vert\right)\Vert u_{n+1}-u_n\Vert_\sU.
\end{align*}
Combining this with \eqref{eq:RatesAssum} gives $J(u_{n+1})\le J(u_n)$ for all $n\ge N$, since $\tfrac{L}{2}<\tfrac{1}{\tau}$. Thus, the sequence of objective function values $\big(J(u_n)\big)_{n\in\N}$ is monotonically decreasing for all $n\ge N$. By continuity of $J$ and compactness of $\U$, $J$ is bounded and $J(u_n)\to\bar{J}$ for some $\bar{J}\in\R$. Therefore,
\begin{equation*}
    -\infty < \bar{J} - J(u_N) = \sum_{n=N}^\infty \big( J(u_{n+1}-J(u_n)\big) \le \tfrac{1}{2}\left(\tfrac{L}{2}-\tfrac{1}{\tau}\right)\sum_{n=N}^\infty \Vert u_{n+1}-u_n\Vert_\sU^2.
\end{equation*}
Hence, the series
\begin{equation*}
    \sum_{n=N}^\infty \Vert u_{n+1}-u_n\Vert_\sU^2
\end{equation*}
converges, contradicting $\Vert u_{n+1}-u_n\Vert\neq\mathcal{O}\left( \ln(n)\cdot n^{-\tfrac{1}{\max\{2,\dran\}}}\right)$.
\end{proof}
\subsection{Numerical Verification}
We want to verify the proclaimed rates numerically. For this purpose, we consider two optimization problems that can easily be scaled to high dimensions. The first problem is given by
\begin{equation}
    \min_{u\in\U}\quad \frac{1}{2}\int_\X \big\Vert u-x\big\Vert_2^2 \mathrm{d}x\label{eq:NumVerFirst},
\end{equation}
where $\X = \left( -\tfrac{1}{2},\tfrac{1}{2}\right)^{\dran}$ and $\U = [-5,5]^{\dran}$, i.e., $\U$ and $\X$ have the same dimension. The second problem, 
\begin{equation}
    \min_{u\in\U}\quad\frac{1}{2}\int_{-0.5}^{0.5}\big\Vert u - x\cdot\mathds{1}_{\ddes}\big\Vert_2^2\mathrm{d}x, \label{eq:NumVerSec}
\end{equation}
fixes $\dran = 1$, while $\U=[-5,5]^{\ddes}$. Here, $\mathds{1}_{\ddes}$ represents the vector $(1,1,\ldots,1)^\top\in\R^{\ddes}$. Note that, in both settings, we have $L_j = 1$. Thus, by \Cref{sec:ErrorEstimates}, we have
\begin{equation*}
    \big\Vert \hat{G}_n - \nabla J(u_n)\big\Vert \le \sup_{x\in\X}\; Z_n(x) \approx \max_{k=1,\ldots,n} Z_n(x_k).
\end{equation*}
The optimal solution to \eqref{eq:NumVerFirst} and \eqref{eq:NumVerSec} is given by the zero vector $u^\ast = 0\in\U$.

In our analysis, for different values of the dimensions $\dran,\ddes\in\N$, problems \eqref{eq:NumVerFirst} and \eqref{eq:NumVerSec} were initialized with 500 random starting points. The constant step size of CSG was chosen as $\tau = \tfrac{1}{2}$. We track $\Vert u_n - u^\ast\Vert$ and $\max_{k=1,\ldots,n} Z_n(x_k)$ during the optimization process and compare the median of the 500 runs to the rates predicted in \Cref{rem:ProposedRates}. The results can be seen in \Cref{fig:RatesZequidim,fig:RatesUequidim,fig:RatesZdim1,fig:RatesUdim1}. Note that, for the plots of the predicted rates, we omitted the factor $\ln(n)$. Therefore, the corresponding graphs are straight lines, where the slope $-\tfrac{1}{\max\{2,\dran\}}$ is equal to the asymptotic slope of the predicted rate, since
\begin{equation*}
    \ln(n)\cdot n^{-\tfrac{1}{\max\{2,\dran\}}} = \mathcal{O}\left( n^{-\tfrac{1}{\max\{2,\dran\}}+\e}\right)\quad\text{for all }\e>0.
\end{equation*}

In the equidimensional, i.e., $\dim(\X)=\dim(\U)$, setting \eqref{eq:NumVerFirst}, the experimentally obtained values for $Z_n$ almost perfectly match the claimed rates. For $\Vert u_n-u^\ast\Vert$, the observed rates also match the predictions for very small and large dimensions. For $\dran=3,4,5$, the convergence obtained in the experiments was even slightly faster than predicted. Investigating the results for \eqref{eq:NumVerSec}, it is clearly visible that increasing the design dimension $\ddes$, while keeping the parameter dimension $\dran$ fixed, has no influence on the obtained rates of convergence, indicating that CSG is able to efficiently handle large-scale optimization problems.  
\begin{figure}
    \centering
    \begin{minipage}[t][][c]{0.48\textwidth}
        \centering
        \includegraphics[width = \textwidth]{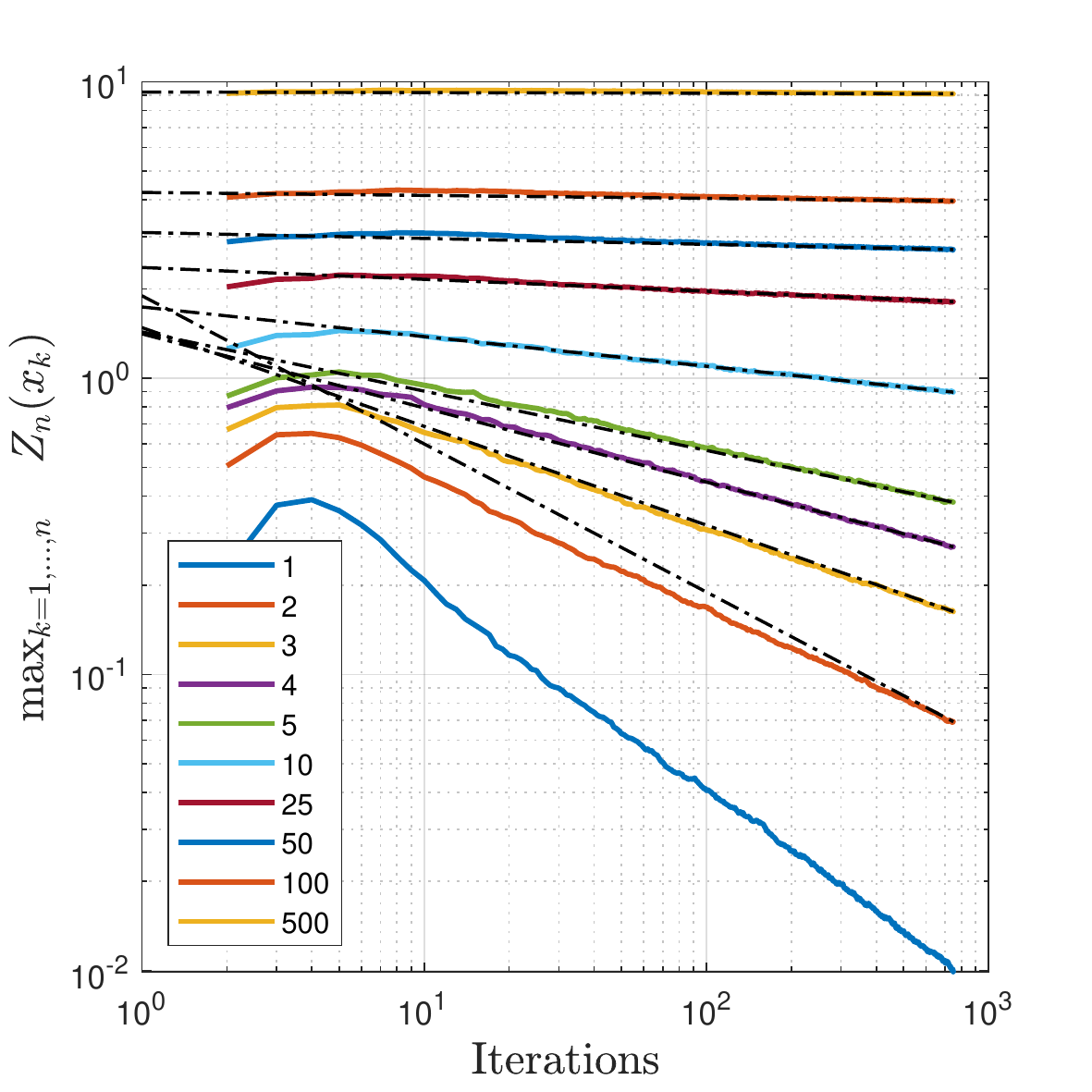}%
        \caption{The bold lines represent the median values of $\max_{k=1,\ldots,n}Z_n(x_k)$ for the equidistant problem \eqref{eq:NumVerFirst} with respect to the iteration counter. The different colors indicate the different dimensions $\dran \in \{1,2,\ldots,500\}$. The dotted lines correspond to the respective predicted rates $n^{-\tfrac{1}{\max\{2,\dran\}}}$. Since the predictions for $\dran=1$ and $\dran=2$ are equal, only the case $\dran=2$ is shown.}
        \label{fig:RatesZequidim}
    \end{minipage}\hfill
    \begin{minipage}[t][][c]{0.48\textwidth}
        \centering
        \includegraphics[width = \textwidth]{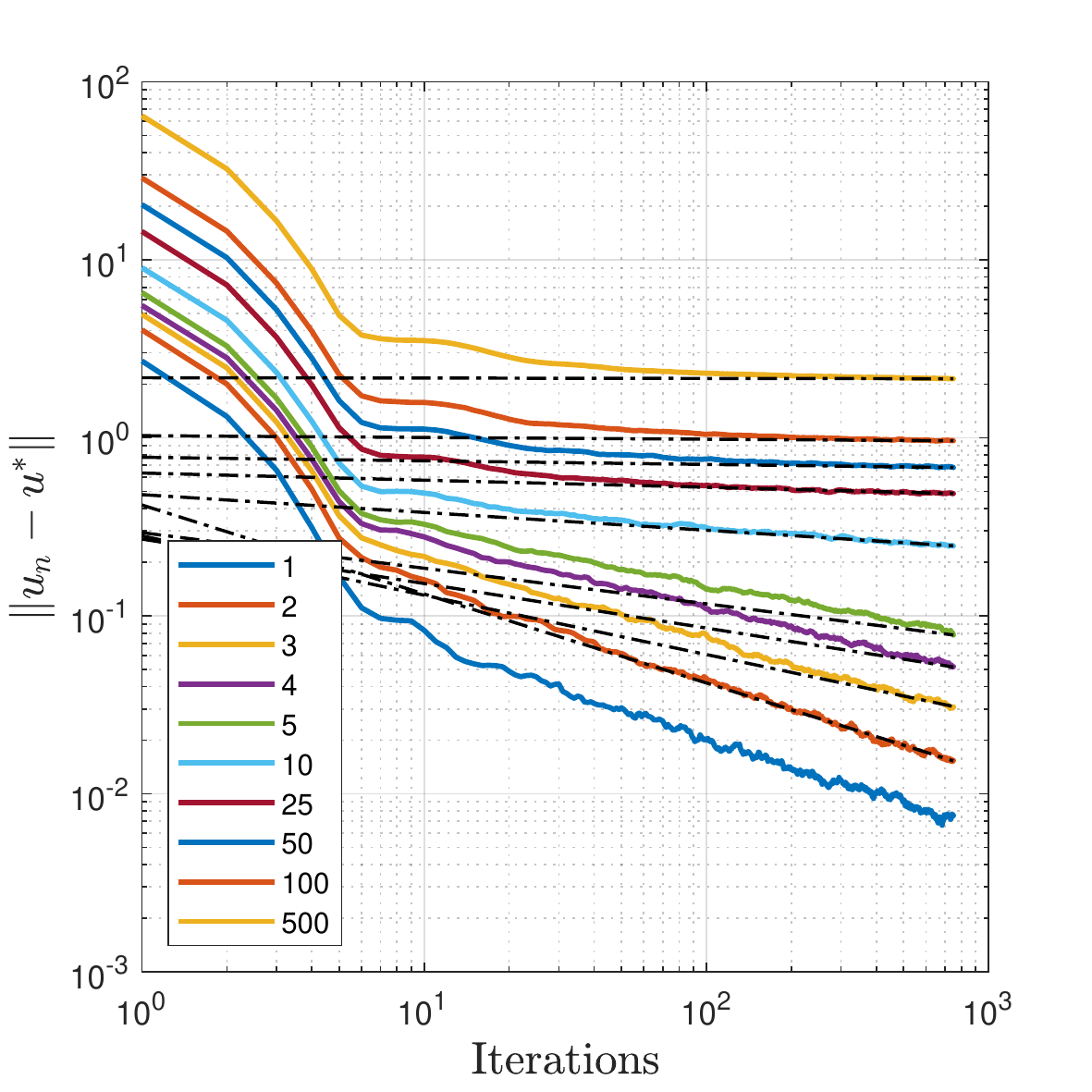}%
        \caption{Median values of $\Vert u_n-u^\ast\Vert$ in the equidimensional setting \eqref{eq:NumVerFirst} for different choices of $\dran\in\{1,2,\ldots,500\}$. For each dimension, the predicted worst-case asymptotic line $n^{-\tfrac{1}{\max\{2,\dran\}}}$ is indicated by the dotted line. Again, we omit the prediction for $\dran=1$, since it has the same slope as in the case for $\dran=1$.}
        \label{fig:RatesUequidim}
    \end{minipage}
    \begin{minipage}[t][][c]{0.48\textwidth}
        \centering
        \includegraphics[width = \textwidth]{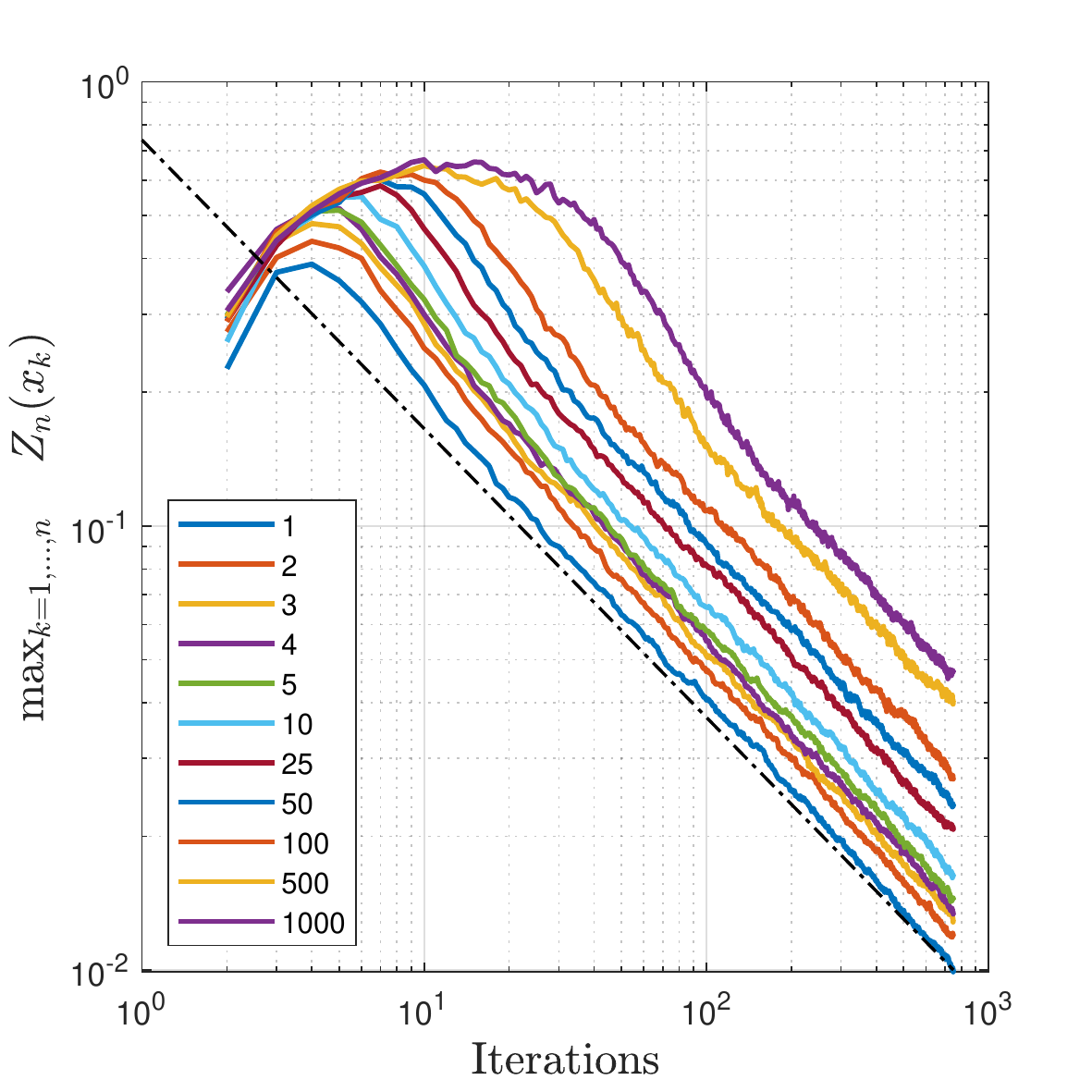}%
        \caption{Results for the median of $\max_{k=1,\ldots,n}Z_n(x_k)$ in setting \eqref{eq:NumVerSec} for different dimensions $\ddes\in\{1,2,\ldots,1000\}$, indicated by different colors. As we conjectured, the asymptotic slope of all curves is equal, since $\dran=1$ is fixed. As a point of reference, we added the graph of $n^{-0.65}$, represented by the dotted line.}
        \label{fig:RatesZdim1}
    \end{minipage}\hfill
    \begin{minipage}[t][][c]{0.48\textwidth}
        \centering
        \includegraphics[width = \textwidth]{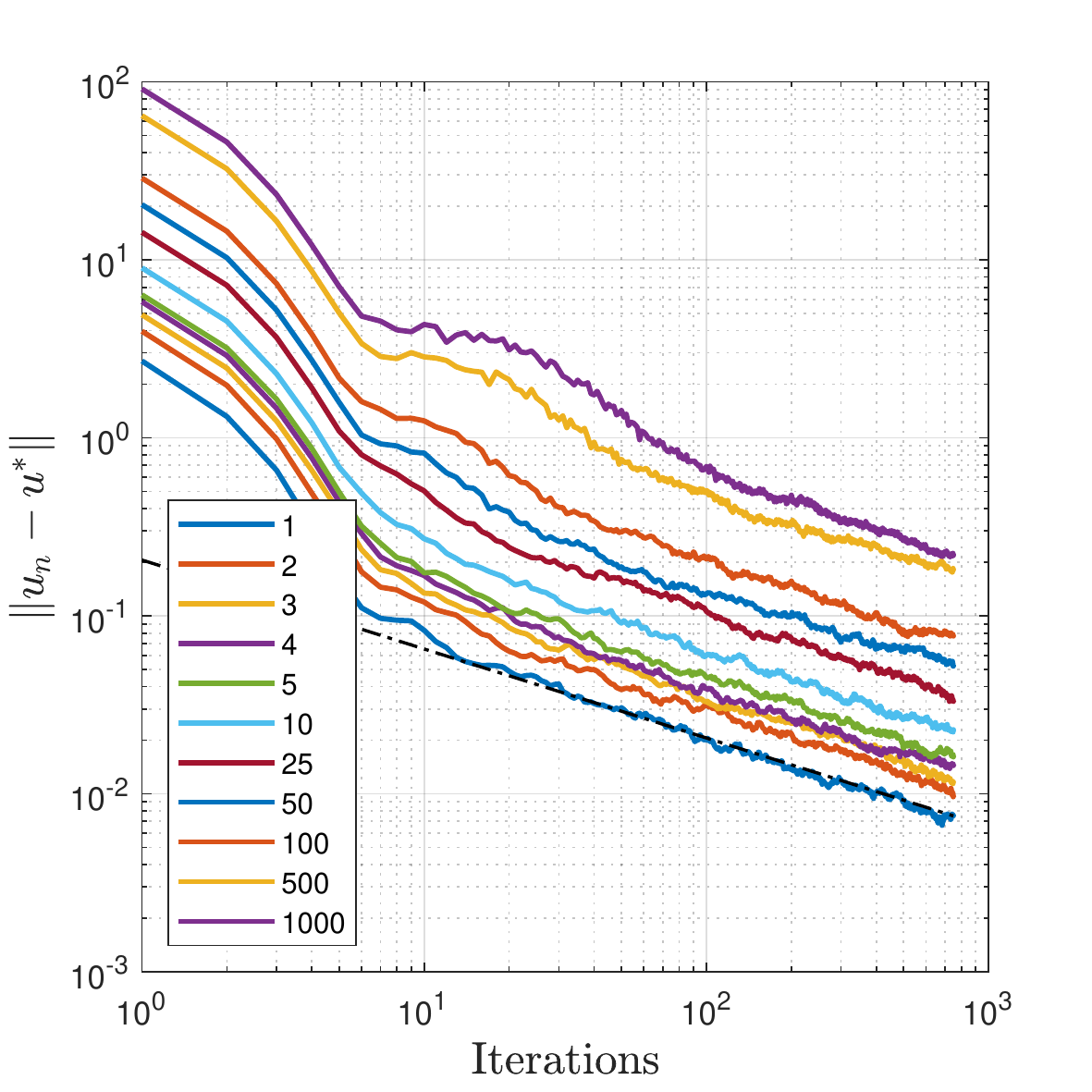}%
        \caption{Median distance to the optimal solution $u^\ast$ during the course of the iterations for $\ddes\in\{1,2,\ldots,1000\}$. Again, the asymptotic slope of all curves is equal and we added the line corresponding to $n^{-0.65}$ for comparison.}
        \label{fig:RatesUdim1}
    \end{minipage}
\end{figure}
\subsection{Circumventing Slow Convergence}\label{sec:CircumventSlow}
As we have seen so far, the convergence rate of the CSG method worsens with increasing dimension of integration $\dran\in\N$. However, it is possible to circumvent this behavior, if the problem admits additional structure. Assume that there exist suitable $\X_1,\X_2,\mu_1,\mu_2,f_1$ and $f_2$ such that the objective function appearing in \eqref{eq:SimpleSetting} can be rewritten as
\begin{equation*}
    J(u) = \int_\X j(u,x)\mu(\mathrm{d}x) = \int_{\X_1} f_1 \left(u,x,\int_{\X_2} f_2(u,y)\mu_2(\mathrm{d}y)\right)\mu_1(\mathrm{d}x).
\end{equation*}
Assume further, that $\X_1,\X_2,\mu_1,\mu_2,f_1$ and $f_2$ satisfy the corresponding equivalents of \cite[Assumptions 2.2 - 2.8]{CSGPart1}. 

Now, we can independently calculate integration weights $(\alpha_k)_{k=1,\ldots,n}$ and $(\beta_k)_{k=1,\ldots,n}$ for the integrals over $\X_1$ and $\X_2$, respectively. The corresponding CSG approximations (indicated by hats) are then given by
\begin{align*}
    f^{(n)} &:= \int_{\X_2} f_2(u,y)\mu_2(\mathrm{d}y) \approx \sum_{i=1}^n \alpha_i f_2(u_i,y_i) =: \hat{f}_n, \\
    g^{(n)} &:= \int_{\X_2} \nabla_1 f_2(u,y)\mu_2(\mathrm{d}y)  \approx \sum_{i=1}^n \alpha_i\nabla_1 f_2(u_i,y_i) =: \hat{g}_n, \\
    \nabla J(u_n)& \approx \sum_{i=1}^n\beta_i\Big( \nabla_1 f_1 (u_i,x_i,\hat{f}_i) + \partial_3 f_1(u_i,x_i,\hat{f}_i)\cdot\hat{g}_i\Big)=:\hat{G}_n.
\end{align*}
The same steps as performed in the proof of \cite[Lemma 4.7]{CSGPart1} yield the existence of a constant $C_1>0$, depending only on the Lipschitz constants of $\nabla f_1$ and $\nabla f_2$, such that
\begin{align}
    &\Big\Vert \nabla J(u_n) - \hat{G}_n \Big\Vert \nonumber\\
    &\le C_1\! \Big( d_W(\mu_1,\nu^{\beta}_n)+\sup_{x\in\X_1}\min_{k=1,\ldots,n}\!\!\big( \Vert u_n - u_k\Vert_\sU\!\!\! + \Vert x - x_k\Vert_{_{\X_1}}\!\!\! + \vert \hat{f}_n - \hat{f}_k\vert\big) \Big). \label{eq:SaveRatesError}
\end{align}
Here, $\nu^\beta_n$ corresponds to the measure related to the integration weights $(\beta_k)_{k=1,\ldots,n}$, see \cite[Assumption 2.8]{CSGPart1}.
Now, denoting by $C_2>0$ a constant depending on the Lipschitz constant $L_{f_2}$ of $f_2$, we decompose the last term: 
\begin{align}
    \vert\hat{f}_n - &\hat{f}_k\vert \nonumber\\
    &\le \vert \hat{f}_n - f_n\vert + \vert \hat{f}_k - f_k\vert + \vert f_n-f_k\vert \nonumber\\
    &\le\vert \hat{f}_n - f_n\vert + \vert \hat{f}_k - f_k\vert+
    L_{f_2} \Vert u_n-u_k\Vert_\sU \nonumber\\
    &\le C_2\Big(\Vert u_n-u_k\Vert_\sU + \sup_{y\in\X_2}\min_{i=1,\ldots,n} \big( \Vert u_n - u_i\Vert_\sU + \Vert y - y_i\Vert_{_{\X_2}}\big) \nonumber\\
    &\qquad+ \sup_{y\in\X_2}\min_{i=1,\ldots,k} \big( \Vert u_k - u_i\Vert_\sU + \Vert y - y_i\Vert_{_{\X_2}}\big) +d_W(\mu_2,\nu^{\alpha}_n) + d_W(\mu_2,\nu^{\alpha}_k)\Big) \nonumber\\
    &= C_2\Big(\Vert u_n-u_k\Vert_\sU + \sup_{y\in\X_2} Z_n(y) + \sup_{y\in\X_2} Z_k(y) +d_W(\mu_2,\nu^{\alpha}_n) + d_W(\mu_2,\nu^{\alpha}_k)\Big).\label{eq:SaveRatesError2}
\end{align}
Assuming that the convergence of the sequence $(u_n)_{n\in\N}$ generated by the CSG method implies
\begin{equation*}
    \mathcal{O}\left(\sup_{y\in\X_2} Z_n (y)\right) =  \mathcal{O}\left(\sup_{y\in\X_2} Z_k (y)\right) \quad\text{and}\quad \mathcal{O}\big( d_W(\mu_2,\nu^{\alpha}_n)\big) = \mathcal{O}\big( d_W(\mu_2,\nu^{\alpha}_k)\big),
\end{equation*}
we insert \eqref{eq:SaveRatesError2} into \eqref{eq:SaveRatesError}, to obtain
\begin{equation*}
    \big\Vert \nabla J(u_n)-\hat{G}_n\Vert \le C(C_1,C_2)\Big( d_W(\mu_1,\nu^{\beta}_n) + d_W(\mu_2,\nu^{\alpha}_n) + \sup_{x\in\X_1} Z_n(x) + \sup_{y\in\X_2} Z_n(y)\Big). 
\end{equation*}
Therefore, by the same arguments as in \Cref{sec:ConvergenceRatesTheory}, we claim
\begin{align*}
    \big\Vert \nabla J(u_n)-\hat{G}_n\big\Vert &= \mathcal{O}\left( \ln(n)\cdot n^{-\tfrac{1}{\max\{2,\dim(\X_1),\dim(\X_2)\}}}\right), \\
    \Vert u_{n+1}-u_n\Vert &= \mathcal{O}\left( \ln(n)\cdot n^{-\tfrac{1}{\max\{2,\dim(\X_1),\dim(\X_2)\}}}\right).
\end{align*}
In conclusion, we claim that, assuming the objective function can be rewritten in terms of nested expectation values
\begin{equation*}
    J(u) = \int_{\X_1} f_1\left(u,x_1,\int_{\X_2}f_2\left( u,x_2,\int_{\X_3}f_3(\cdots)\mu_3(\mathrm{d}x_3) \right)\mu_2(\mathrm{d}x_2)\right)\mu_1(\mathrm{d}x_1),
\end{equation*}
the convergence rate of the CSG method depends only on the \textit{largest} dimension of the occurring $\X_i$, which may be much lower when compared to $\dim(\X)$.

Since this is again a claim and not a rigorous proof, we validate this assumption numerically. For this, we once more consider \eqref{eq:NumVerFirst} and initialize it with 500 random starting points. This time, however, we utilize the fact that the objective function can be written as
\begin{equation*}
    J(u) = \frac{1}{2}\int_{\X} \Vert u-x\Vert_2^2\mathrm{d}x = \frac{1}{2}\int_{\X} \Big( \sum_{i=1}^{\dran}(u_i-x_i)^2\Big) \mathrm{d}x = \frac{1}{2}\sum_{i=1}^{\dran} \int_{-\tfrac{1}{2}}^{\tfrac{1}{2}}(u_i-x_i)^2\mathrm{d}x_i.
\end{equation*}
Thus, we can group the independent coordinates into subintegrals of arbitrary dimension, allowing us to study our claim for a large number of different regroupings without having to change the whole problem formulation. The results for several different decompositions and 500 random starting points in the case $\dran=100$ are shown in \Cref{fig:rates_split}. The improved rates of convergence are clearly visible, independent on whether the subgroup dimensions are equal or not. As claimed above, the highest remaining dimension of integration determines the overall convergence rate of CSG.
\begin{figure}
  \begin{minipage}[c]{0.5\textwidth}
    \includegraphics[width=\textwidth]{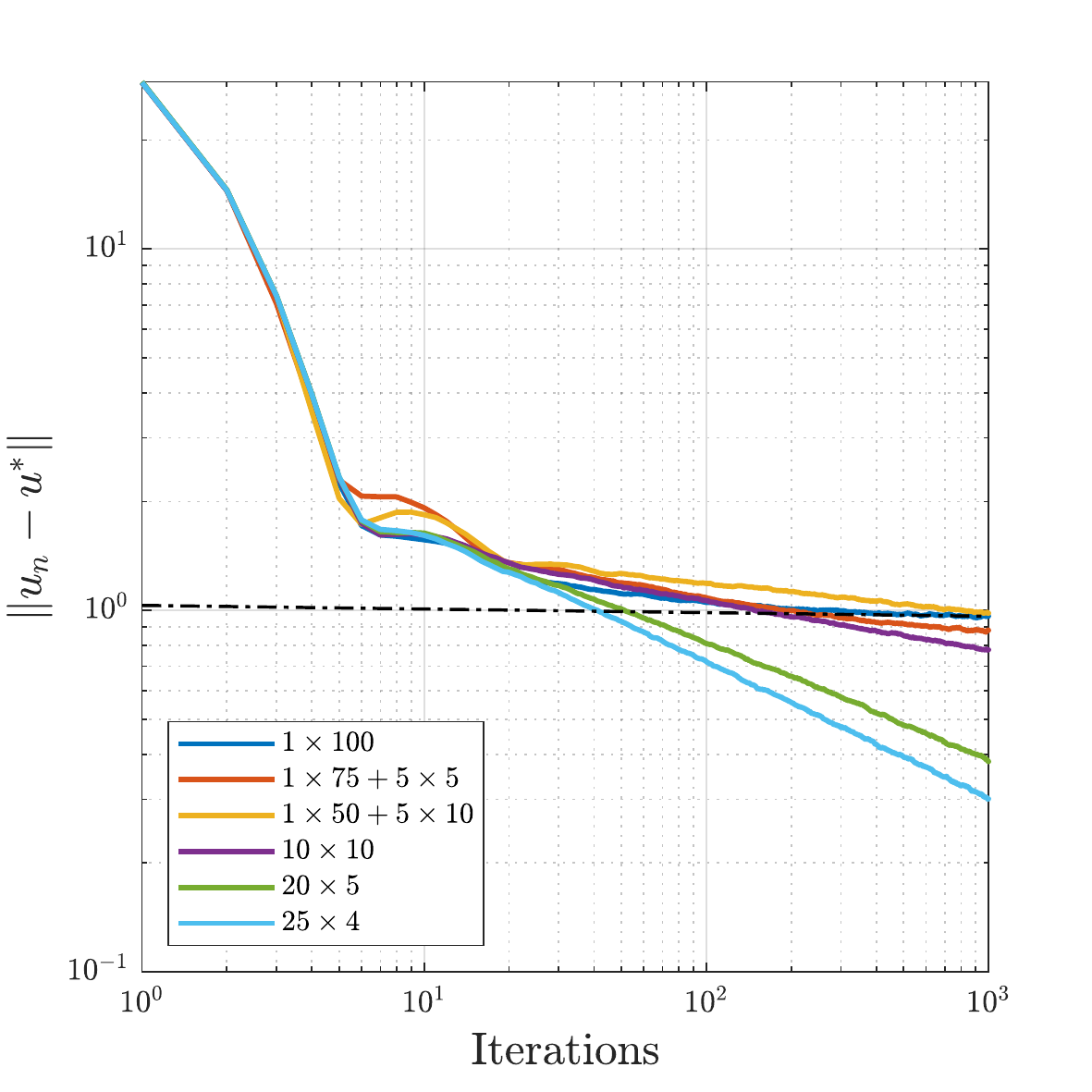}
  \end{minipage}\hfill
  \begin{minipage}[c]{0.45\textwidth}
    \caption{Median total error $\Vert u_n-u^\ast\Vert$ of the CSG iterates for \eqref{eq:NumVerFirst}, for $\dran=100$. The integral over $\X=\left(-\tfrac{1}{1},\tfrac{1}{2}\right)^{\dran}$ has been decomposed into several integrals of smaller dimension. The labels in the bottom left give details about the decomposition, e.g., the orange line corresponds to splitting the whole integral into one integral of dimension 75 and 5 integrals of dimension 5. The dotted line indicates the expected rate of convergence obtained by the CSG method without splitting up the integral.}
    \label{fig:rates_split}
  \end{minipage}
\end{figure}
\FloatBarrier
\section{Conclusion and Outlook}
In this contribution, we presented a numerical analysis of the CSG method. The practical performance of CSG was tested for two applications from nanoparticle design optimization with varying computational complexity. For the low-dimensional problem formulation, CSG was shown to perform superior when compared to the commercial \textit{fmincon} blackbox solver. 
The high-dimensional setting provided an example, for which classic optimization schemes (stochastic as well as deterministic) from literature do not provide optimal solutions within reasonable time.

Convergence rates for CSG with constant step size were proposed and analytically motivated. They were shown to agree with numerically obtained convergence rates in several different instances. Moreover, in the case that the objective function admits additional structure, techniques to circumvent slow convergence for high dimensional integration domains were presented.

While the proposed convergence rates for CSG agree with our experimental results, it remains an open question if they can be proven rigorously. Furthermore, even though the choice of a metric for the nearest neighbor approximation in the integration weights is irrelevant for the convergence results, a problem specific metric could significantly improve the performance of CSG by exploiting additional structure, which might be lost by utilizing an arbitrary metric. How to automatically obtain such a metric during the optimization process requires further research.

\bmhead{Data Availability Statement}
The simulation datasets generated during the current study are available from the corresponding author on reasonable request.
\bmhead{Conflict of Interests}
The authors have no relevant financial or non-financial interests to disclose.
\bibliography{Literature}


\end{document}